\documentclass{article}
\usepackage[english]{babel}
\usepackage[letterpaper,top=2cm,bottom=2cm,left=3cm,right=3cm,marginparwidth=1.75cm]{geometry}

\usepackage{amsmath}
\usepackage{amsthm}
\usepackage{graphicx}
\usepackage[colorlinks=true, allcolors=blue]{hyperref}
\usepackage{comment}
\usepackage{amsfonts}
\usepackage{amssymb}
\usepackage{graphicx}
\usepackage{epstopdf}
\usepackage{enumerate}
\usepackage{enumitem}
\usepackage{mathrsfs}
\usepackage{mathtools}
\usepackage{xcolor}
\usepackage{subcaption}
\usepackage{algorithm} 
\usepackage{algpseudocode} 
\usepackage{bbm}
\allowdisplaybreaks

\usepackage{tikz}
\usetikzlibrary{arrows.meta}
\usepackage[dvipsnames]{xcolor}
\usepackage{authblk}

\begingroup
\newtheorem{theorem}{Theorem}[section]

\newtheorem{lemma}[theorem]{Lemma}

\newtheorem{definition}[theorem]{Definition}
\newtheorem{remark}[theorem]{Remark}
\endgroup

\newcommand{\di}[1]{{\rm div}(#1)}
\newcommand{\diy}[1]{{\rm{div}_y}\big(#1\big)}
\newcommand{\dix}[1]{{\rm{div}_x}\big(#1\big)}
\newcommand{\n}{\mathbf n}

\newcommand{\ve}{\varepsilon}

\newcommand{\dd}{\mathrm{d}}

\newcommand{\eff}{{\rm eff}}

\title{Derivation of a nonlinear, history-dependent tissue-level current law for an electropermeabilization model via homogenization}
\author[1]{Tobias Gebäck}
\author[1]{Ioanna Motschan Ulander}
\author[1]{Irina Pettersson \thanks{Corresponding author, irinap@chalmers.se}}

\affil[1]{Chalmers University of Technology and University of Gothenburg, Sweden}

\date{}
\begin{document}
\maketitle

\tableofcontents

\begin{abstract}
We study a phenomenological electropermeabilization model in a periodic medium  representing biological tissue. Starting from a cell-level model describing the electric potential and the degree of poration, we perform dimensional analysis to identify a relevant scaling in terms of a small parameter $\ve$ - the ratio between the cell and the tissue size. The electric potential satisfies electrostatic equations in the extra- and intracellular domains, while its jump across the cell membrane evolves according to a nonlinear law coupled with an ordinary differential equation for the poration degree. We prove the well-posedness of the microscopic problem, derive a priori estimates, obtain formal asymptotics, and rigorously justify the expansion by combining  two-scale convergence with monotonicity arguments. The resulting macroscopic model exhibits memory effects and a nonlinear, time-dependent effective current. It captures the nontrivial evolution of effective conductivity, including a characteristic drop reflecting the capacitive behavior of the lipid bilayer, in agreement with experimental data. Numerical computations of the effective conductivity confirm that, although the intracellular and extracellular bulk conductivities are constant, tissue conductivity varies nonlinearly with electric field strength, showing a sigmoid trend. 
\end{abstract}


\vspace{2mm}

\section{Introduction}

The goal of this work is to model the reversible electroporation phenomenon at the tissue scale based on the phenomenological cell-level electropermeabilization model presented in \cite{kavian2014classical}, using periodic homogenization. We derive a macroscopic model and show that the effective electric current is nonlinear with respect to the electric field and time-dependent. The main contribution of this work is to provide a rigorous mathematical explanation and numerical simulation of the nonlinear, sigmoid-shaped dependence of tissue conductivity on electric field strength and its time evolution, in qualitative agreement with experimental observations reported in the literature. 

Reversible electropermeabilization (often called electroporation) is a temporary increase in membrane permeability caused by short, high-voltage electric pulses. In reversible electroporation, the pores reseal once the external electric field is removed. This reversible process is widely used to transfer otherwise impermeable molecules—such as DNA plasmids or drugs—into living cells. However, even though electroporation mechanisms are well understood at the cellular level, there are very few dynamic electropermeabilization models on the tissue scale. In particular, computing the time-dependent electric field within biological tissues remains a crucial challenge.

Electroporation models for a single cell focusing on pore creation and evolution have been proposed in \cite{debruin1999modeling}, \cite{neu1999asymptotic}, \cite{krassowska2007modeling}. 
The phenomenological model by Kavian et al. \cite{kavian2014classical} retains the main features of the Neu–Krassowska model but with reduced complexity. An extension of \cite{kavian2014classical} is presented in \cite{leguebe2014conducting}, where conducting and permeable membrane states are treated separately. An interpretation of these phenomenological models in terms of pore creation and pore growth is given in Section 4, \cite{voyer2018dynamical}. A physically motivated phase-field model of electroporation is introduced in \cite{jaramillo2023phase}. A comparison of several static electroporation models can be found in \cite{jankowiak2020comparison}.

As highlighted in \cite{ivorra2009electric}, \cite{ivorra2010electrical}, tissue-level models that ignore the effect of permeabilization on membrane conductivity fail to reproduce experimental observations. Nonlinear models—where tissue conductivity depends on the electric field—fit experimental data better than linear models with constant conductivity \cite{corovic2013modeling}. To compute the distribution of electric potential under electroporation, electrostatic models of the form
${\rm div}(\sigma \nabla u)=0$ are often employed, with a sigmoid-like conductivity $\sigma=\sigma(|\nabla u|)$ which is supported by experimental observations. In \cite{ivorra2010electrical}, an exponential function was used to describe the strong dependence of membrane conductivity on transmembrane potential. However, the literature lacks a mathematical explanation of how nonlinear conductivity emerges from constant intra- and extracellular conductivities at the cell scale. In this work, we derive a macroscopic (effective) current density and demonstrate numerically that tissue-scale conductivity depends nonlinearly on the applied electric field and evolves in time during pulse application.

The problem under consideration contains a dynamic nonlinear transmission condition, which introduces technical difficulties for proving convergence. There are many works that have addressed homogenization of nonlinear transmission problems in stationary and time-dependent settings. We refer specifically to \cite{amar2004evolution}, \cite{amar2013hierarchy}, and \cite{ammari2017towards}. In \cite{amar2004evolution} the authors study a dynamic problem describing the electrical conduction in biological tissues. The electric potential satisfies electrostatic equations in the intra- and extracellular space and a dynamic condition on the membrane separating these domains. It has been shown that, as the spatial period of the medium (size of cells) goes to zero, the electric potential can be approximated by a homogenized limit solving an elliptic equation with memory effects. 
Homogenization of an electroporation model by Neu and Krassowska \cite{neu1999asymptotic} has been performed in \cite{ammari2017towards}. Under a special choice of initial conditions (asymptotically small) and zero boundary excitation, membrane conductivity is constant in cell problems and thus does not depend on membrane potential (see Section 4.2, \cite{ammari2017towards}). This decouples the system and makes pore density independent of membrane potential. In contrast, our assumptions yield a cell problem coupled with an ODE for the degree of poration, which depends on membrane potential (see \eqref{eq:chi^0-gamma-1}), preserving the nonlinear dependence of membrane conductivity on membrane voltage.
Our problem is also similar to elliptic and parabolic problems involving imperfect interfaces (see e.g. \cite{lipton1998heat}, \cite{faella2007memory}, \cite{jose2009homogenization}, \cite{donato2010corrector}, \cite{donato2015homogenization}, \cite{gahn2016homogenization}, \cite{donpet2026}). It is, however, coupled with an ODE for the degree of poration, which makes the limit problem strongly coupled for the relevant scaling. Note also that in the literature, the nonlinearities are assumed to satisfy strong monotonicity and Lipschitz conditions. To prove the homogenization result, one needs to pass to the limit in the nonlinear terms on the oscillating surfaces. Nonlinear transmission conditions appear also in problems stated in domains with thin heterogeneous layers, as the thickness of the layer goes to zero. In \cite{neuss2007effective} such effective transmission conditions were derived for a
nonlinear reaction-diffusion system defined in such a domain, and properly scaled in the layer region. 
We refer also to \cite{piatnitski2017homogenization} where the authors study coupled nonlinear biomechanical models for plant tissue, and \cite{ptashnyk2015locally} for locally periodic unfolding for surface convergence. Under a different scaling in the microscopic problem ($\ve$ instead of $\ve^{-1}$ in \eqref{eq:orig-prob}), the homogenization yields a bidomain model in the context of cardiac electrophysiology (see e.g. \cite{franzone2002degenerate}, \cite{pennacchio2005},  \cite{amar2013hierarchy}, \cite{collin2018mathematical}, \cite{BenMroSaaTal2019}, \cite{jerez2020derivation}, \cite{AmaAndTim2021}, \cite{jerez2023derivation}). 

The novelty of the present work lies in both analytical and numerical aspects, specifically in the fact that we provide a mathematical explanation for the appearance of a time-dependent and nonlinear effective current. We prove the well-posedness of the microscopic problem for the original model \cite{kavian2014classical}, which is neither monotone nor Lipschitz due to the quadratic surface conductivity term. Our approach is inspired by \cite{liu2011existence}, where existence and uniqueness are proved under generalized monotonicity and coercivity conditions, though our nonlinearity does not directly satisfy these conditions. We perform a dimensional analysis that identifies a relevant scaling: introducing a small parameter $\ve$ as the ratio between cell size and tissue sample size yields an $\ve^{-1}$ scaling in the dynamic interface condition, similar to \cite{amar2004evolution}, \cite{amar2013hierarchy}, but with different initial and boundary conditions, leading to a coupled cell problem \eqref{eq:chi^0-gamma-1}. The main difficulty in passing to the limit is the lack of strong convergence on oscillating surfaces (cell membranes), which is required for nonlinear terms. We prove convergence for a modified nonlinearity ensuring Lipschitz continuity and enabling monotonicity arguments. Specifically, instead of the quadratic term $[u_\ve] X_\ve$, we take $T_M([u_\ve]) X_\ve$ (see \eqref{eq:truncation}) with $T_M([u_\ve]) = \max (-M, \min(M, [u_\ve]))$. This modification aligns with standard approaches in nonlinear transmission problems where the nonlinear functions are assumed to satisfy monotonicity conditions (see \cite{amar2004evolution}, \cite{jose2009homogenization}, \cite{donato2010corrector}, \cite{jerez2023derivation}, \cite{pettersson2025nonlinear}). We derive an effective equation for the leading macroscopic bulk potential, coupled to the membrane dynamics, exhibiting memory effects. Note that the cutoff $T_M([u_\ve])$ could be removed in the case when $\|[u_\ve]\|_{L^\infty((0,T)\times \Gamma_\ve)} \le C$, but this uniform estimate is not available under our assumptions. As for the limit problem, the solution $u_{-1}$ need not be bounded either, as shown in \cite{Amar2009stability} (see the counterexample in Section 5). Finally, we compute numerically the effective conductivity for a simplified case where tissue is placed between two parallel plates under constant voltage. Numerical results in two dimensions show that effective conductivity evolves over time with non-monotonic behavior, qualitatively agreeing with experiments (see Figure \ref{fig:effective_conductivity_over_time}, compare with \cite{ivorra2010electrical}, \cite{corovic2013modeling}). The drop of the effective conductivity can be interpreted by regarding the membrane as a capacitor consisting of two conductive plates separated by an insulator.
When an electric field is applied, charges accumulate on either side of the membrane which creates a transmembrane voltage. 
During this charging phase, ions are pulled toward the membrane, reducing their movement in the bulk solution, which causes a temporary drop in measured conductivity.
Once the membrane is sufficiently charged, defects form allowing ions to pass through the membrane, and conductivity increases sharply. This non-monotonic evolution is a key feature captured by our macroscopic model and agrees with experimental observations.  We also study the dependence of effective conductivity on applied voltage (see Figure \ref{fig:sigma_eff_over_g}), which exhibits a sigmoid-like trend consistent with experiments and electrostatic models used in practice.

The paper is organized as follows: Section~\ref{sec:micro-problem} recalls the cell-scale electropermeabilization model by Kavian et al. \cite{kavian2014classical}, performs dimensional analysis, formulates the microscopic problem in a periodic setting, proves its well-posedness, and describes the main homogenization result. Section~\ref{sec:expansion} presents formal asymptotic expansions revealing the coupled homogenized system \eqref{eq:coupled-gamma-1b} for leading terms and the corrector, and discusses possible decoupling strategies inspired by \cite{amar2004evolution}. Section~\ref{sec:homogenization-main} provides a priori estimates, proves convergence, and establishes well-posedness of the limit problem. Section~\ref{sec:numerics} introduces a numerical algorithm for solving the limit problem and computes the homogenized solution, focusing on effective conductivity and its dependence on time and applied voltage (see Section~\ref{sec:numerical_results}).

\section{Electropermeabilization model}
\label{sec:micro-problem}
In this section we recall the electropermeabilization model on the cell scale by Kavian et al. \cite{kavian2014classical} and formulate the microscopic problem.

Let $\widehat \Omega \subset \mathbb R^3$  consist of the intracellular $\widehat \Omega^c$  and extracellular $\widehat \Omega^e$ parts separated by the cell membrane $\widehat \Gamma$. 
The conductivity of the medium and the electric potential are denoted by
\begin{align*}
\widehat\sigma (\widehat x)= \begin{cases}
\widehat\sigma^c(\widehat x) \,\,\mbox{in}\,\, \widehat \Omega^c,\\
\widehat\sigma^e(\widehat x) \,\,\mbox{in}\,\, \widehat \Omega^e,
\end{cases}
\quad 
\widehat{u}(\widehat x)=\begin{cases}\widehat{u}^c(\widehat x) \,\,\mbox{in}\,\, \widehat \Omega^c,\\
\widehat{u}^e(\widehat x) \,\,\mbox{in}\,\, \widehat \Omega^e.
\end{cases}
\end{align*}
The electric potential is discontinuous through the membrane, and its jump through the membrane $[\widehat{u}] = \widehat{u}^c\big|_{\widehat \Gamma}-\widehat{u}^e\big|_{\widehat \Gamma}$ satisfies the following transmission conditions
\begin{align*}
    \widehat \sigma^c \nabla_{\widehat x} \widehat{u}^c \cdot \n = \widehat \sigma^e \nabla_{\widehat x} \widehat{u}^e \cdot \n,\quad
    \widehat c_m \partial_{\widehat t} [\widehat{u}] + \widehat S_m [\widehat{u}] = - \widehat \sigma^c \nabla_{\widehat x} \widehat{u}^c \cdot \n.
\end{align*}
Here $\widehat S_m$ is the surface conductivity of the membrane, and $\n$ is the unit normal vector exterior to $\widehat \Omega^c$.
 
The time-dependent membrane conductivity $\widehat S_m$ is modelled by interpolating two values: the lipid conductivity $\widehat S_{\rm L}$ and the value $\widehat S_{\rm ir}$ above which permeabilization is not reversible.
\begin{align}
\label{def:S_m}
\widehat S_m= \widehat S_{\rm L} + (\widehat S_{\rm ir}-\widehat S_{\rm L})X,
\end{align}
where the interpolation parameter $X(\widehat t,\widehat x)$ solves the following initial value problem on $\widehat \Gamma$:\\
\begin{align*}
    \begin{split}
    &\partial_{\widehat t} X
    = \max \left(\frac{\widehat{\beta}([\widehat{u}])-X}{\widehat \tau_{\rm ep}}\, ,\,
    \frac{\widehat \beta([\widehat{u}])-X}{\widehat \tau_{\rm res}}\right),\quad
    X(0,\widehat x)=X_0(\widehat x).
    \end{split}
\end{align*}
The function $0\le \widehat \beta\le 1$ is a sigmoid satisfying assumption \ref{H2} (see \cite{kavian2014classical} and \cite{leguebe2014conducting} for specific choices of $\beta$). The constants $\widehat{\tau}_{\rm ep}$ and $\widehat{\tau}_{\rm res}$ designate characteristic electroporation and resealing times, respectively, with $\widehat{\tau}_{\rm ep}$ being several orders of magnitude smaller than $\widehat{\tau}_{\rm res}$ (see Table \ref{tab:dimensionless_coeffs}). 
For a given jump of the membrane potential $[\widehat{u}]$, if $\widehat \beta([\widehat{u}]) - X \geq 0$ the electric pulse on the membrane is high enough to increase the membrane permeability, and the characteristic time is  $\widehat \tau_{\rm ep}$. If, on the other hand, $\widehat \beta([\widehat{u}]) - X < 0$, the membrane reseals on the characteristic timescale $\widehat \tau_{\rm res}$, called resealing time.

\subsection{Dimensional analysis}

Nondimensionalization is a necessary first step in the asymptotic analysis of the problem {that allows us to identify a relevant scaling in the microscopic model. The physical quantities in the model are the electric potential $\widehat u$ [V], the conductivities $\widehat \sigma^e, \widehat \sigma^c$ [S/m], the characteristic membrane capacitance $\widehat c_m$ [F/m$^2$], and the membrane surface conductivity values $\widehat S_{\rm L}$ and $\widehat{S}_{\rm ir}$ [S/m$^2$]. Their respective characteristic scales are $U_0$ [V], $\Sigma_0$ [S/m], and  $C_0$ [F/m$^2$]. The functions $X$ and $\beta$ are dimensionless. 
Let us denote by $L_0$ and $T_0$ a characteristic length of a tissue sample and a characteristic time, respectively, and introduce
\begin{align*}
&x=\frac{\widehat x}{L_0}, \,\, t=\frac{\widehat t}{T_0},\,\, 
\Omega^\alpha=\frac{\widehat{\Omega}^{\alpha}}{L_0}, \,\, 
\Gamma=\frac{\widehat{\Gamma}}{L_0},\,\, u^\alpha(t, x) = \frac{\widehat{u}^\alpha(\widehat{x}, \widehat{t})}{U_0},\\
&\sigma^\alpha = \frac{\widehat{\sigma}^\alpha}{\Sigma_0},\,\, g=\frac{\widehat{g}}{G_0}, \,\, c_m = \frac{\widehat{c}_m}{C_0}, \,\, \tau_{\rm ep}=\frac{\widehat{\tau}_{\rm ep}}{T_0}, \,\, \tau_{\rm res}=\frac{\widehat{\tau}_{\rm res}}{T_0}.
\end{align*}
The values for the physical parameters are taken from \cite{kavian2014classical} and are given in Table \ref{tab:phys_par}.
\begin{table}[H]
\centering
\begin{tabular}{l|l|l|l|l|l|l|l}
Parameter & $\widehat{\sigma}^c$ & $\widehat{\sigma}^e$& $\widehat{c}_m$ & $\widehat S_{\rm L}$&$\widehat{S}_{\rm ir}$& $\widehat{\tau}_{\rm ep}$ & $\widehat{\tau}_{\rm res}$ \\ 
\hline
Value & $0.455$ & $5$ & $9.5\times 10^{-3}$ & $1.9$ & $2.5\times 10^8$ & $10^{-6}$ & $10^{-3}$  \\  \hline
Units & S/m & S/m & F/m$^2$ & S/m$^2$ & S/m$^2$ & s & s  \\ \hline
\end{tabular}
\caption{Parameters in the model \cite{kavian2014classical}.}
\label{tab:phys_par}
\end{table}
This yields the equations in dimensionless variables:
\begin{align}
&{\mbox{div}_{x}}(\sigma^\alpha \nabla_{x} u^\alpha)  =0, \, &\mbox{in}\,\, &\Omega^\alpha, \,\, \alpha=c, e,\nonumber \\
&\sigma^e \nabla_{x} u^e\cdot \n = \sigma^c \nabla_{x} u^c\cdot \n, \, &\mbox{on}\,\, &\Gamma, \nonumber\\
\label{eq:orig-prob_dimless}
&{c}_m \partial_{t} [u] + S_m(X)\, [u]
= - \frac{\Sigma_0 T_0}{C_0 L_0}\sigma^c \nabla_{x} u^c\cdot \n
, \, &\mbox{on}\,\,&\Gamma,
\\
&\partial_{t} X
    = \max \left(\frac{\beta([u])-X}{{\tau}_{\rm ep}}\, ,\,
    \frac{\beta([u])-X}{{\tau}_{\rm res}}\right),&\mbox{on}\,\,&\Gamma, \nonumber\\
&u   =  {\frac{G_0}{U_0}} {g}, \, &\mbox{on}\,\,& \partial \Omega.\nonumber
\end{align}
Here 
\begin{align}
    \label{eq:nondimensional_S_m}
    S_m(X)=S_{\rm L} + ({S}_{\rm ir}-S_{\rm L})X, \,\,
    S_{\rm L} =\frac{T_0}{C_0}\widehat S_{\rm L}, \,\,
    {S}_{\rm ir} =\frac{T_0}{C_0}\widehat{S}_{\rm ir}, \,\, \beta([u])=\widehat \beta(U_0[u]).
\end{align}
We take a typical cell size in the tissue to be $l_0=2\times 10^{-4}$ m. On the microscopic level, the contributions of the terms $\widehat c_m \partial_{\widehat t} [\widehat u]$ and $\widehat \sigma^c \nabla_{\widehat x} \widehat u^c \cdot \n$ on the cell membrane are of the same order of magnitude, that is $C_0/T_0$ is of the same order as $\Sigma_0/l_0$, which gives a characteristic timescale $T_0$ of the order of microseconds. Thus, we set $T_0=10^{-6}$ s. 
In an experimental setup, one can use a cubic piece of e.g. liver with two electrodes inserted. The distance between the electrodes would be several millimetres, so we choose $L_0=10^{-2}$ m. We set $U_0=1$ V and $C_0=\widehat{c}_m=9.5\times 10^{-3}$ F/m$^2$. The applied voltage $\widehat{g}$ is chosen to be in the range between $500$ and  $5000$ V for the characteristic length $L_0=10^{-2}$ m, which corresponds to an electrical field strength $E$ of $0.5-5$ kV/cm. Introducing a small parameter $\ve = \frac{l_0}{L_0}$, we see that $\frac{\Sigma_0 T_0}{C_0 L_0}$, for $\Sigma_0=O(1)$ S/m, is of the same order as $\ve$, $\ve \to 0$. So, we set $\frac{\Sigma_0 T_0}{C_0 L_0}= \ve$, which gives the value $\Sigma_0=l_0 C_0/T_0=1.9$ S/m. Similarly, we set $\widehat{g}=G_0 g = U_0 g/\ve$, which gives $G_0=U_0 L_0/l_0=50$ V. This motivates the scaling in the microscopic model \eqref{eq:orig-prob} in the next section. 
The values of the dimensionless parameters are given in Table~\ref{tab:dimensionless_coeffs} and are calculated from the parameters in Table 1, \cite{kavian2014classical}.}

\begin{table}[H]
\centering
\begin{tabular}{l|l|l|l}
    Parameter & Symbol & Definition & Value \\[1mm]
    \hline
    Intracellular conductivity & $\sigma^c$ & $\frac{\widehat{\sigma}^c}{\Sigma_0}$ & {$0.239$} \\[1mm]
    Extracellular conductivity & $\sigma^e$ & $\frac{ \widehat{\sigma}^e}{\Sigma_0} $ & {2.632}\\[1mm]
    Membrane surface conductivity &$S_{\rm L}$ & $\frac{T_0 \widehat S_{\rm L}}{C_0}$ & $2 \cdot 10^{-4}$ \\[1mm]
    EP membrane surface conductivity & $S_{\rm ir}$ & $\frac{T_0 \widehat{S}_{ir}}{C_0}$ & $2.63 \cdot 10^4 $ \\[1mm]
    EP  time & $\tau_{\text{ep}}$ & $ \frac{\widehat{\tau}_{\text{ep}}}{T_0}$ & 1 \\[1mm]
    Resealing time & $\tau_{\text{res}}$ & $\frac{\widehat{\tau}_{\text{res}}}{T_0}$ & 1000 \\[1mm]
    Membrane capacitance & ${c}_m$ & $\frac{\widehat c_m}{C_0}$ & 1 \\
    \hline
\end{tabular}
\caption{Dimensionless coefficients.}
\label{tab:dimensionless_coeffs}
\end{table}

\subsection{Problem statement and homogenization result}
We denote by $Y=(0,1)^3$ the periodicity cell consisting of the extracellular and intracellular parts, and the membrane $Y = Y^e \cup \Gamma \cup Y^c$. Assume that
\[
Y^c = \cup_{j=1}^N Y_j^c, \quad
\Gamma = \cup_{j=1}^N \Gamma_j, \quad
Y^e = Y\setminus \overline{Y^c},
\]
where $Y_j^c$, $j=1, \ldots, N$ are simply connected bounded Lipschitz domains with pairwise disjoint closures, compactly contained in $Y$. Assume also that $\Gamma_j = \partial Y_j^c$ and $Y^e$ are connected.
Let us rescale the periodicity cell $Y=(0,1)^3$ by a small parameter $\ve>0$ representing the ratio between the typical length of a cell and the diameter of the domain under consideration. 
We assume that $\Omega\subset\mathbb R^3$ is a bounded domain with $C^2$ boundary.
For a fixed $\kappa_0>0$, independent of $\ve$, define the index set
\[
K_\ve
=
\left\{
k\in\mathbb Z^3: \,\, \ve(k + \overline{Y^c}) \subset \Omega, \,\, 
\operatorname{dist}\bigl(\partial\Omega,\ve(Y^c+k)\bigr)
\ge \kappa_0\ve
\right\}.
\]
The intracellular and extracellular regions, and the corresponding membrane are then given by
\[
\Omega_\ve^c
=
\bigcup_{k\in K_\ve}\ve(Y^c+k), \quad \Omega_\ve^e
=
\Omega\setminus \overline{\Omega_\ve^c},
\quad
\Gamma_\ve
=
\bigcup_{k\in K_\ve}\ve(\Gamma+k).
\]
Thus
$
\Omega
=
\Omega_\ve^e\cup \Gamma_\ve\cup \Omega_\ve^c$,
where \(\Gamma_\ve=\partial\Omega_\ve^c\cap \Omega\).
The numerical computations in Section \ref{sec:numerics} are performed for the case when $Y^c$ is strictly included in $Y$ (see Figure \ref{fig:omega}), so for simplicity we assume that this is the case throughout the paper. The result is, however, true for the case when $Y^c$ intersects the boundary of $Y$, as in \cite{amar2004evolution}, \cite{hummel2000homogenization}.   

\begin{figure}
\centering
\begin{subfigure}{0.45\textwidth}
\centering
    \def\svgwidth{0.7\textwidth}
\begingroup%
  \makeatletter%
  \providecommand\color[2][]{%
    \errmessage{(Inkscape) Color is used for the text in Inkscape, but the package 'color.sty' is not loaded}%
    \renewcommand\color[2][]{}%
  }%
  \providecommand\transparent[1]{%
    \errmessage{(Inkscape) Transparency is used (non-zero) for the text in Inkscape, but the package 'transparent.sty' is not loaded}%
    \renewcommand\transparent[1]{}%
  }%
  \providecommand\rotatebox[2]{#2}%
  \newcommand*\fsize{\dimexpr\f@size pt\relax}%
  \newcommand*\lineheight[1]{\fontsize{\fsize}{#1\fsize}\selectfont}%
  \ifx\svgwidth\undefined%
    \setlength{\unitlength}{282.4151347bp}%
    \ifx\svgscale\undefined%
      \relax%
    \else%
      \setlength{\unitlength}{\unitlength * \real{\svgscale}}%
    \fi%
  \else%
    \setlength{\unitlength}{\svgwidth}%
  \fi%
  \global\let\svgwidth\undefined%
  \global\let\svgscale\undefined%
  \makeatother%
  \begin{picture}(1,0.70627464)%
    \lineheight{1}%
    \setlength\tabcolsep{0pt}%
    \put(0,0){\includegraphics[width=\unitlength,page=1]{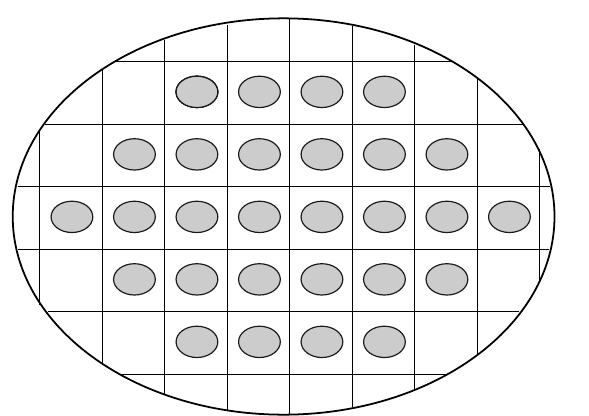}}%
    \put(0.90555806,0.50375245){\color[rgb]{0,0,0}\makebox(0,0)[lt]{\lineheight{1.25}\smash{\begin{tabular}[t]{l}$\Omega_\varepsilon^e$\end{tabular}}}}%
    \put(0,0){\includegraphics[width=\unitlength,page=2]{omega_eps.pdf}}%
    \put(0.91618073,0.15320456){\color[rgb]{0,0,0}\makebox(0,0)[lt]{\lineheight{1.25}\smash{\begin{tabular}[t]{l}$\Omega_\varepsilon^c$\end{tabular}}}}%
    \put(0,0){\includegraphics[width=\unitlength,page=3]{omega_eps.pdf}}%
    \put(0.75684082,0.68433773){\color[rgb]{0,0,0}\makebox(0,0)[lt]{\lineheight{1.25}\smash{\begin{tabular}[t]{l}$\Gamma_\varepsilon$\end{tabular}}}}%
    \put(0,0){\includegraphics[width=\unitlength,page=4]{omega_eps.pdf}}%
    \put(-0.00267987,0.59935648){\color[rgb]{0,0,0}\makebox(0,0)[lt]{\lineheight{1.25}\smash{\begin{tabular}[t]{l}$\partial \Omega$\end{tabular}}}}%
    \put(0,0){\includegraphics[width=\unitlength,page=5]{omega_eps.pdf}}%
  \end{picture}%
\endgroup%

    \caption{}
    \label{fig:omega}
\end{subfigure}
\begin{subfigure}{0.45\textwidth}
\centering
    \def\svgwidth{0.4\textwidth}
\begingroup%
  \makeatletter%
  \providecommand\color[2][]{%
    \errmessage{(Inkscape) Color is used for the text in Inkscape, but the package 'color.sty' is not loaded}%
    \renewcommand\color[2][]{}%
  }%
  \providecommand\transparent[1]{%
    \errmessage{(Inkscape) Transparency is used (non-zero) for the text in Inkscape, but the package 'transparent.sty' is not loaded}%
    \renewcommand\transparent[1]{}%
  }%
  \providecommand\rotatebox[2]{#2}%
  \newcommand*\fsize{\dimexpr\f@size pt\relax}%
  \newcommand*\lineheight[1]{\fontsize{\fsize}{#1\fsize}\selectfont}%
  \ifx\svgwidth\undefined%
    \setlength{\unitlength}{290.53346889bp}%
    \ifx\svgscale\undefined%
      \relax%
    \else%
      \setlength{\unitlength}{\unitlength * \real{\svgscale}}%
    \fi%
  \else%
    \setlength{\unitlength}{\svgwidth}%
  \fi%
  \global\let\svgwidth\undefined%
  \global\let\svgscale\undefined%
  \makeatother%
  \begin{picture}(1,0.95820571)%
    \lineheight{1}%
    \setlength\tabcolsep{0pt}%
    \put(0,0){\includegraphics[width=\unitlength,page=1]{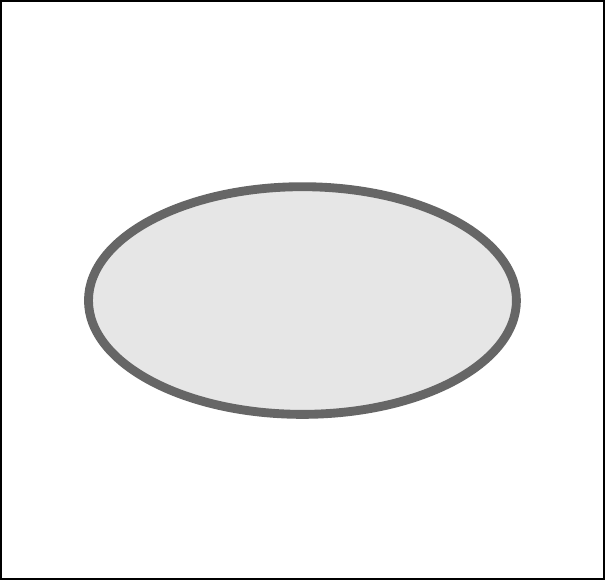}}%
    \put(0.41641138,0.46517138){\color[rgb]{0,0,0}\makebox(0,0)[lt]{\lineheight{1.25}\smash{\begin{tabular}[t]{l}$Y^c$\end{tabular}}}}%
    \put(0.72618629,0.77494621){\color[rgb]{0,0,0}\makebox(0,0)[lt]{\lineheight{1.25}\smash{\begin{tabular}[t]{l}$Y^e$\end{tabular}}}}%
    \put(0.20989493,0.77494621){\color[rgb]{0,0,0}\makebox(0,0)[lt]{\lineheight{1.25}\smash{\begin{tabular}[t]{l}$\Gamma$\end{tabular}}}}%
    \put(0,0){\includegraphics[width=\unitlength,page=2]{Y.pdf}}%
    \put(0.65129434,0.13276847){\color[rgb]{0,0,0}\makebox(0,0)[lt]{\lineheight{1.25}\smash{\begin{tabular}[t]{l}$\mathbf n$\end{tabular}}}}%
  \end{picture}%
\endgroup%

\end{subfigure}
\caption{(a) $\Omega=\Omega_\ve^c \cup \Gamma_\ve \cup \Omega_\ve^e$; (b) Periodicity cell $Y=Y^c \cup \Gamma \cup Y^e$.}
\label{fig:Y}
\end{figure}

We denote by $w_\ve^\alpha$, $\alpha=c, e$, the electric potential in $\Omega_\ve^\alpha$, and by $[w_\ve]=w_\ve^c\big|_{\Gamma_\ve} - w_\ve^e\big|_{\Gamma_\ve}$ its jump through the membrane $\Gamma_\ve$. The microscopic dynamical problem describing the evolution of the electric potential for $t\in[0,T]$ reads
\begin{align}
&\di{\sigma_\ve \nabla w_\ve}  =0, \, &\mbox{in}\,\, &(0,T]\times(\Omega_\ve^c \cup \Omega_\ve^e), \nonumber \\
&[\sigma_\ve \nabla w_\ve \cdot \n]  = 0, \, &\mbox{on}\,\, &(0,T]\times\Gamma_\ve, \nonumber\\
\label{eq:orig-prob}
&\ve^{-1} (c_m \partial_t [w_\ve] + S_m(X_\ve)\, [w_\ve])
= - \sigma^c\big(\frac{x}{\ve}\big) \nabla w_\ve^c\cdot \n
, \, &\mbox{on}\,\,&(0,T]\times\Gamma_\ve,
\\
&\partial_t X_\ve
    = f([w_\ve], X_\ve),&\mbox{on}\,\,&(0,T]\times\Gamma_\ve, \nonumber\\
& [w_\ve](0,x) =V_\ve^{\rm in}(x), \quad X_\ve(0,x)=X_\ve^{\rm in}(x)\, &\mbox{on}\,\,&\Gamma_\ve, \nonumber \\
&w_\ve   = \ve^{-1} g, \, &\mbox{on}\,\,& (0,T]\times\partial \Omega. \nonumber
\end{align}
Here $\sigma_\ve=\sigma\big(\frac{x}{\ve}\big)$ is $\ve Y$-periodic, and denoting $S_0=S_{\rm L}, S_1=S_{\rm ir}-S_{\rm L}$, we set
\begin{align}
\label{def:S_m_f}
\begin{alignedat}{2}
\displaystyle
S_m(X)&= S_0 + S_1 X,\quad S_0, S_1 > 0,\\
f([u], X)&=\max \left(\frac{\beta([u])-X}{\tau_{\rm ep}}\, ,\,
\frac{\beta([u])-X}{\tau_{\rm res}}\right).
\end{alignedat}
\end{align}
To reduce the problem to one with homogeneous Dirichlet boundary conditions,  we make a change of unknowns and set 
\begin{align*}
w_\ve =u_\ve + \ve^{-1} p_\ve,
\end{align*}
where $p_\ve$, for each $t\in [0,T]$, solves the nonhomogeneous Dirichlet problem
\begin{align}
\label{eq:p_eps}
&\di{\sigma_\ve \nabla p_\ve} = 0\, &\mbox{in}\,\,& \Omega,  \\
&p_\ve = g \, &\mbox{on}\,\,& \partial \Omega.\nonumber
\end{align}
Since both $p_\ve$ and the flux $\sigma_\ve \nabla p_\ve \cdot \n$ are continuous through $\Gamma_\ve$, the jump through the membrane does not change $[u_\ve] = [w_\ve]$, and the new unknown function $u_\ve$ solves the following problem:
\begin{align}
&\di{\sigma_\ve \nabla u_\ve}  = 0, \, &\mbox{in}\,\, &(0,T]\times(\Omega_\ve^c \cup \Omega_\ve^e), \nonumber \\
&[\sigma_\ve \nabla u_\ve  \cdot \n] = 0, \, &\mbox{on}\,\, &(0,T]\times\Gamma_\ve, \nonumber\\
\label{eq:microscopic-prob}
&\ve^{-1} (c_m \partial_t [u_\ve] + S_m(X_\ve)\, [u_\ve])
= - \sigma^c\big(\frac{x}{\ve}\big) \nabla u_\ve\cdot \n- \ve^{-1} \sigma^c\big(\frac{x}{\ve}\big) \nabla p_\ve\cdot \n
, \, &\mbox{on}\,\,&(0,T]\times\Gamma_\ve,
\\
&\partial_t X_\ve
    = f([u_\ve], X_\ve),&\mbox{on}\,\,&(0,T]\times\Gamma_\ve, \nonumber\\
& [u_\ve](0,x) =V_\ve^{\rm in}(x), \quad X_\ve(0,x)=X_\ve^{\rm in}(x)\, &\mbox{on}\,\,&\Gamma_\ve, \nonumber\\
&u_\ve   = 0, \, &\mbox{on}\,\,& (0,T]\times\partial \Omega. \nonumber
\end{align}
We impose the following assumptions:
\begin{enumerate}[label=(H\arabic*)]
\setcounter{enumi}{0}
\item \label{H1} $\sigma_\ve(x)=\sigma\big(\frac{x}{\ve}\big)$, where $\sigma(y)=\begin{cases}
    \sigma^c(y) \, \mbox{in} \, Y^c,\\
    \sigma^e(y) \, \mbox{in} \, Y^e
\end{cases}$ with $0< \underline{\sigma}\le \sigma \le \overline{\sigma}<\infty$. 

\item \label{H2} $f$ is defined in \eqref{def:S_m_f} where $\beta$ is even, nondecreasing {on $\mathbb{R}_+$}, Lipschitz continuous on $\mathbb R$, and such that $
0 \leq \beta(\xi) \leq 1$, $\beta(\xi) \to 1$ as $|\xi| \to \infty$.

\item \label{H3} 
The initial data satisfy $V_\ve^{\rm in} \in H^{1/2}(\Gamma_\ve)$, $X_\ve^{\rm in}\in H^{1/2}(\Gamma_\ve) \cap L^\infty(\Gamma_\ve)$ and $0\le X_\ve^{\rm in}\le 1$.

{The function $V_\ve^{\rm in}$ is assumed to admit an extension, still denoted by $V_\ve^{\rm in}$, such that $V_\ve^{\rm in}\big|_{\Gamma_\ve} = [u_\ve]\big|_{t=0}$ and $V_\ve^{\rm in}\in H^1(\Omega_\ve^c)$ satisfies the uniform estimate
\begin{align*}
\ve \|\nabla V_\ve^{\rm in}\|_{L^2(\Omega_\ve^c)} \le C.
\end{align*}
}
We assume in addition that 
\begin{align*}
\ve \|X_\ve^{\rm in}\|_{L^2(\Gamma_\ve)}^2
+ \ve \|V_\ve^{\rm in}\|_{L^2(\Gamma_\ve)}^2\le C, 
\end{align*} 
and $V_\ve^{\rm in}, X_\ve^{\rm in}$ converge strongly two-scale to $V^{\rm in}, X^{\rm in}\in L^2(\Omega \times \Gamma)$ (see Definition \ref{def:2scale-Gamma_eps-stationary}):
\begin{align*}
&\lim_{\ve \to 0}\ve \int_{\Gamma_\ve} |V_\ve^{\rm in}(x)|^2\, \dd S = \int_\Omega \int_\Gamma |V^{\rm in}(x,y)|^2\, \dd S\, \dd x,\\
&\lim_{\ve \to 0}\ve \int_{\Gamma_\ve} |X_\ve^{\rm in}(x)|^2\, \dd S = \int_\Omega \int_\Gamma |X^{\rm in}(x,y)|^2\, \dd S\, \dd x,
\end{align*}

\item \label{H4} $g\in C([0,T]; H^{1/2}(\partial \Omega))$, $\partial_t g \in L^2(0,T; H^{1/2}(\partial \Omega))$.

\item \label{H5} $g\in C([0,T]; H^{3/2}(\partial \Omega))$, $\partial_t g \in L^2(0,T; H^{3/2}(\partial \Omega))$.

\end{enumerate}
{
\begin{remark}
Under assumption \ref{H4}, $p_\ve \in C([0,T]; H^1(\Omega))$, $\partial_t p_\ve \in L^2(0,T; H^1(\Omega))$ and $\sigma_\ve \nabla p_\ve \cdot \n \in C([0,T]; H^{-1/2}(\Gamma_\ve))$ is well-defined. 
\end{remark}
\begin{remark}
For the existence of a solution, uniform a priori estimates in Lemma \ref{lm:apriori-est}, and the passage to the limit it is enough to assume \ref{H4}. An additional regularity \ref{H5} implies the $H^2(\Omega)$-regularity of $p_0$ solving \eqref{eq:p_0}, which in its turn yields the strong two-scale convergence of $\nabla p_\ve$ used in Lemma \ref{lm:compactness} (see \eqref{eq:convergence_p_eps}) when identifying the limits of the nonlinear functions.
\end{remark}
}
We prove the well-posedness of \eqref{eq:microscopic-prob} in Section \ref{sec:existence} and derive a priori estimates in Section \ref{sec: a_priori}.
The main goal is to derive a macroscopic model for \eqref{eq:microscopic-prob} and simulate the effective conductivity and the transmembrane potential (Section \ref{sec:numerics}). We will pass to the limit as $\ve \to 0$ using the two-scale convergence. 

To identify the limits of the nonlinear functions,  we impose an extra assumption on the membrane conductivity term replacing $S_m(X_\ve) [u_\ve]$ with a modified function 
\begin{align}
\label{eq:truncation}
S_M([u_\ve], X_\ve)= {S_0 [u_\ve] + S_1 X_\ve T_M([u_\ve])}, \quad  T_M([u_\ve]) = \max\big(-M, \min(M, [u_\ve])\big),
\end{align}
for some large $M>0$. 
This makes the nonlinear term $S_M$ satisfy a one-sided Lipschitz condition for $X_\ve$ solving \eqref{eq:microscopic-prob} and thus allows us to utilize the monotonicity along the solutions (Lemma \ref{lm:Lip-S}) to identify the limits of the nonlinear terms and prove the strong two-scale convergence for $([u_\ve], X_\ve)$.
Namely, we prove the convergence (in the sense of Lemma \ref{lm:compactness}) of $(u_\ve, [u_\ve], X_\ve)$ solving
\begin{align}
&\di{\sigma_\ve \nabla u_\ve}  = 0, \, &\mbox{in}\,\, &(0,T]\times(\Omega_\ve^c \cup \Omega_\ve^e), \nonumber \\
&[\sigma_\ve \nabla u_\ve ] \cdot \n = 0, \, &\mbox{on}\,\, &(0,T]\times\Gamma_\ve, \nonumber\\
\label{eq:microscopic-prob-truncated}
&\ve^{-1} (c_m \partial_t [u_\ve] + S_M([u_\ve], X_\ve))
= - \sigma^c\big(\frac{x}{\ve}\big) \nabla u_\ve\cdot \n- \ve^{-1} \sigma^c\big(\frac{x}{\ve}\big) \nabla p_\ve\cdot \n
, \, &\mbox{on}\,\,&(0,T]\times\Gamma_\ve,
\\
&\partial_t X_\ve
    = f([u_\ve], X_\ve),&\mbox{on}\,\,&(0,T]\times\Gamma_\ve, \nonumber\\
& [u_\ve](0,x) =V_\ve^{\rm in}(x), \quad X_\ve(0,x)=X_\ve^{\rm in}(x)\, &\mbox{on}\,\,&\Gamma_\ve, \nonumber\\
&u_\ve   = 0, \, &\mbox{on}\,\,& (0,T]\times\partial \Omega. \nonumber
\end{align}
to a solution $(u_{-1}, u_0, X_0)$ of the coupled macroscopic problem
\begin{align}
&{\rm div}_x {\int_{Y^c\cup Y^e} \sigma(\nabla_y u_0 + \nabla_x u_{-1})\, dy}=0, \, & \mbox{in} \,\, &(0,T]\times\Omega, \nonumber\\
&\diy{\sigma (\nabla_y u_0 + \nabla_x u_{-1})}  =0, \, & \mbox{in} \,\, &(0,T]\times\Omega \times (Y^e\cup Y^c),  \nonumber\\
&[\sigma (\nabla_y u_0+\nabla_x u_{-1}) \cdot \n] =0, \, &\mbox{in} \,\, &(0,T]\times \Omega \times \Gamma, \nonumber \\
\label{eq:effective-problem-coupled-strong-intro}
& -\sigma^c \nabla_y u_0^c \cdot \n
=c_m \partial_t [u_0] + S_M([u_0], X_0) + \sigma^c\nabla_x u_{-1}\cdot \n &&\\
&\hspace{2.4cm} + \partial_{x_j} p_0\, {\sigma^c (\nabla_y \mathcal{M}_j + \mathbf{e_j})\cdot \n},\, &\mbox{on} \,\, &(0,T]\times \Omega \times \Gamma, \nonumber\\
& \partial_t X_0 = f([u_0], X_0), \, &\mbox{on} \,\, &(0,T]\times \Omega \times \Gamma,\nonumber\\
& [u_0](0,x,y)=V^{\rm in}(x,y),\quad X_0(0,x,y)=X^{\rm in}(x, y),\, &\mbox{on} \,\, &\Omega \times \Gamma, \nonumber\\
& u_{-1}\big|_{x\in \partial \Omega}=0, \quad u_0^e(x,\cdot) \,\, \mbox{is} \, \, Y-\mbox{periodic}.\nonumber
\end{align}
Here we denote $[u_0]=u_0^c - u_0^e$; $p_0$ is the strong $L^2(0,T; L^2(\Omega))$ limit  of $p_\ve$ solving the classical homogenized problem \eqref{eq:p_0}:
\begin{align*}
&\di{A \nabla p_0} = 0, \quad x\in\Omega,  \\
&p_0 = g, \quad x\in \partial \Omega,\nonumber
\end{align*}
with $A = \int_Y \sigma(\nabla \mathcal{M} +\text{Id})\, \dd y$, and $\mathcal{M}_j(y)$, $j=1, 2, 3$, solving standard cell problems
\begin{align*}
&\di{\sigma (\nabla \mathcal{M}_j+e_j)} = 0, \quad y\in Y,  \\
&\mathcal{M}_j(y) \,\, \mbox{is}\,\, Y-\mbox{periodic}.\nonumber
\end{align*}
Moreover, by Theorem 2.6 in \cite{allaire1992homogenization},
\begin{align*}
\begin{split}
    &\Big(\nabla p_\ve - \nabla p_0 - \nabla_y \mathcal{M}_j\big(\frac{x}{\ve}\big) \partial_{x_j} p_0\Big) \to 0 \quad \mbox{strongly in}\,\, L^2(0,T; L^2(\Omega)).
\end{split}
\end{align*}

In Section \ref{sec:homogenization} we prove our main homogenization result which we formulate below. A complete description of the convergence is given in Lemma \ref{lm:compactness}.
\begin{theorem}
\label{th:main}
    Under assumptions \ref{H1}--\ref{H3}, \ref{H5}, $(\ve u_\ve, [u_\ve], X_\ve)$ solving \eqref{eq:microscopic-prob-truncated} converges as $\ve \to 0$ to the solution $(u_{-1}, [u_0], X_0)$ of \eqref{eq:effective-problem-coupled-strong-intro} in the following sense:
    \begin{itemize}
        \item[(i)] 
        $\ve u_\ve \rightharpoonup u_{-1}$ weakly in $L^2(0,T; L^2(\Omega))$;\\
        $\ve \mathbf{1}_{\Omega_\ve^\alpha}(x) \nabla u_\ve(t,x) \, \stackrel{2}{\rightharpoonup }\, \mathbf{1}_{Y^\alpha}(y)\Big(\nabla_x u_{-1}(t,x) + \nabla_y u_0^\alpha(t,x,y)\Big)$ weakly two-scale in $L^2(0,T; L^2(\Omega))^3$ (Definition \ref{def:2-scale}), where  $u_0^e(t,x,\cdot)$ is $Y$-periodic.

        \item[(ii)] $([u_\ve], X_\ve) \rightharpoonup ([u_0], X_0)$ weakly $t$-pointwise  two-scale in $L^2(\Gamma_\ve)^2$ in the sense of Definition \ref{def:2scale-Gamma_eps-pointwise}.
        \item[(iii)] Moreover, $([u_\ve]\big|_{\Gamma_\ve}, X_\ve)$ converges strongly two-scale, pointwise in $t\in[0,T]$, to $([u_0]\big|_{\Gamma}, X_0)=((u_0^c - u_0^e)\big|_{y\in \Gamma}, X_0)$:
        \begin{align}
        \label{eq:strong-2scale}
        \displaystyle \quad \lim_{\ve \to 0} \ve \int_{\Gamma_\ve} [u_\ve]^2\, \dd S = \int_{\Omega}\int_\Gamma [u_0]^2 \, \dd S\, \dd x, \qquad \lim_{\ve \to 0} \ve \int_{\Gamma_\ve} X_\ve^2\, \dd S = \int_{\Omega}\int_\Gamma X_0^2 \, \dd S\, \dd x.
        \end{align}
    \end{itemize}
\end{theorem}
The coupled system \eqref{eq:effective-problem-coupled-strong-intro} cannot be decoupled, as in the classical homogenization \cite{allaire1992homogenization} and in \cite{amar2004evolution} or \cite{ammari2017towards}. We propose an ansatz for the corrector term $u_0$ (see \eqref{eq:u_0}) which, even though not explicit in $u_{-1}$, allows us to derive a nonlinear limit equation for the leading term $u_{-1}$.  
More precisely, we can formulate an equation for the leading term
\begin{align}
\label{eq:u_(-1)-intro}
\begin{split}
-&\di{A\nabla u_{-1} + \int_0^t B(t,\tau) \nabla(u_{-1}+p_0)(\tau)\, d\tau}={\rm div}\, C(t, x), \quad x\in \Omega,\\
&u_{-1}=0, \quad x\in \partial \Omega,
\end{split}
\end{align}
where the coefficient $A$ \eqref{def:A} is computed in terms of a standard time-independent cell problem \eqref{eq:cell-M}; $B(t,\tau)$ \eqref{def:B} is computed in terms of a coupled, nonlinear, and time-dependent cell problem \eqref{eq:chi^0-gamma-1}, and $C$ \eqref{def:C} is computed in terms of a time-dependent cell problem \eqref{eq:T-eps^(-1)-g/eps}. This is a nonlinear integro-differential equation with memory effects. The effective electric current density vector (given as the term under the divergence in \eqref{eq:u_(-1)-intro}) depends in a nonlinear way on the electric potential and contains an instantaneous and a time-dependent part, depending on the history of the potential gradient. 

In Section \ref{sec:numerical_results}, for a specific case of tissue between two parallel plates in two dimensions, we rewrite the equation \eqref{eq:u_(-1)-intro} as $\di{\sigma_\eff(t, \nabla (u_{-1}+p_0)) \nabla (u_{-1}+p_0)}=0$ and compute the effective conductivity $\sigma_\eff$. We see that it depends in a nonlinear way on the applied voltage $g$ and shows nontrivial behavior in time $t$.


\subsection{Well-posedness of the microscopic problem \eqref{eq:microscopic-prob}}
In this section, we prove the well-posedness of the non-truncated system \eqref{eq:microscopic-prob}.
The weak formulation of \eqref{eq:microscopic-prob} reads: Find $u_\ve:[0,T]\times \Omega \to \mathbf{R}$, $X_\ve:[0,T]\times \Gamma_\ve \to \mathbf{R}$ satisfying the initial conditions $[u_\ve](0,x) =V_\ve^{\rm in}(x)$, $X_\ve(0,x)=X_\ve^{\rm in}(x)$ and the following identities
\begin{align}
\begin{split}
\label{eq:weak-v_eps_w_eps}
    &\frac{1}{\ve} \int_{\Gamma_{\ve}} (c_m \partial_t [u_{\ve}]+ S_m(X_{\ve}) [u_{\ve}])[\phi]  \, \dd S + \int_{\Omega_\ve^c \cup \Omega_\ve^e} \sigma_{\ve} \nabla u_{\ve} \cdot \nabla \phi \, \dd x = -\frac{1}{\ve}\int_{\Omega_\ve^c \cup \Omega_\ve^e} \sigma_\ve \nabla p_\ve \cdot \nabla \phi\, \dd x,\\
    &\int_{\Gamma_{\ve}} \partial_t X_\ve\, \psi\, \dd S
    = \int_{\Gamma_{\ve}} f([u_\ve],X_\ve)\, \psi\, \dd S,
\end{split}
\end{align}
for any test function
$\phi=(\phi^c, \phi^e)\in  H^1(\Omega_\ve^c)\times H^1( \Omega_\ve^e)$ such that $\phi |_{\partial \Omega} = 0$, and $\psi \in L^2(\Gamma_\ve)$.


\begin{theorem}
\label{th:existence-eps-prob}
    Let the assumptions \ref{H1}--\ref{H4} hold.  Then, for each fixed $\ve >0$, there exists a unique solution $([u_\ve], X_\ve)$ of problem \eqref{eq:microscopic-prob} such that
    \begin{align*}
    &[u_\ve] \in C([0,T]; L^2(\Gamma_\ve))\cap L^2(0,T; H^{1/2}(\Gamma_\ve)), \quad \partial_t [u_\ve] \in L^2(0,T; H^{-1/2}(\Gamma_\ve)),\\
    &u_\ve=(u_\ve^c, u_\ve^e) \in L^2(0,T; H^1(\Omega_\ve^c)\times  H^1(\Omega_\ve^e)),\quad u_\ve^e \big|_{\partial \Omega}=0,\\
    &X_\ve \in C([0,T];H^{1/2}(\Gamma_\ve)), \quad \partial_t X_\ve \in C([0,T]; L^2(\Gamma_\ve)).
    \end{align*}
If in addition \ref{H5} holds, we have $[u_\ve] \in C([0,T]; H^{1/2}(\Gamma_\ve))$, $\partial_t [u_\ve] \in L^2(0,T; L^2(\Gamma_\ve))$.
\end{theorem}
\begin{remark}
The well-posedness of the system with truncation \eqref{eq:microscopic-prob-truncated} is classical because the nonlinearities are Lipschitz. It can also be proved in the same way as Theorem \ref{th:existence-eps-prob} for \eqref{eq:microscopic-prob}.
\end{remark}
\begin{proof}
In order to rewrite problem \eqref{eq:weak-v_eps_w_eps} as an evolution equation on $\Gamma_\ve$, we introduce the operator $\mathcal{L}_{\ve} : H^{1/2}(\Gamma_{\ve}) \to H^{-1/2}(\Gamma_{\ve})$ defined by
\begin{align}
\label{def:L_eps}
    \langle \mathcal{L}_{\ve} W , [\phi] \rangle = \int_{\Omega_\ve^c \cup \Omega_\ve^e} \sigma_{\ve} \nabla w_{\ve} \cdot \nabla \phi \, \dd x, \quad\phi=(\phi^c, \phi^e) \in H^1(\Omega_\ve^c)\times H^1( \Omega_\ve^e), \,\,\phi |_{\partial \Omega} = 0,
\end{align}
where $w_{\ve}=(w_\ve^c, w_\ve^e) \in H^1(\Omega_\ve^c)\times H^1( \Omega_\ve^e)$, for a given jump $W\in H^{1/2}(\Gamma_\ve)$, is a solution to the following elliptic transmission problem:
\begin{align}
-&\di{\sigma_\ve \nabla w_\ve}  =0, \, &\mbox{in}\,\, &\Omega_\ve^c \cup \Omega_\ve^e, \nonumber \\
&[\sigma_\ve \nabla w_\ve  \cdot \n] = 0, \, &\mbox{on}\,\, &\Gamma_\ve, \nonumber\\
\label{eq:microscopic-prob_stat}
&w_{\ve}^c - w_{\ve}^e = W
, \, &\mbox{on}\,\,&\Gamma_\ve,
\\
&w_\ve   = 0, \, &\mbox{on}\,\,& \partial \Omega. \nonumber
\end{align}
{Note that zero jump $[\phi]=0$ yields $\langle \mathcal{L}_{\ve} W , [\phi] \rangle=0$.
A solution $w_\ve$ of \eqref{eq:microscopic-prob_stat} exists and is unique. Indeed, we can construct $\widetilde  w_\ve =(\widetilde  w_\ve^c, \widetilde  w_\ve^e) \in H^1(\Omega_\ve^c)\times H^1(\Omega_\ve^e)$ such that $W=\widetilde  w_\ve^c - \widetilde  w_\ve^e$ and $\widetilde  w_\ve^e\big|_{\partial \Omega}=0$. For example, by setting $\widetilde{w}_\ve^c=0$ and solving
\begin{align*}
        &\Delta \widetilde  w_\ve^e =0 \quad\mbox{in}\,\, \Omega_\ve^e, \quad \widetilde  w_\ve^e|_{\Gamma_\ve}=- W,\quad\widetilde  w_\ve^e|_{\partial \Omega}=0. 
\end{align*}
Next, we change the unknowns $z_\ve^\alpha = w_\ve^\alpha - \widetilde  w_\ve^\alpha$, which implies $[z_\ve] = [w_\ve] - [\widetilde  w_\ve] = 0$ on $\Gamma_\ve$.
The equations \eqref{eq:microscopic-prob_stat} transform into
\begin{align}
\label{eq:no-jump}
    \begin{split}
        -&{\rm div}(\sigma_\ve \nabla z_\ve)={\rm div}(\sigma_\ve \nabla \widetilde  w_\ve), \quad x\in \Omega_\ve^c \cup \Omega_\ve^e,\\
        &[\sigma_\ve\nabla z_\ve \cdot \n]=-[\sigma_\ve\nabla \widetilde  w_\ve \cdot \n], \quad x\in\Gamma_\ve,\\
        &[z_\ve]=0, \quad x\in\Gamma_\ve,\\
        &z_\ve^e\big|_{\partial \Omega}=0.
    \end{split}
\end{align}
A solution of \eqref{eq:no-jump} is understood in the weak sense: Find $z_\ve \in H_0^1(\Omega)$ such that 
\begin{align*}
\int_{\Omega} \sigma_\ve \nabla z_\ve\cdot \nabla \phi\, dx = - \int_{\Omega\setminus \Gamma_\ve} \sigma_\ve \nabla \widetilde  w_\ve\cdot \nabla \phi\, dx
\end{align*}
holds for any $\phi \in H_0^1(\Omega)$.
Note that $z_\ve$ with zero jump $[z_\ve]\big|_{\Gamma_\ve}=0$ has a global gradient. Indeed, for $\phi\in C_c^\infty(\Omega)$,
\begin{align*}
\int_\Omega z_\ve \partial_k \phi\, \dd x
&= \int_{\Omega_\ve^c} z_\ve^c \partial_k \phi\, \dd x
+ \int_{\Omega_\ve^e} z_\ve^e \partial_k \phi\, \dd x
= - \int_{\Omega_\ve^c} \partial_k z_\ve^c \phi\, \dd x
- \int_{\Omega_\ve^e} \partial_k z_\ve^e \phi\, \dd x
+ \int_{\Gamma_\ve} [z_\ve] \n_k \phi\, \dd S_x\\
&=
-\int_\Omega \partial_k z_\ve \phi\, \dd x.
\end{align*}
The last problem has a unique solution by the Lax-Milgram lemma. Thus, changing back, we get a unique solution $w_\ve^\alpha=z_\ve^\alpha+\widetilde  w_\ve^\alpha$, $\alpha \in \{c, e\}$, of \eqref{eq:microscopic-prob_stat}. 
}

In this way we obtain a system of evolution equations on $\Gamma_\ve$:
\begin{align}
\label{eq:prob-on-Gamma_eps}
\begin{split}
    &\ve^{-1} c_m \partial_t [u_{\ve}] + \mathcal{L}_{\ve} [u_{\ve}] + \ve^{-1} S_m(X_{\ve})[u_{\ve}] = -\ve^{-1}\sigma^c\big(\frac{x}{\ve}\big) \nabla p_\ve \cdot \n, \quad [u_\ve](0,x)=V_\ve^{\rm in}(x) \quad \mbox{on}\ \Gamma_\ve,\\   
    &\partial_t X_\ve
    = f([u_\ve],X_\ve),\quad X_\ve(0,x)=X_\ve^{\rm in}(x), \,\,\mbox{on}\,\,\Gamma_\ve.
\end{split}
\end{align}
A similar reduction to an evolution equation on a manifold was used in \cite{lions1969quelques} (Chapter 1, Section 11) for proving an existence result for nonlinear monotone parabolic and hyperbolic problems.

For each fixed $\ve>0$, the existence and uniqueness of a solution of \eqref{eq:prob-on-Gamma_eps} is given by Lemma~\ref{lm:existence-v_w}. Note that \eqref{eq:prob-on-Gamma_eps} contains a quadratic term $S_m(X_\ve) [u_\ve]$ which is neither monotone nor Lipschitz.
To show the well-posedness, we verify the properties of the linear operator $\mathcal{L}_\ve$ used throughout the paper.
\begin{lemma}
\label{lm:L_eps}
For a fixed $\ve >0$, the operator $\mathcal{L}_\ve$ defined by \eqref{def:L_eps} is symmetric, non-negative, and satisfies the following properties: 
\begin{itemize}
\item[(i)] For any $\alpha > 0$, there exists $\Lambda_\ve >0$ depending on $\alpha$ and $\ve$ such that, for any $v \in H^{1/2}(\Gamma_\ve)$,
\[
\langle \mathcal{L}_\ve v , v\rangle_{H^{-1/2}(\Gamma_\ve), H^{1/2}(\Gamma_\ve)} + \alpha \|v\|_{L^2(\Gamma_\ve)}^2\ge \Lambda_\ve \|v\|^2_{H^{1/2}(\Gamma_\ve)}.
\]
\item[(ii)] 
There exists a constant $C_\ve$ such that
\[
|\langle \mathcal{L}_\ve v , \varphi\rangle_{H^{-1/2}(\Gamma_\ve), H^{1/2}(\Gamma_\ve)}|
\leq C_\ve \| v \|_{H^{1/2}(\Gamma_\ve)} \| \varphi\|_{H^{1/2}(\Gamma_\ve)},
\]
for any $v, \varphi \in H^{1/2}(\Gamma_\ve)$.
\item[(iii)] 
$
\|\mathcal{L}_\ve v\|_{H^{-1/2}(\Gamma_\ve)}
\le C_\ve\| v \|_{H^{1/2}(\Gamma_\ve)}$, for any $v \in H^{1/2}(\Gamma_\ve)$.
\end{itemize}
\end{lemma}

By the definition of $f$, using that $\max(a,b)=\frac{1}{2}(a+b) + \frac{1}{2}|a-b|$, we get
\begin{align*}
    f(v,X) = \overline{\tau}_1 (\beta(v)-X) + \overline{\tau}_2|\beta(v)-X|, \quad \overline{\tau}_1 = \frac{1}{2}\Big(\frac{1}{\tau_{ep}} + \frac{1}{\tau_{res}}\Big), \, \overline{\tau}_2 = \frac{1}{2}\Big(\frac{1}{\tau_{ep}} - \frac{1}{\tau_{res}}\Big),
\end{align*}
and therefore
\begin{align*}
    |f(v_1,X) - f(v_2,X)|
    \le |\overline{\tau}_1 (\beta(v_1)-\beta(v_2))|
    + \Big|\overline{\tau}_2 |\beta(v_1)-X| - \overline{\tau}_2 |\beta(v_2)-X| \Big|.
\end{align*}
The Lipschitz continuity of $\beta$ and the reverse triangle inequality $|a|-|b| \le |a-b|$ yield the Lipschitz continuity of $f$ with respect to $v$ for a fixed $X$. Similarly we obtain the Lipschitz continuity in $X$ for a fixed $v$. Thus, $f(\cdot, \cdot)$ is one-sided Lipschitz:
\begin{align}
\label{eq:Lip-f}
(f(v_1,X_1) - f(v_2,X_2)\, , X_1-X_2)_{L^2(\Gamma_\ve)} 
\le C(\|v_1-v_2\|_{L^2(\Gamma_\ve)}^2 + \|X_1-X_2\|_{L^2(\Gamma_\ve)}^2).
\end{align}
In this way, the assumptions of Lemma \ref{lm:existence-v_w} are satisfied, and Theorem~\ref{th:existence-eps-prob} is proved.
\end{proof}


\section{Formal asymptotic expansions}
\label{sec:expansion}
The goal of this section is to derive a formal asymptotic expansion for the original microscopic problem \eqref{eq:microscopic-prob} and obtain a coupled problem for the first two terms of the asymptotics \eqref{eq:coupled-gamma-1b}. We derive also a nonlinear equation for the leading term of the asymptotics with memory effects \eqref{eq:u_(-1)-gamma-1}.  

As the formal asymptotic expansions require differentiability, we assume in this section that $X_\ve$ satisfies the equation $\partial_t X_\ve = f([u_\ve], X_\ve)$ with a smooth function $f$. 
Assume that the initial functions can be expanded as 
\begin{align*}
    V_\ve^{\rm in}&= {V^{\rm in}(x,\frac{x}{\ve})}+ \ve V_1(x, \frac{x}{\ve}) + \cdots,\quad
     X_\ve^{\rm in}= X^{\rm in}(x, \frac{x}{\ve}) + \ve X_1^{in}(x,\frac{x}{\ve}) + \cdots.
\end{align*}
We postulate the following ansatz for the solution: 
\begin{align*} 
u_\ve^\alpha(t,x)  &\sim 
\ve^{-1}u_{-1}^\alpha(t,x,\frac{x}{\ve})+ \ve^0 u_0^\alpha (t,x,\frac{x}{\ve}) + \ve u_1^\alpha (t,x,\frac{x}{\ve})+ \cdots , \quad \alpha\in \{e, c\},\\
X_\ve(t, x) &\sim X_0(t,x,\frac{x}{\ve}) + \ve X_1(t,x,\frac{x}{\ve}) + \cdots, \\
p_\ve(t,x) & \sim p_0(t,x) + \ve p_1(t,x,\frac{x}{\ve}) + \cdots,
\end{align*}
where the functions $u_k^e(x,\cdot)$ are $Y$-periodic. {The functions $u_k^c$ are not assumed to be periodic since the cells are disconnected.} 
It is known that the corrector term can be written as $p_1(t,x,y)=\mathcal{M}(y)\cdot \nabla_x p_0$, and  the limit function $p_0$ solves the homogenized problem 
\begin{align}
\label{eq:p_0}
&\di{A \nabla p_0} = 0, \quad x\in\Omega,  \\
&p_0 = g, \quad x\in \partial \Omega,\nonumber
\end{align}
with $A = \int_Y \sigma(\nabla \mathcal{M} +\text{Id})\, \dd y$, and $\mathcal{M}_j(y)$, $j=1, 2, 3$, solving standard cell problems
\begin{align}
\label{eq:cell-M}
&\di{\sigma (\nabla \mathcal{M}_j+e_j)} = 0, \quad y\in Y,  \\
&\mathcal{M}_j(y) \,\, \mbox{is}\,\, Y-\mbox{periodic}.\nonumber
\end{align}
$\mathcal{M}$ is defined up to an additive constant and is unique under a normalization condition; for example, it can be chosen to satisfy $\int_Y \mathcal{M}\, \dd y=0$.
Equating the terms of order $\ve^{-3}$ in the bulk and of order $\ve^{-2}$ on the interface, we obtain the following problem in $Y^c \cup Y^e$ for $u_{-1}(t,x,y)$:
\begin{align*}
-&\diy{\sigma \nabla_y u_{-1}}  =0, \, & y \in \,\, &Y^e\cup Y^c,  \nonumber\\
& [\sigma \nabla_y u_{-1}  \cdot \n] = 0, \, & y \in \,\, &\Gamma, \nonumber\\
-&\sigma^c \nabla_y u_{-1}^c \cdot \n=c_m \partial_t [u_{-1}] + S_m(X_0)[u_{-1}],\, & y \in \,\, &\Gamma, \\
& [u_{-1}](0,x,y)=0,\, & y \in \,\, &\Gamma, \nonumber\\
& u_{-1}^e(t, x,\cdot) \,\, \mbox{is} \, \, Y-\mbox{periodic}.\nonumber
\end{align*}
The solution $u_{-1}$ exists and is unique up to an additive constant in $y$. Multiplying by $u_{-1}$ and integrating by parts in $Y^c\cup Y^e$, we obtain
\begin{align*}
    \frac{c_m}{2}\partial_t \int_{\Gamma} |[u_{-1}]|^2\, \dd S_y
    + \int_{Y^c \cup Y^e} \sigma |\nabla_y u_{-1}|^2\, \dd y
    + \int_\Gamma S_m(X_0) |[u_{-1}]|^2\, \dd S_y=0,
\end{align*}
which together with the zero initial condition yields $[u_{-1}]=0$ for all $t$, and $\nabla_y u_{-1}=0$, and thus $u_{-1}^e=u_{-1}^c = u_{-1}(t,x)$. This conclusion will be made rigorous in Lemma \ref{lm:compactness}.

Equating the terms of order $\ve^{-2}$ in the bulk and of order $\ve^{-1}$ on the interface, we obtain the following problem in $Y^c \cup Y^e$ for $u_0(t,x,y)$:
\begin{align}
\nonumber
-&\diy{\sigma \nabla_y u_0}  =\diy{\sigma \nabla_x u_{-1}}, \, & y \in \,\, &Y^e\cup Y^c,  \\
\nonumber
&[\sigma \nabla_y u_0 \cdot \n] =- [\sigma \nabla_x u_{-1} \cdot \n], \, & y \in \,\, &\Gamma, \\
\nonumber
-&\sigma^c \nabla_y u_0^c \cdot \n
=c_m \partial_t [u_0] + S_m(X_0)[u_0] + \sigma^c \nabla_x u_{-1} \cdot \n &&\\
&\hspace{2.2cm} {\ + \sigma^c (\nabla_y \mathcal{M}_j + \mathbf{e}_j) \partial_{x_j} p_0 \cdot \n},\, & y \in \,\, &\Gamma, \nonumber\\
\label{eq:w_0-gamma-1b}
& \partial_t X_0 = f([u_0], X_0), \quad X_0(0,x,y)=X^{\rm in}(x,y),\, & y \in \,\, &\Gamma, \\
& [u_0](0,x,y)=V^{\rm in}(x,y),\, & y \in \,\, &\Gamma, \nonumber\\
& u_0^e(t,x,\cdot) \,\, \mbox{is} \, \, Y-\mbox{periodic}.\nonumber
\end{align}
Similar to the original problem, the solution $u_0$ exists and is unique, up to an additive constant in $y$. To ensure uniqueness, one can normalize by $\int_Y u_0\, \dd y=0$. 

Because of the nonlinearity, we cannot decouple this problem to get a cell problem independent of the unknown $u_{-1}$, as in the classical two-scale limit or as in \cite{amar2004evolution}. Namely, the time-dependent cell problem \eqref{eq:chi^0-gamma-1} is coupled with the equation for $X_0$ that depends in its turn on $u_0$. Below we present one option to write $u_0$ in terms of $u_{-1}$. 
We seek $u_0$ in the form:
\begin{align}
    \label{eq:u_0}
     u_0(t,x,y) &= {\mathcal{M}(y) \cdot \nabla_x u_{-1}(t,x)} + z(t,x,y)+ \int_0^t \nabla (u_{-1}+p_0)(\tau,x) \cdot \chi^0(t,\tau,x,y) \, \dd \tau.
\end{align}
In \eqref{eq:u_0}, the first term will compensate for the term containing $\nabla_x u_{-1}$ in the first equation of \eqref{eq:w_0-gamma-1b}; the second term will ensure that the initial condition is satisfied; the third term will compensate for the contribution of the first term in the right-hand side of \eqref{eq:w_0-gamma-1b} on $\Gamma$.
We choose $\mathcal{M}(y)$ to satisfy the cell problem \eqref{eq:cell-M}.
The function $z(t,x,y)$, for a given initial condition $V^{\rm in}(x,y)$, solves the problem below
\begin{align}
&\diy{\sigma \nabla_y z}  = 0, \, & y \in \,\, &Y^c\cup Y^e,  \nonumber\\
&[\sigma \nabla_y z \cdot \n] = 0, & y \in \,\, &\Gamma, \nonumber\\
&-\sigma^c \nabla_y z^c \cdot \n = c_m \partial_t [z] + S_m(X_0)[z], \, & y \in \,\, &\Gamma, \label{eq:T-eps^(-1)-g/eps}\\
&[z](0,x, y) = V^{\rm in}(x,y) \, & y \in \,\, &\Gamma, \nonumber\\
& z \,\, \mbox{is} \, \, Y-\mbox{periodic} \,\, \mbox{and} \, \int_Y z = 0.\nonumber
\end{align}


For each $\tau \in [0,T]$ the function $\chi_j^0(t,\tau,x,y)$ satisfies for $t \in [\tau,T]$ the evolution problem on the periodicity cell
\begin{align}
-&\diy{\sigma  \nabla_y \chi_j^0(t,\tau,x,y) }  = 0, \, & y \in \,\, &Y^c\cup Y^e,  \nonumber\\
& [\sigma  \nabla_y \chi_j^0(t,\tau,x,y)\cdot \n]  = 0  , \, & y \in \,\, &\Gamma, \nonumber\\
-&\sigma^c \nabla_y \chi_j^0(t,\tau,x,y) \cdot \n = c_m \partial_t[\chi_j^0](t,\tau,x,y) + S_m(X_0(t,x,y))[\chi_j^0](t,\tau,x,y), \, & y \in \,\, &\Gamma, \label{eq:chi^0-gamma-1}\\
& c_m[\chi_j^0](\tau,\tau,x,y) =- \sigma^c ( \nabla_y \mathcal{M}_j(y) + e_j) \cdot \n.\, & y \in \,\, &\Gamma,\nonumber \\
&\chi_j^0 \,\, \mbox{is} \, \, Y-\mbox{periodic}. \nonumber
\end{align}
where for $t \in [0,T]$
\begin{align}
\label{eq:X_0-gamma-1}
& \partial_t X_0(t,x,y) = f([u_0](t,x,y), X_0(t,x,y)), \quad X_0(0,x,y)=X^{\rm in}(x, y),\, & y \in \,\, &\Gamma, \\
\label{eq:[w0]-gamma-1}
& [u_0](t,x,y)=[z](t,x,y) + \int_0^t \nabla (u_{-1}+p_0)(\tau, x) \cdot [\chi^0](t,\tau,x,y) \, \dd \tau,\, & y \in \,\, &\Gamma.
\end{align}

Note that the dependence of $\chi_j^0$ on $\tau$ is only through the solution interval, so that the initial condition is applied at time $t=\tau$ for each $\tau$. The problem for $\chi_j^0$ is thus solved in a triangular domain in the $(t,\tau)$-plane with evolution along the $t$-axis for each $\tau$ (see Figure~\ref{fig:2}). The cell problem \eqref{eq:chi^0-gamma-1} cannot be decoupled in the same way as it is done in \cite{amar2004evolution} due to the time-dependent factor $S_m(X_0)$ and the (possibly) nontrivial initial condition $V^{\rm in}$. As a result, the dependence of $\chi_j^0(t, \tau, x, y)$ on $(t,\tau)$ is not simply through the difference $(t-\tau)$, and the time integral \eqref{eq:u_0} is not of convolution type. 







\begin{figure}
\centering
    \def\svgwidth{0.25\textwidth}
\begingroup%
  \makeatletter%
  \providecommand\color[2][]{%
    \errmessage{(Inkscape) Color is used for the text in Inkscape, but the package 'color.sty' is not loaded}%
    \renewcommand\color[2][]{}%
  }%
  \providecommand\transparent[1]{%
    \errmessage{(Inkscape) Transparency is used (non-zero) for the text in Inkscape, but the package 'transparent.sty' is not loaded}%
    \renewcommand\transparent[1]{}%
  }%
  \providecommand\rotatebox[2]{#2}%
  \newcommand*\fsize{\dimexpr\f@size pt\relax}%
  \newcommand*\lineheight[1]{\fontsize{\fsize}{#1\fsize}\selectfont}%
  \ifx\svgwidth\undefined%
    \setlength{\unitlength}{250.97483718bp}%
    \ifx\svgscale\undefined%
      \relax%
    \else%
      \setlength{\unitlength}{\unitlength * \real{\svgscale}}%
    \fi%
  \else%
    \setlength{\unitlength}{\svgwidth}%
  \fi%
  \global\let\svgwidth\undefined%
  \global\let\svgscale\undefined%
  \makeatother%
  \begin{picture}(1,0.77431814)%
    \lineheight{1}%
    \setlength\tabcolsep{0pt}%
    \put(0,0){\includegraphics[width=\unitlength,page=1]{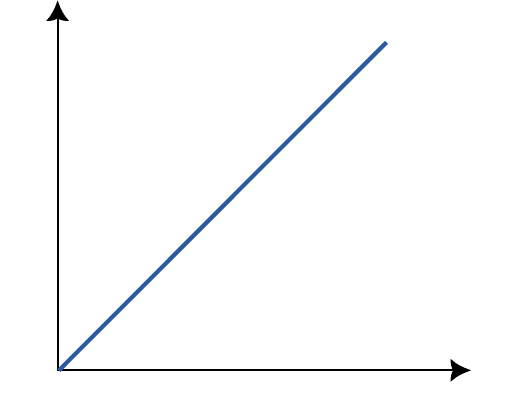}}%
    \put(0.81682714,0.00983273){\color[rgb]{0,0,0}\makebox(0,0)[lt]{\lineheight{1.25}\smash{\begin{tabular}[t]{l}$\tau$\end{tabular}}}}%
    \put(0.02566168,0.72006052){\color[rgb]{0,0,0}\makebox(0,0)[lt]{\lineheight{1.25}\smash{\begin{tabular}[t]{l}$t$\end{tabular}}}}%
    \put(0.41753703,0.31544791){\color[rgb]{0,0,0}\makebox(0,0)[lt]{\lineheight{1.25}\smash{\begin{tabular}[t]{l}$t=\tau$\end{tabular}}}}%
    \put(0,0){\includegraphics[width=\unitlength,page=2]{triangle.pdf}}%
    \put(0.14146102,0.7210576){\color[rgb]{0,0,0}\makebox(0,0)[lt]{\lineheight{1.25}\smash{\begin{tabular}[t]{l}$T$\end{tabular}}}}%
    \put(0,0){\includegraphics[width=\unitlength,page=3]{triangle.pdf}}%
  \end{picture}%
\endgroup%

\caption{Computational domain in the $(\tau, t)$-plane.}
\label{fig:2}
\end{figure}

Next, we equate the terms of order $\ve^{-1}$ in the bulk and $\ve^0$ on $\Gamma$ to derive the problem for $u_1(t,x,y)$:
\begin{align*}
&-\diy{\sigma \nabla_y u_1}  = \diy{\sigma \nabla_x u_0} + \dix{\sigma (\nabla_y u_0+\nabla_x u_{-1})}, \, & y \in \,\, &Y^c\cup Y^e,  \\
& [\sigma \nabla_y u_1 \cdot \n] =  - [\sigma \nabla_x u_0 \cdot \n], \, & y \in \,\, &\Gamma, \\
&- \sigma^c \nabla_y u_1^c \cdot \n = c_m \partial_t [u_1] + S_m(X_0)[u_1] + S_1 X_1 [u_0] + \sigma^c \nabla_x u_0^c \cdot \n \\
&\hspace{7cm} + {\sigma^c\nabla_x p_1 \cdot \n + \sigma^c \nabla_y p_2 \cdot \n}, \, & y \in \,\, &\Gamma, \\
& \partial_t X_1 = \partial_v f([u_0], X_0)[u_1] + \partial_X f([u_0], X_0)\, X_1, \quad X_1(0,x,y)= X_1^{in}(x,y), \, & y \in \,\, &\Gamma \\
& [u_1](0,x,y)= V_1(x,y), \, & y \in \,\, &\Gamma, \\
& u_1^e \,\, \mbox{is} \, \, Y-\mbox{periodic}.
\end{align*}
Integrating both sides of the equation for $u_1$ over $Y^c\cup Y^e$ we obtain 
\[
{\rm div}_x {\int_{Y^c\cup Y^e} \sigma(\nabla_y u_0 + \nabla_x u_{-1})\, dy}=0, \quad \mbox{in} \, (0,T]\times\Omega.
\]
Recalling the factorization $p_1=\mathcal{M}(y)\cdot \nabla_x p_0$, we see that the triple $(u_{-1}, u_0, X_0)$ satisfies
the following coupled system:
\begin{align}
-&{\rm div}_x {\int_{Y^c\cup Y^e} \sigma(\nabla_y u_0 + \nabla_x u_{-1})\, dy}=0, \, & x \in \,\, &\Omega, \nonumber\\
-&\diy{\sigma (\nabla_y u_0 + \nabla_x u_{-1})}  =0, \, & y \in \,\, &Y^e\cup Y^c,  \nonumber\\
&[\sigma (\nabla_y u_0+\nabla_x u_{-1}) \cdot \n] =0, \, & y \in \,\, &\Gamma, \nonumber \\
\label{eq:coupled-gamma-1b}
-&\sigma^c \nabla_y u_0^c \cdot \n
=c_m \partial_t [u_0] + S_m(X_0)[u_0] + \sigma^c\nabla_x u_{-1}\cdot \n &&\\
&\hspace{4.5cm} + {\sigma^c (\nabla_y \mathcal{M}_j + \mathbf{e}_j)\cdot \n \, \partial_{x_j} p_0},\, & y \in \,\, &\Gamma, \nonumber\\
& \partial_t X_0 = f([u_0], X_0), \quad X_0(0,x,y)=X^{\rm in}(x, y),\, & y \in \,\, &\Gamma, \nonumber\\
& [u_0](0,x,y)=V^{\rm in}(x,y),\, & y \in \,\, &\Gamma, \nonumber\\
& u_{-1}\big|_{\partial \Omega}=0, \quad u_0^e(x,\cdot) \,\, \mbox{is} \, \, Y-\mbox{periodic}.\nonumber
\end{align}
The coupled effective equations \eqref{eq:coupled-gamma-1b} cannot be decoupled in general when $f$ is nonlinear. Indeed, while decoupling is seen for simple problems, the homogenized equations in general can be nonexplicit or contain integro-differential operators. 
Taking into account the representation \eqref{eq:u_0}, \eqref{eq:coupled-gamma-1b} becomes 
\begin{align}
\label{eq:u_(-1)-gamma-1}
\begin{split}
-&\dix{A\nabla(u_{-1}+p_0) + \int_0^t B(t,\tau) \nabla(u_{-1}+p_0)(\tau)\, d\tau}={\rm div}\, C(t, x), \quad x\in \Omega,\\
&u_{-1}=0, \quad x\in \partial \Omega,
\end{split}
\end{align}
where 
\begin{align}
    \label{def:A}
    A_{ij}&= \int_{Y^c\cup Y^e} \sigma(y)(\partial_{y_i}\mathcal{M}_j(y) + \delta_{ij})\, dy,\\
    \label{def:B}
    B_{ij}(t,\tau, x)&=\int_{Y^c\cup Y^e} \sigma(y) \partial_{y_i} \chi_j^0(t,\tau,x,y)\, \dd y,\\
    \label{def:C}
    C(t,x)&= \int_{Y^c\cup Y^e} \sigma(y) \nabla_y z(t,x,y)\, dy,
\end{align}
with $\mathcal{M}_j$ solving \eqref{eq:cell-M}, $\chi^0$ solving \eqref{eq:chi^0-gamma-1}, $z$ solving \eqref{eq:T-eps^(-1)-g/eps}, and $p_0$ being the solution of \eqref{eq:p_0}. Note that $B, C$ are defined in terms of solutions of auxiliary cell problems coupled with the equation for $X_0$ which is nonlinear in $[u_0]$, and $[u_0]$ depends on $\nabla u_{-1}$. 

It is classical that $A$ defined by \eqref{def:A} is symmetric and positive definite.

Unlike the setting considered by Amar et al. \cite{amar2004evolution}, the membrane coefficient in the present cell problem is time dependent through \(S_m(X_0(t,x,y))\). Consequently, the associated family of cell operators is non-autonomous, and its evolution operators are not self-adjoint in general; therefore, the memory kernel \(B(t,\tau,x)\) is not expected to be symmetric. Symmetry can only be recovered under additional assumptions, for instance when \(X_0\), and hence \(S_m(X_0)\), is independent of time.

Making the change of unknowns $w_{-1}=u_{-1}+p_0$, $w_0=u_0 + \mathcal{M}_j \partial_{x_j} p_0$, and taking into account \eqref{eq:p_0} and \eqref{eq:cell-M}, \eqref{eq:coupled-gamma-1b} transforms into a limit problem for the original potential $w_\ve$: 
\begin{align}
-&{\rm div}_x {\int_{Y^c\cup Y^e} \sigma(\nabla_y w_0 + \nabla_x w_{-1})\, dy}=0, \, & x \in \,\, &\Omega, \nonumber\\
-&\diy{\sigma (\nabla_y w_0 + \nabla_x w_{-1})}  =0, \, & y \in \,\, &Y^e\cup Y^c,  \nonumber\\
&[\sigma (\nabla_y w_0+\nabla_x w_{-1}) \cdot \n] =0, \, & y \in \,\, &\Gamma, \nonumber \\
\label{eq:coupled-u_eff}
-&\sigma^c \nabla_y w_0^c \cdot \n
=c_m \partial_t [w_0] + S_m(X_0)[w_0] + \sigma^c\nabla_x w_{-1}\cdot \n,\, & y \in \,\, &\Gamma,\\
& \partial_t X_0 = f([w_0], X_0), \quad X_0(0,x,y)=X^{\rm in}(x, y),\, & y \in \,\, &\Gamma, \nonumber\\
& [w_0](0,x,y)=V^{\rm in}(x,y),\, & y \in \,\, &\Gamma, \nonumber\\
& w_{-1}\big|_{\partial \Omega}=g, \quad w_0^e(x,\cdot) \,\, \mbox{is} \, \, Y-\mbox{periodic}.\nonumber
\end{align}

\section{Homogenization}
\label{sec:homogenization-main}
In this section we justify the formal asymptotics in Section \ref{sec:expansion}. We start by deriving a priori estimates for $(u_\ve, X_\ve)$ uniform in $\ve$ (Lemma \ref{lm:apriori-est}).
Then we give definitions of the two-scale convergence in $\Omega$ and on the periodic oscillating surface $\Gamma_\ve$ (see Definitions \ref{def:2-scale}, \ref{def:2scale-Gamma_eps}, and \ref{def:2scale-Gamma_eps-pointwise}). In Lemma \ref{lm:compactness} we formulate the compactness result. After that we proceed with the passage to the limit in the weak formulation of \eqref{eq:weak-v_eps_w_eps}. Note that, by the a priori estimates, we can pass to the limit, but not identify the limits $\overline{S}$ and $\overline{f}$ of the nonlinear functions $S_m(X_\ve)[u_\ve]$ and $f([u_\ve], X_\ve)$, respectively. To do so, we impose an extra assumption on the membrane conductivity term, introducing a truncation \eqref{eq:S(v,X)}. 
Finally, we prove the well-posedness of the macroscopic problem \eqref{eq:effective-problem-coupled-strong}. 

\subsection{A priori estimates} \label{sec: a_priori}
The main result of this section is given in the following lemma.
\begin{lemma}[A priori estimates]
\label{lm:apriori-est} Let $(u_{\ve}, [u_\ve], X_\ve)$ solve \eqref{eq:microscopic-prob} and assume that \ref{H1}--\ref{H4} are satisfied. Then there exists a constant $C$ independent of $\ve$ such that the following estimates hold:
\begin{enumerate}[label=(\roman*)]
    \item \label{lem: ape1} 
    $\displaystyle \sup_{t\in[0,T]} \Big(\ve\int_{\Gamma_\ve} [u_{\ve}]^2\, dS
    + \ve\int_{\Gamma_\ve} |X_{\ve}|^2\, dS\Big)\\
    \leq C \Big(1 + \ve\|V_\ve^{\rm in}\|_{L^2(\Gamma_\ve)}^2 +\ve \| X_\ve^{\rm in} \|_{L^2(\Gamma_{\ve})}^2 + \sup_{0\le t\le T}\| g \|_{H^{1/2}(\partial \Omega)}^2\Big)$.
    \item \label{lem: ape2} $\displaystyle \sup_{t\in[0,T]} \ve^2 \int_{\Omega_\ve^c \cup \Omega_\ve^e} (|u_\ve|^2 + | \nabla u_{\ve} |^2) \, \dd x\\   
    \leq C \Big( \ve \|V_\ve^{\rm in}\|_{L^2(\Gamma_\ve)}^2 
    + {\ve^2 \|\nabla V_\ve^{\rm in}\|_{L^2(\Omega_\ve^c)}^2}+\sup_{0\le t\le T}\| g \|_{H^{1/2}(\partial \Omega)}^2 + \| \partial_t g \|_{L^2(0,T; H^{1/2}(\partial \Omega))}^2\Big)$.
    \item \label{lem: ape3} 
    $
    \displaystyle
   \ve \int_0^T \int_{\Gamma_{\ve}} |\partial_t [u_{\ve}]|^2 \, \dd S \, \dd \tau 
   + \ve \int_0^T \int_{\Gamma_{\ve}} |\partial_t X_{\ve}|^2 \, \dd S \, \dd \tau\\
   \leq C \Big( \ve |\Gamma_\ve|+ \ve \| V_\ve^{\rm in} \|_{L^2(\Gamma_{\ve})}^2
   + {\ve^2 \|\nabla V_\ve^{\rm in}\|_{L^2(\Omega_\ve^c)}^2}+ \ve \| X_\ve^{\rm in} \|_{L^2(\Gamma_{\ve})}^2+ \sup_{0\le t\le T}\| g \|_{H^{1/2}(\partial \Omega)}^2 + \| \partial_t g \|_{L^2(0,T; H^{1/2}(\partial \Omega))}^2 \Big).$
\end{enumerate}
\end{lemma}
\begin{proof}
We {use $u_\ve$ as a test function in} the weak formulation \eqref{eq:weak-v_eps_w_eps}:
\begin{align*}
   &\frac{\ve^{-1}}{2} \partial_t\int_{\Gamma_{\ve}} [u_\ve]^2 \, \dd S + \frac{\ve^{-1}}{c_m} \int_{\Gamma_{\ve}} S_m(X_{\ve}) [u_\ve]^2 \, \dd S + \frac{1}{c_m} \int_{\Omega_\ve^c \cup \Omega_\ve^e} \sigma_{\ve} |\nabla u_\ve|^2\, \dd x\\
   & 
   = - \frac{\ve^{-1}}{c_m}\int_{\Omega_\ve^c \cup \Omega_\ve^e} \sigma_\ve \nabla p_\ve \cdot \nabla u_\ve\, dx.
\end{align*}
By the definition of the membrane conductivity $S_m$ \eqref{def:S_m_f}, we can bound the second term on the left-hand side of the last identity from below:
\begin{align*}
    \int_{\Gamma_{\ve}} S_m(X_\ve) [u_\ve]^2 \, \dd S \geq S_0 \int_{\Gamma_{\ve}} [u_\ve]^2 \, \dd S.
\end{align*}
Using the Young inequality with a small enough $\delta>0$ and a priori estimates for $p_\ve$ (see Corollary 8.7 in \cite{gilbarg2001elliptic}) 
\begin{align*}
    &\|p_\ve\|_{L^2(0,T; H^1(\Omega))} \le C\|g\|_{L^2(0,T; H^{1/2}(\partial \Omega))},
\end{align*}
for some $C$ independent of $\ve$, 
we get 
\begin{align*}
\begin{split}
   \ve^{-1} \partial_t\int_{\Gamma_{\ve}} [u_\ve]^2 \, \dd S +   \int_{\Omega_\ve^c \cup \Omega_\ve^e}  |\nabla u_\ve |^2 \, \dd x
   &\leq C \ve^{-2}\int_{\Omega_\ve^c \cup \Omega_\ve^e} | \nabla p_{\ve} |^2 \, \dd x \leq C \ve^{-2}\| g \|_{H^{1/2}(\partial \Omega)}^2.
   \end{split}
\end{align*}
Integrating in time yields
\begin{align*} 
    &\ve \int_{\Gamma_\ve} [u_\ve]^2\, \dd S
    + \ve^2\int_0^T \int_{\Omega_\ve^c \cup \Omega_\ve^e}  |\nabla u_\ve |^2 \, \dd x \dd t
    \leq C \big(\ve \| V_\ve^{\rm in}
    \|_{L^2(\Gamma_{\ve})}^2 + \| g \|_{L^2(0,T; H^{1/2}(\partial \Omega))}^2\big).
\end{align*} 
which proves \ref{lem: ape1}.

To estimate the $L^2(\Omega)$-norm of $u_\ve$ we use the Poincaré inequality in $\Omega_\ve^c \cup \Omega_\ve^e$ for functions $u_\ve \in L^2(0,T; H^1(\Omega_\ve^c \cup \Omega_\ve^e))$ such that $u_\ve\big|_{\partial \Omega}=0$ (see Lemma 3.3, \cite{hummel2000homogenization}):
\begin{align*}
\int_\Omega |u_\ve|^2\, \dd x \le
C_0 \Big(\int_{\Omega_\ve^c \cup \Omega_\ve^e}|\nabla u_\ve|^2\, \dd x + \frac{1}{\ve} \int_{\Gamma_\ve} [u_\ve]^2 \, \dd S\Big),
\end{align*}
for some constant $C_0>0$ independent of $\ve$.

In order to prove \ref{lem: ape2} and \ref{lem: ape3} we first perform a formal computation. Namely, we multiply \eqref{eq:microscopic-prob} by $\partial_t u_\ve$ and integrate by parts:
\begin{align*}
   &\ve^{-1} \int_{\Gamma_{\ve}} |\partial_t [u_\ve]|^2 \, \dd S + \frac{\ve^{-1}}{c_m} \int_{\Gamma_{\ve}} S_m(X_\ve) [u_\ve] \partial_t [u_\ve] \, \dd S + \frac{1}{2 c_m} \partial_t \int_{\Omega_\ve^c \cup \Omega_\ve^e} \sigma_{\ve} | \nabla u_\ve |^2  \, \dd x\\ 
   &
   = - \frac{\ve^{-1}}{c_m} \int_{\Gamma_{\ve}} (\sigma^c\big(\frac{x}{\ve}\big) \nabla p_{\ve}^c \cdot \n ) \partial_t [u_\ve] \, \dd S.
\end{align*}
Multiplying \eqref{eq:p_eps} by $\partial_t u_\ve$ and integrating by parts, we translate the surface integral in the right-hand side into the bulk one:
\begin{align*}
    &\ve^{-1} \int_{\Gamma_{\ve}} |\partial_t [u_\ve]|^2 \, \dd S + \frac{\ve^{-1}}{c_m} \int_{\Gamma_{\ve}} S_m(X_\ve) [u_\ve] \partial_t [u_\ve] \, \dd S + \frac{1}{2 c_m} \partial_t \int_{\Omega_\ve^c \cup \Omega_\ve^e} \sigma_{\ve} | \nabla
    u_\ve |^2  \, \dd x\\ 
   &= - \frac{ \ve^{-1}}{c_m}
    \partial_t \int_{\Omega_\ve^c \cup \Omega_\ve^e} \sigma_{\ve} \nabla  p_{\ve} \cdot \nabla u_\ve \, \dd x +
    \frac{\ve^{-1}}{c_m}
    \int_{\Omega_\ve^c \cup \Omega_\ve^e} \sigma_{\ve} \nabla \partial_t p_{\ve} \cdot \nabla u_\ve \, \dd x.
\end{align*}
Integrating in time we have 
\begin{align*}
    &\ve^{-1} \int_0^t \int_{\Gamma_{\ve}} |\partial_\tau [u_\ve]|^2 \, \dd S \dd \tau 
    + \frac{\ve^{-1}}{c_m} \int_0^t \int_{\Gamma_{\ve}} S_m(X_\ve) [u_\ve] \partial_\tau [u_\ve] \, \dd S \dd \tau 
    + \frac{1}{2 c_m} \int_{\Omega_\ve^c \cup \Omega_\ve^e} \sigma_{\ve} | \nabla u_\ve |^2  \, \dd x \big|_{\tau=0}^{\tau=t}\\ 
   &
   = 
       - \frac{\ve^{-1}}{c_m}\int_{\Omega_\ve^c \cup \Omega_\ve^e} \sigma_{\ve} \nabla  p_{\ve} \cdot \nabla u_\ve \, \dd x \big|_{\tau=0}^{\tau=t}+
    \frac{\ve^{-1}}{c_m}
    \int_0^t \int_{\Omega_\ve^c \cup \Omega_\ve^e} \sigma_{\ve} \nabla \partial_\tau p_{\ve} \cdot \nabla u_\ve \, \dd x \dd \tau.
\end{align*}
Applying the Young inequality with a parameter to estimate the integrals in the right-hand side of the last identity, together with 
\begin{align*}
    \ve^{-1} \Big|\int_{\Gamma_{\ve}} S_m(X_\ve) [u_\ve] \partial_t [u_\ve] \, \dd S\Big| 
    \leq \frac{S_0+S_1}{2}\Big( \ve^{-1} \delta \int_{\Gamma_{\ve}} (\partial_t[u_\ve])^2 \, \dd S + 
    \ve^{-1} \delta^{-1}
    \int_{\Gamma_{\ve}} [u_\ve]^2 \, \dd S\Big),
\end{align*}
and applying estimate (i), we have
\begin{align*}
\begin{split}
    &\ve^{-1} \int_0^t \int_{\Gamma_{\ve}} |\partial_t [u_\ve]|^2 \, \dd S \dd \tau 
    + \int_{\Omega_\ve^c \cup \Omega_\ve^e} | \nabla u_\ve |^2  \, \dd x
   \\
   &\le 
   C\, \ve^{-1} \|V_\ve^{\rm in}\|_{L^2(\Gamma_\ve)}^2+ 
   C \ve^{-2}
    \sup_{t\in [0,T]}\int_{\Omega_\ve^c \cup \Omega_\ve^e} |\nabla  p_{\ve}|^2\, \dd x 
    +
    C \ve^{-2}
    \int_0^t \int_{\Omega_\ve^c \cup \Omega_\ve^e} |\nabla \partial_t p_{\ve}|^2\, \dd x \dd \tau\\
    &+ \int_{\Omega_\ve^c \cup \Omega_\ve^e} |\nabla u_\ve|^2\, \dd x \big|_{t=0}
    + \int_0^t \int_{\Omega_\ve^c \cup \Omega_\ve^e} |\nabla u_\ve|^2\, \dd x \dd \tau.
\end{split}
\end{align*}
The norm of $\nabla u_\ve$ at $t=0$ is estimated in the following way. 

Let $u_{\ve}$ be the solution of the stationary problem \eqref{eq:microscopic-prob_stat} with the jump $[u_{\ve}]$:
\begin{align*}
&\di{\sigma_\ve \nabla u_\ve}  =0, \, &\mbox{in}\,\, &\Omega_\ve^c \cup \Omega_\ve^e,  \\
&[\sigma_\ve \nabla u_\ve \cdot \n]  = 0, \, &\mbox{on}\,\, &\Gamma_\ve, \\
&u_{\ve}^c - u_{\ve}^e = [u_{\ve}]
, \, &\mbox{on}\,\,&\Gamma_\ve,
\\
&u_\ve   = 0, \, &\mbox{on}\,\,& \partial \Omega. 
\end{align*}
To estimate $\nabla u_\ve$ at time $t=0$, we write equivalently
\begin{align*}
\langle \mathcal L_\ve [u_\ve], [u_\ve]\rangle = 
\min_{[\phi_\ve]=[u_\ve],\,\, \phi_\ve|_{\partial \Omega}=0} \int_{\Omega_\ve^c \cup \Omega_\ve^e} \sigma_\ve |\nabla \phi_\ve|^2\, \dd x.
\end{align*}
Choosing the test function $\phi_\ve$ such that $\phi_\ve=V_\ve^{\rm in}$ in $\Omega_\ve^c$ and $\phi_\ve=0$ in $\Omega_\ve^e$, in view of \ref{H3}, we obtain
\begin{align*}
    \ve^2\int_{\Omega_\ve^c \cup \Omega_\ve^e} \sigma_\ve |\nabla u_\ve|^2\, \dd x\big|_{t=0}
    \le 
    \ve^2\int_{\Omega_\ve^c} \sigma^c\big(\frac{x}{\ve}\big)|\nabla V_\ve^{\rm in}|^2\, \dd x \le C.
\end{align*}
Multiplying by $\ve^2$, using the estimates obtained above and the Grönwall inequality, we arrive at
\begin{align*}
    &\ve \int_0^t \int_{\Gamma_{\ve}} |\partial_t [u_\ve]|^2 \, \dd S\, \dd \tau
    + \ve^2 \int_{\Omega_\ve^c \cup \Omega_\ve^e} | \nabla u_\ve |^2  \, \dd x \\
    &\leq 
    C\big(\ve \| V_\ve^{\rm in}\|_{L^2(\Gamma_{\ve})}^2 + {\ve^2 \|\nabla V_\ve^{\rm in}\|_{L^2(\Omega_\ve^c)}^2} + \sup_{t\in[0,T]}\|g\|_{H^{1/2}(\partial \Omega)}^2 + 
    \|\partial_t g\|_{L^2(0,T; H^{1/2}(\partial \Omega))}^2\big).
\end{align*}
The argument above can be made rigorous by performing the corresponding estimates on the Galerkin approximations, as it is done in Lemma \ref{lm:est-fin-dim}. 

The functions $X_\ve, \partial_t X_\ve$ are estimated using \ref{H2}, \ref{H3}, and Lemma \ref{lm:bound-w}. 
Lemma \ref{lm:apriori-est} is proved.
\end{proof}

\subsection{Compactness}
\label{sec:compactness}
To pass to the limit as $\ve \to 0$ we will use the two-scale convergence  in each subdomain $\Omega_\ve^e, \Omega_\ve^c$ and on the surface $\Gamma_\ve$. Below we give the corresponding definitions. For the convergence of the initial data we will use the stationary versions, while for the convergence of solutions the time-dependent versions will be employed. Note that $u_\ve, X_\ve$ do not oscillate in time, as confirmed by uniform estimates for the time derivatives. 
\begin{definition}
\label{def:2scale-Gamma_eps-stationary}
    A sequence $v_\ve$ in $L^2(\Gamma_\ve)$ converges weakly two-scale to $v_0\in L^2(\Omega \times \Gamma)$ on $\Gamma_\ve$ if
    \begin{itemize}
        \item $\displaystyle \ve \int_{\Gamma_\ve} |v_\ve|^2\, \dd S <C$, where $C>0$ is independent of $\ve$.
        \item For any $\phi \in C(\overline{\Omega}; C_{\rm per}(Y))$,
        \begin{align*}
            \lim_{\ve \to 0} \ve \int_{\Gamma_{\ve}} v_{\ve} (x) \phi\Big(x,\frac{x}{\ve}\Big)  \, \dd S_x = \int_{\Omega} \int_{\Gamma} v_0 (x,y) \phi(x,y) \, \dd S_y \, \dd x.
        \end{align*}
    \end{itemize}
A sequence $v_\ve$ converges strongly two-scale to $v_0 \in L^2(\Omega \times \Gamma)$ if it converges weakly two-scale to $v_0$ and
\begin{align*}
\lim_{\ve \to 0} \ve \int_{\Gamma_\ve} |v_\ve|^2\, \dd S_x 
= \int_{\Omega} \int_{\Gamma} |v_0 (x,y)|^2 \, \dd S_y \, \dd x.
\end{align*}
\end{definition}
\begin{definition}
\label{def:2-scale}
    We say that $u_{\ve}(t,x)$ two-scale converges to the function $u_0(t,x,y) \in L^2(0,T;L^2(\Omega \times Y))$, as $\ve \to 0$ and write $u_{\ve} (t,x) \stackrel{2}{\rightharpoonup } u_0(t,x,y)$
    if the following hold
    \begin{enumerate}[label=(\roman*)]
        \item $\displaystyle\int_0^T \int_{\Omega} | u_{\ve} |^2 \, \dd x \, \dd t < C $, where $C>0$ is independent of $\ve$.
        \item For any $\phi(t,x) \in C([0,T]; L^2(\Omega))$ and $\psi(y) \in C_{\rm per}(Y)$, we have 
        \begin{align*}
            \lim_{\ve \to 0} \int_0^T \int_{\Omega} u_{\ve} (t,x) \phi(t,x) \psi\Big(\frac{x}{\ve}\Big)  \, \dd x \, \dd t = \int_0^T \int_{\Omega} \int_{Y} u_0 (t,x,y) \phi(t,x) \psi(y) \, \dd y \, \dd x \, \dd t.
        \end{align*}
    \end{enumerate}
\end{definition}

\begin{definition}
\label{def:2scale-Gamma_eps}
    We say that a sequence $v_{\ve}\in L^2(0,T;L^2(\Gamma_{\ve}))$ converges two-scale on $\Gamma_{\ve}$, to the function $v_0(t,x,y) \in L^2(0,T;L^2(\Omega \times \Gamma))$, as $\ve \to 0$,
    if the following hold
    \begin{enumerate}[label=(\roman*)]
        \item $\displaystyle \ve \int_0^T \int_{\Gamma_{\ve}}  |v_{\ve}|^2 \, \dd S \, \dd t < C $.
        \item For any $\phi(t,x) \in C([0,T]; C(\overline{\Omega}))$ and $\psi(y) \in C_{\rm per}(Y)$, we have 
        \begin{align*}
            \lim_{\ve \to 0} \ve \int_0^T \int_{\Gamma_{\ve}} v_{\ve} (t,x) \phi(t,x) \psi\Big(\frac{x}{\ve}\Big)  \, \dd S_x \, \dd t = \int_0^T \int_{\Omega} \int_{\Gamma} v_0 (t,x,y) \phi(t,x) \psi(y) \, \dd S_y \, \dd x \, \dd t.
        \end{align*}
    \end{enumerate}
\end{definition}
\begin{definition}
\label{def:2scale-Gamma_eps-pointwise}
A sequence $v_\ve(t,x)$ converges $t$-pointwise two-scale to some $v_0 \in L^2(0,T;L^2(\Omega \times \Gamma))$ on $\Gamma_{\ve}$ if 
\begin{enumerate}[label=(\roman*)]
        \item $\displaystyle \sup_{t\in[0,T]}\ve \int_{\Gamma_{\ve}} |v_{\ve}|^2 \, \dd S_x < C $, and 
        \item for almost all $t \in [0,T]$,
        \begin{align*}
             \lim_{\ve \to 0} \ve \int_{\Gamma_{\ve}} v_{\ve} (t,x) \phi(x) \psi\Big(\frac{x}{\ve}\Big)  \, \dd S_x = \int_{\Omega} \int_{\Gamma} v_0 (t,x,y) \phi(x) \psi(y) \, \dd S_y \, \dd x,
        \end{align*}
         for all $\phi(x) \in C(\overline{\Omega}), \, \psi(y) \in C_{\rm per}(Y)$
        \end{enumerate}
\end{definition}
Having derived a priori estimates in Lemma \ref{lm:apriori-est}, we prove the following compactness result.
\begin{lemma}[Compactness]\label{lm:compactness}
Denote $\mathbf{1}_{Y^\alpha}(y)$ the indicator function of $Y^\alpha$, $\alpha= e, c$. Then there exist $u_{-1}^\alpha \in L^2(0,T; H^1(\Omega))$, a $Y$-periodic in $y$ function $u_0^e(t,x,y)\in L^2((0,T)\times \Omega;  H_{\mathrm{per}}^1(Y^e))$, $u_0^c(t,x,y)\in L^2((0,T)\times \Omega; H^1(Y^c))$, as well as $\overline{S}, \overline{f}\in L^2(0,T; L^2(\Omega \times \Gamma))$
such that, along a subsequence, the following convergence holds as $\ve \to 0$:
\begin{itemize}
\item[(a)] Extra- and intracellular potential:\\ 
$\ve \mathbf{1}_{\Omega_\varepsilon^\alpha}(x) u_\ve(t,x)  \, \stackrel{2}{\rightharpoonup }\, \mathbf{1}_{Y^\alpha}(y) u_{-1}^\alpha(t,x)$ weakly two-scale in $L^2(0,T; L^2(\Omega))$.\\
$\ve \mathbf{1}_{\Omega_\varepsilon^\alpha}(x) \nabla u_\ve(t,x) \, \stackrel{2}{\rightharpoonup }\, \mathbf{1}_{Y^\alpha}(y)\Big(\nabla_x u_{-1}^\alpha(t,x) + \nabla_y u_0^\alpha(t,x,y)\Big)$ weakly two-scale in $L^2(0,T; L^2(\Omega))^3$, where $u_0^e(t,x,\cdot)$ is $Y$-periodic.

\item[(b)] Convergence of the jump $[u_\ve]$ on $\Gamma_\ve$:\\
$[u_\ve]\big|_{\Gamma_\ve} \, \stackrel{2}{\rightharpoonup }\, (u_0^c - u_0^e)\big|_{\Gamma}\in C([0,T]; L^2(\Omega\times \Gamma))$ $t$-pointwise two-scale on $\Gamma_\ve$,\\
$\partial_t [u_\ve]\big|_{\Gamma_\ve} \, \stackrel{2}{\rightharpoonup }\, \partial_t (u_0^c - u_0^e)\big|_{\Gamma}\in L^2(0,T; L^2(\Omega \times \Gamma))$ weakly two-scale on $\Gamma_\ve$.

\item[(c)] $u_{-1}^c = u_{-1}^e=u_{-1}(t,x)$ in $L^2((0,T)\times \Omega)$.

\item[(d)]
Convergence of $X_\ve$ on $\Gamma_\ve$:\\
$X_\ve \, \stackrel{2}{\rightharpoonup }\, X_0 \in C([0,T]; L^2(\Omega\times \Gamma))$ $t$-pointwise
two-scale on $\Gamma_\ve$,\\
$\partial_t X_\ve \, \stackrel{2}{\rightharpoonup }\, \partial_t X_0\in L^2(0,T; L^2(\Omega \times \Gamma))$ weakly  two-scale on $\Gamma_\ve$.

\item[(e)] Convergence of nonlinear functions:\\
$S_m(X_\ve)[u_\ve] \, \stackrel{2}{\rightharpoonup }\, \overline{S}(t,x,y)\in L^2(0,T; L^2(\Omega \times \Gamma))$ weakly 
two-scale on $\Gamma_\ve$.\\
$f([u_\ve], X_\ve) \, \stackrel{2}{\rightharpoonup }\, \overline{f}(t,x,y)\in L^2(0,T; L^2(\Omega \times \Gamma))$ weakly 
two-scale on $\Gamma_\ve$.
\end{itemize}
\end{lemma}

\begin{proof}
(a) Due to (ii) in Lemma~\ref{lm:apriori-est}, $\ve u_\ve^\alpha$ and $\ve \nabla u_\ve^\alpha$ are uniformly bounded in $\ve$ in $L^2((0, T) \times \Omega_\ve^\alpha)$. Thus, there exist $u_{-1}^e(t,x,y) \in L^2((0,T)\times \Omega; L_{\mathrm{per}}^2(Y))$, $u_{-1}^c(t,x,y) \in L^2((0,T)\times \Omega; L^2(Y))$,  $q^e(t,x,y)\in L^2((0,T)\times \Omega; L_{\mathrm{per}}^2(Y))^3$, and $q^c(t,x,y)\in L^2((0,T)\times \Omega; L^2(Y))^3$ such that $\ve \mathbf{1}_{\Omega_\varepsilon^\alpha} u_\ve$ converges two-scale to $\mathbf{1}_{Y^\alpha}(y) u_{-1}^\alpha(t,x,y)$ and $\ve \mathbf{1}_{\Omega_\varepsilon^\alpha} \nabla u_\ve$ to $\mathbf{1}_{Y^\alpha}(y) q^\alpha(t,x,y)$ in the sense of Definition \ref{def:2-scale}. Take a test function $\phi(t,x) \psi\big(\frac{x}{\ve}\big)$ such that $\phi \in C([0,T]; C_0^\infty(\Omega))$, $\psi \in H_{\mathrm{per}}^1(Y)^3$, and $\psi\big|_{\Gamma}  = 0$. We integrate by parts and obtain
    \begin{align*}
        &\ve^2 \int_0^T \int_{\Omega_\ve^\alpha} \nabla u_\ve^\alpha(t,x) \phi(t,x) \cdot \psi\big(\frac{x}{\ve}\big) \, \dd x \, \dd t \\
        & = - \ve^2 \int_0^T \int_{\Omega_\ve^\alpha} u_\ve^\alpha(t,x) \nabla \phi(t,x)\cdot \psi\big(\frac{x}{\ve}\big) \, \dd x \, \dd t - \ve \int_0^T \int_{\Omega_\ve^\alpha} u_\ve^\alpha(t,x) \phi(t,x) \diy{ \psi\big(\frac{x}{\ve}\big)} \, \dd x \, \dd t, \,\, \alpha=c, e.
    \end{align*}
    We pass to the limit as $\ve \rightarrow 0$ and obtain
    \begin{align*}
        &\frac{1}{| Y |} \int_0^T \int_{\Omega} \int_{Y^\alpha} u_{-1}^\alpha(t,x,y) \phi(t,x) \diy{\psi(y)} \, \dd y \, \dd x \, \dd t = 0.
    \end{align*}
This implies that $\nabla_y u_{-1}^\alpha = 0$, and thus $u_{-1}^\alpha$ is independent of $y$ in each connected component of $Y^\alpha$. Since $Y^e$ is connected, $u_{-1}^e(t,x,y)=u_{-1}^e(t,x)$ in $Y^e$.
On the other hand, on each connected component $Y_j^c$, we
obtain functions $c_j=c_j(t,x)$ such that
\[
u_{-1}^c(t,x,y)=c_j(t,x)
\qquad\text{for }y\in Y_j^c.
\]
We next show, using the extension of the extracellular potential,
that all the functions $c_j$ coincide with $u_{-1}^e$. Let $\mathcal{E}_\ve:H^1(\Omega_\ve^e) \to H^1(\Omega)$ be an extension operator such that $\mathcal{E}_\ve z = z$ in $\Omega_\ve^e$ and $\|\mathcal{E}_\ve z\|_{H^1(\Omega)} \le C \|z\|_{H^1(\Omega_\ve^e)}$, where $C$ is independent of $\ve$. Due to the a priori estimates, $\mathcal{E}_\ve (\ve u_\ve^e)$ converges weakly to $u_{-1}$ in $L^2(0,T; H^1(\Omega))$. Since $\mathcal{E}_\ve (\ve u_\ve^e)= \ve u_\ve^e$ in $\Omega_\ve^e$, $\ve u_\ve^e$ converges weakly two-scale to $\mathbf{1}_{Y^e}(y) u_{-1}(t,x)$, and thus $u_{-1}^e = u_{-1} \in L^2(0,T; H^1(\Omega))$.
Denote
\[
r_\ve
=
\ve u_\ve^c-\mathcal{E}_\ve(\ve u_\ve^e)
\quad\text{in }\Omega_\ve^c.
\]
Then
\(
r_\ve|_{\Gamma_\ve}
=
\ve[u_\ve].
\)
Applying the trace-Poincaré inequality on every
$\ve(k+Y_j^c)$, summing over $k$ and $j$, and integrating in
time, we obtain
\[
\begin{aligned}
\|r_\ve\|_{L^2((0,T)\times\Omega_\ve^c)}^2
&\le
C\ve^2
\|\nabla r_\ve\|_{L^2((0,T)\times\Omega_\ve^c)}^2\\
&\quad+
C\ve^3
\|[u_\ve]\|_{L^2((0,T)\times\Gamma_\ve)}^2
\le C\ve^2.
\end{aligned}
\]
Hence
\[
\mathbf{1}_{\Omega_\ve^c}(\ve u_\ve^c
-
\mathcal{E}_\ve(\ve u_\ve^e))
\longrightarrow 0
\quad\text{strongly in }
L^2((0,T)\times\Omega).
\]
Passing to the two-scale limit gives
\[
\sum_{j=1}^N c_j(t,x)\mathbf{1}_{Y_j^c}(y)
=
u_{-1}^e(t,x)\mathbf{1}_{Y^c}(y).
\]
Therefore
\[
c_j(t,x)=u_{-1}^e(t,x)
\qquad
\text{for every }j=1,\ldots,N.
\]
Consequently, we conclude that on every connected component of the corresponding phases
\[
u_{-1}^c=u_{-1}^e=u_{-1}\in L^2(0,T; H^1(\Omega)).
\]
This proves the first convergence in (a) and assertion (c).

\vspace{3mm}
\noindent
(b) Due to the uniform estimate (i) for $[u_\ve]$ on $\Gamma_\ve$, there exists $v_0\in L^2(0,T; L^2(\Omega \times \Gamma))$ such that $[u_\ve]$ converges weakly two-scale to $v_0$ in $L^2(0,T; L^2(\Gamma_\ve))$ (see \cite{damlamian1995two}). Let us show that $q^\alpha = \nabla u_{-1}^\alpha(t,x) + \nabla_y u_0^\alpha(t,x,y)$ ($\alpha= e, c$) and that $[u_\ve]$ converges weakly two-scale on $\Gamma_\ve$ to $[u_0]=(u_0^c - u_0^e)\big|_{y\in \Gamma}$. 

Take a smooth test function $\Phi(t,x,y)=(\Phi^c, \Phi^e)$ with $\Phi^e$ being $Y$-periodic, $\diy{\Phi^\alpha}=0$ in $Y^\alpha$, $\alpha=c, e$, and $\Phi^c \cdot \n = \Phi^e \cdot \n$ on $\Gamma$.
Integrating by parts yields
\begin{align}
\label{eq:test_2}
\begin{split}
&\ve \int_0^T \int_{\Omega_\ve^c \cup \Omega_\ve^e} \nabla u_\ve(t,x) \cdot \Phi(t,x, \frac{x}{\ve})\, \dd x \, \dd t  = - \ve \int_0^T \int_{\Omega_\ve^c \cup \Omega_\ve^e} u_\ve(t,x)  \dix{\Phi}(t,x,\frac{x}{\ve}) \, \dd x \, \dd t\\
-&\int_0^T \int_{\Omega_\ve^c \cup \Omega_\ve^e} u_\ve(t,x)\diy{\Phi}(t,x,\frac{x}{\ve}) \, \dd x \, \dd t
+ \ve \int_0^T \int_{\Gamma_\ve} [u_\ve]\, \Phi^c(t,x,\frac{x}{\ve})\cdot \n \, \dd S \, \dd t.
\end{split}
\end{align}
Passing to the limit in \eqref{eq:test_2} we obtain 
\begin{align*}
&\int_0^T \int_{\Omega} \int_{Y^c} ( q^c(t,x,y)- \nabla u_{-1}^c(t,x))\cdot  \Phi^c(t,x,y) \, \dd y \, \dd x \, \dd t\\
&+
\int_0^T \int_{\Omega} \int_{Y^e} ( q^e(t,x,y)- \nabla u_{-1}^e(t,x))\cdot  \Phi^e(t,x,y) \, \dd y \, \dd x \, \dd t\\
&= \int_0^T \int_\Omega\int_{\Gamma} v_0(t,x,y)\, \Phi^c(t,x,y)\cdot \n \, \dd S \,\dd x\, \dd t.
\end{align*}
Since it holds for all solenoidal $\Phi$, by the Helmholtz decomposition (see Lemma 3.9 in \cite{hummel2000homogenization}), there exist $u_0^e(t,x,y) \in L^2((0,T)\times \Omega; H_{\mathrm{per}}^1(Y^e))$ and $u_0^c(t,x,y) \in L^2((0,T)\times \Omega; H^1(Y^c))$ such that
\begin{align*}
&q^\alpha(t,x,y) -\nabla u_{-1}^\alpha(t,x) = \nabla_y u_0^\alpha(t,x,y),\quad
v_0\big|_{y\in \Gamma}=(u_0^c - u_0^e)\big|_{y\in \Gamma}.
\end{align*}
Consequently, $[u_\ve]\big|_{\Gamma_\ve} \, \stackrel{2}{\rightharpoonup}\, (u_0^c - u_0^e)\big|_{y\in \Gamma}$ weakly two-scale on $\Gamma_\ve$. 

The a priori estimate for $\partial_t [u_\ve]$ yields the existence of a two-scale limit on $\Gamma_\ve$, which is identified to be $\partial_t [u_0]= \partial_t(u_0^c - u_0^e)\big|_{y\in \Gamma} \in L^2((0,T) \times \Omega\times \Gamma)$ by integration by parts in time. This in turn ensures the $t$-pointwise two-scale convergence of the membrane potential $[u_\ve]$. 
Indeed, fix $t\in(0,T]$, $\xi\in C^\infty(\overline\Omega)$, and $\psi\in C_{\rm per}(Y)$. Choose
$\eta_t\in C^1([0,T])$ such that
$\eta_t(0)=0$ and $\eta_t(t)=1$, and set
\(
\varphi(\tau,x)=\eta_t(\tau)\xi(x).
\)
Using the product rule,
\[
\begin{aligned}
&\ve\int_{\Gamma_\ve}
[u_\ve](t,x)\xi(x)
\psi\left(\frac{x}{\ve}\right)\,\dd S_x
=
\ve\int_0^t\int_{\Gamma_\ve}
\left(
[u_\ve]\partial_\tau\varphi
+
\partial_\tau[u_\ve]\varphi
\right)
\psi\left(\frac{x}{\ve}\right)
\,\dd S_x\,\dd \tau.
\end{aligned}
\]
By extending the class of admissible time test functions from
continuous functions to $L^2(0,T)$ by density, we may pass to the
two-scale limit on $(0,t)$. Thus,
\[
\begin{aligned}
&\lim_{\ve\to0}
\ve\int_{\Gamma_\ve}
[u_\ve](t,x)\xi(x)
\psi\left(\frac{x}{\ve}\right)\,\dd S_x
\\
&=
\int_0^t\int_\Omega\int_\Gamma
\left(
[u_0]\partial_\tau\varphi
+
\partial_\tau[u_0]\varphi
\right)\psi(y)
\,\dd S_y\,\dd x\,\dd \tau
\\
&=
\int_\Omega\int_\Gamma
[u_0](t,x,y)\xi(x)\psi(y)
\,\dd S_y\,\dd x.
\end{aligned}
\]

\vspace{2mm}
\noindent
(d) By Lemma \ref{lm:bound-w}, $0\le X_\ve\le 1$ and thus, the weak space-time two-scale convergence of $X_\ve$ is clear. By \ref{H2} and the estimates obtained above, one obtains a uniform in $\ve$ estimate for $\partial_t X_\ve$, and the two-scale limit is identified by integration by parts in time. Repeating the argument for $[u_\ve]$, integrating by parts in time, we obtain the t-pointwise convergence to $X_0\in H^1(0,T; L^2(\Omega \times \Gamma))$. \\

\noindent
Assertion (e) follows immediately from the boundedness of $X_\ve$ and Lipschitz continuity of $f(\cdot, \cdot)$.
Lemma \ref{lm:compactness} is proved.
\end{proof}

While passing to the limit in \eqref{eq:microscopic-prob}, we will need the following classical homogenization results for the auxiliary function $p_\ve$ solving \eqref{eq:p_eps} (See Theorem 2.6, \cite{allaire1992homogenization}, and Theorem 5.1 in \cite{holmbom1997homogenization}).  
\begin{align}
\label{eq:convergence_p_eps}
\begin{split}
    &p_\ve \to p_0 \quad \mbox{strongly in}\,\, L^2(0,T; L^2(\Omega)),\\
    &\Big(\nabla p_\ve - \nabla p_0 - \nabla_y \mathcal{M}_j\big(\frac{x}{\ve}\big) \partial_{x_j} p_0\Big) \to 0 \quad \mbox{strongly in}\,\, L^2(0,T; L^2(\Omega)),
\end{split}
\end{align}
where the limit function $p_0$ solves the homogenized problem
\begin{align*}
\begin{split}
&\di{A \nabla p_0} = 0, \quad x\in \Omega,  \\
&p_0 = g, \quad x\in \partial \Omega,
\end{split}
\end{align*}
with $A_{ij}= \int_Y \sigma(\partial_i \mathcal{M}_j +\delta_{ij})\, \dd y$, and $\mathcal{M}_j(y)$, $j=1, 2, 3$, solving cell problems \eqref{eq:cell-M}. The assumptions of Theorem 2.6 in \cite{allaire1992homogenization} are satisfied since $p_0\in H^2(\Omega)$ for $g$ satisfying \ref{H5} and $\partial \Omega$ having $C^2$ regularity. 

\subsection{Proof of the homogenization result}
\label{sec:homogenization}
In this section we prove Theorem \ref{th:main}.
The compactness given by Lemma \ref{lm:compactness} is enough to pass to the limit in the weak formulation of \eqref{eq:microscopic-prob}, but not sufficient to identify the limit functions $\overline{S}$ and $\overline{f}$. 

Indeed, to pass to the limit, we start by writing the weak formulation of \eqref{eq:microscopic-prob}:
\begin{align}
\begin{split}
\label{eq:weak-v_eps_w_eps-1}
    &\int_0^T \int_{\Gamma_{\ve}} (-c_m [u_{\ve}] \partial_t [\phi_1] + S_m(X_{\ve}) [u_{\ve}][\phi_1]) \, \dd S\,\dd t 
    + \ve\int_0^T \int_{\Omega_\ve^c \cup \Omega_\ve^e} \sigma_{\ve} \nabla u_{\ve} \cdot \nabla \phi_1 \, \dd x\, \dd t\\ 
    &= \int_{\Gamma_{\ve}} c_m V_\ve^{\rm in}[\phi_1](0,x)\, \dd S 
    -\int_0^T \int_{\Omega_\ve^c \cup \Omega_\ve^e} \sigma_\ve \nabla p_\ve \cdot \nabla \phi_1\, \dd x\, \dd t,\\
    &-\ve \int_0^T \int_{\Gamma_{\ve}} X_\ve\, \partial_t \phi_2\, \dd S\, \dd t
    = \ve \int_0^T \int_{\Gamma_{\ve}} f([u_\ve],X_\ve)\, \phi_2\, \dd S\, \dd t + \ve \int_{\Gamma_{\ve}} X_\ve^{\rm in}\, \phi_2(0,x)\, \dd S,
\end{split}
\end{align}
for any test function
$\phi_1=(\phi_1^c, \phi_1^e)\in  C^1([0,T]; H^1(\Omega_\ve^c)) \times C^1([0,T];H^1( \Omega_\ve^e))$ such that $\phi_1^e |_{\partial \Omega} = 0$, $\phi_1\big|_{t=T}=0$, and $\phi_2 \in C^1([0,T]; L^2(\Gamma_\ve))$ such that $\phi_2\big|_{t=T}=0$.

We test the first equation in \eqref{eq:weak-v_eps_w_eps-1} with the function $\phi_1=\varphi(t,x) + \ve \psi(t,x,\frac{x}{\ve})$, where 
$\varphi \in C_0^1([0,T)\times\Omega)$, $\psi \in C_0^1([0,T)\times \Omega; C^1(\overline{Y^c})\times C_{\mathrm{per}}^1(\overline{Y^e}))$, and the second one with $\phi_2\in
C_0^1\left(
[0,T)\times\Omega;
C_{\mathrm{per}}^1(\overline Y)
\right)$, and pass to the limit as $\ve \to 0$ using Lemma \ref{lm:compactness}. Notice that $[\varphi]=0$ on $\Gamma_\ve$. Then the limit functions $(u_{-1}, u_0, X_0)$ satisfy
\begin{align}
\begin{split}
\label{eq:weak-v_eps_w_eps-3}
    &\int_0^T \int_\Omega \int_{\Gamma} (-c_m [u_0]\partial_t [\psi] + \overline{S}(t,x,y)[\psi](t,x,y))  \, \dd S \, \dd x\, \dd t\\ 
    &+ 
    \int_0^T \int_{\Omega} \int_{Y^c \cup Y^e} \sigma (\nabla_x u_{-1} + \nabla_y u_0) \cdot \big(\nabla_x \varphi(t,x) + \nabla_y \psi (t,x,y)\big) \, \dd y\, \dd x\, \dd t\\ 
    &= \int_\Omega \int_{\Gamma} c_m V^{\rm in}(x,y) [\psi](0,x,y)\, \dd S\, \dd x\\
    &-\int_0^T \int_{\Omega} \int_{Y^c \cup Y^e} \sigma (\mathbf{e}_j + \nabla_y \mathcal{M}_j)\partial_{x_j} p_0 \cdot \big(\nabla_x \varphi(t,x) + \nabla_y \psi (t,x,y)\big)\, \dd y \, \dd x\, \dd t\\[2mm]
    & \mbox{and}\\
    &-\int_0^T \int_\Omega \int_{\Gamma} X_0\, \partial_t \phi_2\, \dd S\,\dd x\, \dd t
    = \int_0^T \int_\Omega \int_{\Gamma} \overline{f}(t,x,y)\, \phi_2\, \dd S\,\dd x\, \dd t + \int_\Omega \int_{\Gamma} X^{\rm in}\, \phi_2(0,x, y)\, \dd S\, \dd x.
\end{split}
\end{align}

In order to identify the limit functions $\overline{S}, \overline{f}$, we will modify \eqref{eq:prob-on-Gamma_eps} so that the passage to the limit can be handled by the monotonicity arguments. One way is to replace $S_m(X_\ve)[u_\ve]$ in \eqref{eq:prob-on-Gamma_eps} by a bounded perturbation of a linear function by setting
\begin{align}
\label{eq:S(v,X)}
S_M([u_\ve], X_\ve)= S_0 [u_\ve] + S_1 X_\ve T_M([u_\ve]), \quad  T_M([u_\ve]) = \max\big(-M, \min(M, [u_\ve])\big),
\end{align}
for some large $M>0$.
This modification is clearly justified in practice since the voltage values are always finite. Note that the results of numerical simulations in Section \ref{sec:numerics} are not affected. The question of the uniform in $\ve$ boundedness of the membrane potential $\|[u_\ve]\|_{L^\infty((0,T)\times \Gamma_\ve)}\le C$, is not trivial and will be considered elsewhere.  

The modified microscopic problem stated on $\Gamma_\ve$ reads:
\begin{align}
\label{eq:Lip-prob}
\begin{split}
    &c_m \partial_t [u_{\ve}] + \ve \mathcal{L}_{\ve} [u_{\ve}] + S_M([u_\ve], X_{\ve}) = -\sigma_\ve \nabla p_\ve \cdot \n, \quad [u_\ve](0,x)=V_\ve^{\rm in}(x) \quad \mbox{on}\ \Gamma_\ve,\\   
    &\partial_t X_\ve
    = f([u_\ve],X_\ve),\quad X_\ve(0,x)=X_\ve^{\rm in}(x), \,\,\mbox{on}\,\,\Gamma_\ve,
\end{split}
\end{align}
where $\mathcal{L}_\ve$ is given, as before, by \eqref{def:L_eps}. 
Note that this modification does not change the a priori estimates, and the convergence in Lemma \ref{lm:compactness} is valid. The following lemma states that $S_M$ satisfies one-sided Lipschitz (or monotonicity) condition along the solutions of \eqref{eq:Lip-prob}.
\begin{lemma}
\label{lm:Lip-S}
Let $([u_\ve], X_\ve)$ be a solution of \eqref{eq:Lip-prob} and $\varphi(t,x), \psi(t,x) \in L^2((0,T)\times\Gamma_\ve)$. Then for $S_M$ defined by \eqref{eq:S(v,X)} we have
\begin{align*}
(S_M([u_\ve], X_\ve) - S_M(\varphi, \psi) \, , \, [u_\ve]-\varphi)_{L^2(\Gamma_\ve)} \ge
- C_M (\|[u_\ve]-\varphi\|_{L^2(\Gamma_\ve)}^2 + \|X_\ve-\psi\|_{L^2(\Gamma_\ve)}^2),
\end{align*}
where $C_M$ depends only on $M$ in the cut-off \eqref{eq:S(v,X)}, on the constants $S_0, S_1$, and is independent of $\ve$. Moreover, there exists $\overline{S}_M \in L^2(0,T; L^2(\Omega \times \Gamma))$ such that $S_M([u_\ve], X_\ve)$ converges weakly two-scale to $\overline{S}_M$ in $L^2(0,T; L^2(\Gamma_\ve))$ (in the sense of Definition \ref{def:2scale-Gamma_eps}).
\end{lemma}
\begin{proof}
Due to Lemma \ref{lm:bound-w}, $0\le X_\ve\le 1$. 
Adding and subtracting $T_M(\varphi)\, X_\ve$ and using the monotonicity of the truncation $T_M$, we have
\begin{align*}
&(S_M([u_\ve], X_\ve) - S_M(\varphi, \psi) \, , \, [u_\ve]-\varphi)_{L^2(\Gamma_\ve)}\\ 
&=
S_0\|[u_\ve]-\varphi\|_{L^2(\Gamma_\ve)}^2
+ S_1\int_{\Gamma_\ve} X_\ve (T_M([u_\ve])-T_M(\varphi))([u_\ve]-\varphi)\, \dd S\\
&+ S_1\int_{\Gamma_\ve} (X_\ve-\psi) T_M(\varphi)([u_\ve]-\varphi)\, \dd S\\
&\ge - S_1\int_{\Gamma_\ve} |X_\ve-\psi| |T_M(\varphi)||[u_\ve]-\varphi|\, \dd S\\
&\ge - \frac{S_1\, M}{2} (\|[u_\ve]-\varphi\|_{L^2(\Gamma_\ve)}^2 + \|X_\ve-\psi\|_{L^2(\Gamma_\ve)}^2).
\end{align*}
The two-scale convergence is proved in the same way as (e) in Lemma \ref{lm:compactness}. Lemma \ref{lm:Lip-S} is proved.
\end{proof}

Let us introduce a coupled limit problem for the unknowns $(u_{-1}, u_0, X_0)$:
\begin{align}
&{\rm div}_x {\int_{Y^c\cup Y^e} \sigma(\nabla_y u_0 + \nabla_x u_{-1})\, dy}=0, \, & x \in \,\, &\Omega, \nonumber\\
&\diy{\sigma (\nabla_y u_0 + \nabla_x u_{-1})}  =0, \, & y \in \,\, &Y^e\cup Y^c,  \nonumber\\
&[\sigma (\nabla_y u_0+\nabla_x u_{-1}) \cdot \n] =0, \, & y \in \,\, &\Gamma, \nonumber \\
\label{eq:effective-problem-coupled-strong}
& -\sigma^c \nabla_y u_0^c \cdot \n
=c_m \partial_t [u_0] + S_M([u_0], X_0) + \sigma^c\nabla_x u_{-1}\cdot \n &&\\
&\hspace{4.cm} + {\sigma^c (\nabla_y \mathcal{M}_j + \mathbf{e_j})\cdot \n \, \partial_{x_j} p_0},\, & y \in \,\, &\Gamma, \nonumber\\
& \partial_t X_0 = f([u_0], X_0), \quad X_0(0,x,y)=X^{\rm in}(x, y),\, & y \in \,\, &\Gamma, \nonumber\\
& [u_0](0,x,y)=V^{\rm in}(x,y),\, & y \in \,\, &\Gamma, \nonumber\\
& u_{-1}\big|_{\partial \Omega}=0, \quad u_0^e(x,\cdot) \,\, \mbox{is} \, \, Y-\mbox{periodic}.\nonumber
\end{align}
Here $p_0$ solves the classical limit problem \eqref{eq:p_0}. Note that $u_0$ is defined up to an additive constant in $y$, as the classical corrector in homogenization.
The weak formulation of \eqref{eq:effective-problem-coupled-strong} reads: Find $u_{-1}\in L^2(0,T; H_0^1(\Omega))$, $u_0\in L^2((0,T) \times \Omega; (H^1(Y^c) \times H_{\rm per}^1(Y^e))/\mathbb R)$, $X_0\in H^1(0,T; L^2(\Omega\times \Gamma))$ such that for any $\varphi \in C_0^1([0,T)\times\Omega)$, $\psi \in C_0^1([0,T) \times \Omega; C_{\mathrm{per}}^1(Y^c \cup Y^e)))$, and  $\phi_2 \in C_0^1([0,T)\times \Omega; C_{\mathrm{per}}^1(\Gamma))$, we have
\begin{align*}
\begin{split}
    &\int_0^T \int_\Omega \int_{\Gamma} (-c_m [u_0]\partial_t [\psi] -X_0\, \partial_t \phi_2)  \, \dd S \, \dd x\, \dd t \\ 
    &+ 
    \int_0^T \int_{\Omega} \int_{Y^c \cup Y^e} \sigma (\nabla_x u_{-1} + \nabla_y u_0) \cdot \big(\nabla_x \varphi(t,x) + \nabla_y \psi (t,x,y)\big) \, \dd y\, \dd x\, \dd t\\
    &+\int_0^T \int_\Omega \int_{\Gamma} (S_M([u_0], X_0)[\psi](t,x,y) - f([u_0], X_0)\, \phi_2)\, \dd S\,\dd x\, \dd t\\ 
    &= \int_\Omega \int_{\Gamma} c_m V^{\rm in}(x,y) [\psi](0,x,y)\, \dd S\, \dd x + \int_\Omega \int_{\Gamma} X^{\rm in}(x,y)\, \phi_2(0,x, y)\, \dd S\, \dd x\\
    &-\int_0^T \int_{\Omega} \int_{Y^c \cup Y^e} \sigma (\mathbf{e}_j + \nabla_y \mathcal{M}_j)\partial_{x_j} p_0 \cdot \big(\nabla_x \varphi(t,x) + \nabla_y \psi (t,x,y)\big)\, \dd y \, \dd x\, \dd t.
\end{split}
\end{align*}

Above, we have derived a limit equality \eqref{eq:weak-v_eps_w_eps-3} without identifying the limits for the nonlinear functions. Using Lemma \ref{lm:Lip-S}, we will employ the monotonicity argument to prove that $\overline{S}=S_M([u_0], X_0)$, $\overline{f}=f([u_0], X_0)$. 
Let us write \eqref{eq:Lip-prob} in a more compact form 
\begin{align*}
\partial_t U_\ve  + \mathbb{A}_\ve(U_\ve) = F_\ve, \quad U_\ve\big|_{t=0}=U_\ve^{\rm in} \quad \mbox{on}\,\, \Gamma_\ve,
\end{align*}
where
\begin{align*}
\begin{split}
&U_\ve=\begin{pmatrix}
    [u_\ve]\\ X_\ve
\end{pmatrix}, \quad
U_\ve^{\rm in}=\begin{pmatrix}
    V_\ve^{\rm in}\\ X_\ve^{\rm in}
\end{pmatrix},\\[2mm]
&\mathbb{A}_\ve(U_\ve) =
\begin{pmatrix}
\displaystyle
\frac{\ve}{c_m} \mathcal{L}_\ve [u_\ve] + \frac{1}{c_m}S_M([u_\ve], X_{\ve})\\[3mm] - f([u_\ve], X_\ve)
\end{pmatrix},\quad
F_\ve = 
\begin{pmatrix}
\displaystyle
-\frac{1}{c_m}\sigma_\ve \nabla p_\ve \cdot \n \\[3mm]
0
\end{pmatrix}.     
\end{split}
\end{align*}
Denote $H_\ve=L^2(\Gamma_\ve) \times L^2(\Gamma_\ve)$, {$V_\ve=H^{1/2}(\Gamma_\ve) \times L^2(\Gamma_\ve)$, $V_\ve^\ast=H^{-1/2}(\Gamma_\ve)\times L^2(\Gamma_\ve)$}. For any test function $\Phi=(\phi_1, \phi_2)\in C([0,T]; H_\ve) \cap L^2(0,T; V_\ve)$, due to the coercivity of $(\mathcal{L}_\ve+ S_0{\rm Id})$ (see Lemma \ref{lm:L_eps}), Lipschitz continuity of $f$ \eqref{eq:Lip-f}, and Lemma \ref{lm:Lip-S} we obtain 
\begin{align}
\label{eq:A-monotonicity}
\int_0^t \langle \mathbb{A}_\ve(U_\ve(s, \cdot)) - \mathbb{A}_\ve(\Phi(s, \cdot)) \, , \, U_\ve(s, \cdot) - \Phi(s, \cdot)\rangle_{V_\ve^\ast, V_\ve}\, \dd s
\ge - \rho \int_0^t \|U_\ve(s, \cdot) - \Phi(s, \cdot)\|_{H_\ve}^2 \, \dd s.
\end{align}
Property \eqref{eq:A-monotonicity} can be interpreted as the generalized monotonicity of $\mathbb{A}_\ve$ along the solutions, since we use the uniform boundedness of $X_\ve$ to derive this estimate.
Using the product rule for $e^{-2\rho t}\|U_\ve\|_{H_\ve}^2$ with $\rho$ defined in \eqref{eq:A-monotonicity}, we have
\begin{align*}
&\frac{1}{2}\frac{\dd}{\dd t} \Big(e^{-2\rho t}\|U_\ve\|_{H_\ve}^2\Big)
+ \rho e^{-2\rho t}\|U_\ve\|_{H_\ve}^2
+ e^{-2\rho t}\langle \mathbb{A}_\ve(U_\ve), U_\ve\rangle_{V_\ve^\ast, V_\ve} = e^{-2\rho t} \langle F_\ve, U_\ve\rangle_{V_\ve^\ast, V_\ve},\\
& 
\frac{1}{2}e^{-2\rho t}\|U_\ve\|_{H_\ve}^2 =
\frac{1}{2} \|U_\ve^{\rm in}\|_{H_\ve}^2
- \int_0^t e^{-2\rho s}\Big(\langle \mathbb{A}_\ve(U_\ve(s, \cdot)), U_\ve(s, \cdot)\rangle_{V_\ve^\ast, V_\ve} + \rho \|U_\ve(s, \cdot)\|_{H_\ve}^2\Big)\, \dd s\\
& \hspace{2.5cm} + \int_0^t e^{-2\rho s} \langle F_\ve(s, \cdot), U_\ve(s, \cdot)\rangle_{V_\ve^\ast, V_\ve}\, \dd s.
\end{align*}
We add and subtract {under the integral sign} in the last identity $\langle \mathbb{A}_\ve(U_\ve) - \mathbb{A}_\ve(\Phi), \Phi\rangle_{V_\ve^\ast, V_\ve}$, $\langle \mathbb{A}_\ve(\Phi), U_\ve\rangle_{V_\ve^\ast, V_\ve}$, and write
\begin{align*}
\|U_\ve\|_{H_\ve}^2 = \|U_\ve-\Phi\|_{H_\ve}^2 + \|\Phi\|_{H_\ve}^2 + 2(U_\ve - \Phi, \Phi)_{H_\ve},
\end{align*}
which yields
\begin{align*}
&e^{-2\rho t}\|U_\ve(t, \cdot)\|_{H_\ve}^2 =
\|U_\ve^{\rm in}\|_{H_\ve}^2
- 2 \int_0^t e^{-2\rho s}\Big(\langle \mathbb{A}_\ve(U_\ve(s, \cdot))-\mathbb{A}_\ve(\Phi(s, \cdot)), U_\ve(s, \cdot)-\Phi(s, \cdot)\rangle_{V_\ve^\ast, V_\ve}\\ 
&+ \rho \|U_\ve(s, \cdot)-\Phi(s, \cdot)\|_{H_\ve}^2\Big)\, \dd s\\
& - 2 \int_0^t e^{-2\rho s}\Big(\langle \mathbb{A}_\ve(\Phi(s, \cdot)), U_\ve(s, \cdot)\rangle_{V_\ve^\ast, V_\ve} + \rho \|\Phi(s, \cdot)\|_{H_\ve}^2\Big)\, \dd s\\
& - 2 \int_0^t e^{-2\rho s}\Big(\langle \mathbb{A}_\ve(U_\ve(s, \cdot))-\mathbb{A}_\ve(\Phi(s, \cdot)),\Phi(s, \cdot)\rangle_{V_\ve^\ast, V_\ve} + 2 \rho (U_\ve(s, \cdot) - \Phi(s, \cdot), \Phi(s, \cdot))_{H_\ve} \Big)\, \dd s\\
&+ 2\int_0^t e^{-2\rho s} \langle F_\ve(s, \cdot), U_\ve(s, \cdot)\rangle_{V_\ve^\ast, V_\ve}\, \dd s.
\end{align*}
Using the monotonicity property \eqref{eq:A-monotonicity}, we can estimate the RHS of the last equality:
\begin{align}
\label{eq:ineq-1}
\begin{split}
&e^{-2\rho t}\|U_\ve(t, \cdot)\|_{H_\ve}^2 \le
\|U_\ve^{\rm in}\|_{H_\ve}^2
- 2 \int_0^t e^{-2\rho s}\Big(\langle \mathbb{A}_\ve(\Phi(s, \cdot)), U_\ve(s, \cdot)\rangle_{V_\ve^\ast, V_\ve} + \rho \|\Phi(s, \cdot)\|_{H_\ve}^2\Big)\, \dd s\\
& - 2 \int_0^t e^{-2\rho s}\Big(\langle \mathbb{A}_\ve(U_\ve(s, \cdot))-\mathbb{A}_\ve(\Phi(s, \cdot)),\Phi\rangle_{V_\ve^\ast, V_\ve} + 2 \rho (U_\ve(s, \cdot) - \Phi(s, \cdot), \Phi(s, \cdot))_{H_\ve} \Big)\, \dd s\\
&+ 2 \int_0^t e^{-2\rho s} \langle F_\ve(s, \cdot), U_\ve(s, \cdot)\rangle_{V_\ve^\ast, V_\ve}\, \dd s.
\end{split}
\end{align}
Multiplying \eqref{eq:ineq-1} by $\ve$ and a nonnegative $\Psi\in C_0^\infty(0,T)$ and integrating over $(0,T)$, we obtain
\begin{align}
\label{eq:ineq-2}
\begin{split}
&\limsup_{\ve \to 0}  \int_0^T \Psi(t) \Big(e^{-2\rho t} \ve \|U_\ve(t, \cdot)\|_{H_\ve}^2 - \ve \|U_\ve^{\rm in}\|_{H_\ve}^2\Big)\, \dd t\\
&\le -
\liminf_{\ve \to 0} \Big(2 \ve \int_0^T \Psi(t)\int_0^t e^{-2\rho s}\Big(\langle \mathbb{A}_\ve(\Phi(s, \cdot)), U_\ve(s, \cdot)\rangle_{V_\ve^\ast, V_\ve} + \rho \|\Phi(s, \cdot)\|_{H_\ve}^2\Big)\, \dd s\, \dd t\\
& + 2 \ve \int_0^T \Psi(t)\int_0^t e^{-2\rho s}\Big(\langle \mathbb{A}_\ve(U_\ve(s, \cdot))-\mathbb{A}_\ve(\Phi(s, \cdot)),\Phi(s, \cdot)\rangle_{V_\ve^\ast, V_\ve}\\ 
& \hspace{25mm}+ 2 \rho (U_\ve(s, \cdot) - \Phi(s, \cdot), \Phi(s, \cdot))_{H_\ve} \Big)\, \dd s\, \dd t\\
&+ 2\int_0^T \Psi(t)\int_0^t e^{-2\rho s} \langle \ve F_\ve(s, \cdot), U_\ve(s, \cdot)\rangle_{V_\ve^\ast, V_\ve}\, \dd s\, \dd t\Big).
\end{split}
\end{align}
Let us choose a test function $\Phi_\ve=([\phi_1^\ve], \phi_2^\ve)$ in the form of the asymptotic ansatz for the solution $(u_\ve, X_\ve)$. Specifically, let
\[
\varphi_{-1}\in C([0,T];C_0^1(\Omega)),
\]
and let $\varphi_0=(\varphi_0^c,\varphi_0^e)$ satisfy
\[
\varphi_0
\in
C\left(
[0,T];
C_0^1\left(
\Omega;
C^1(\overline{Y^c})
\times C_{\mathrm{per}}^1(\overline{Y^e})
\right)
\right).
\]
We define, for $\alpha\in\{c,e\}$,
\[
\phi_1^{\ve,\alpha}(t,x)
=
\ve^{-1}\varphi_{-1}(t,x)
+
\varphi_0^\alpha\left(t,x,\frac{x}{\ve}\right),
\qquad x\in\Omega_\ve^\alpha,
\]
where the cell functions are periodically repeated.

For the second component, take
\[
\phi_2
\in
C\left(
[0,T];
C^1\left(\overline{\Omega};
C_{\mathrm{per}}^1(\overline Y)\right)
\right)
\]
and set
\[
\phi_2^\ve(t,x)
=
\phi_2\left(t,x,\frac{x}{\ve}\right),
\qquad x\in\Gamma_\ve.
\]
Note that $[\varphi_{-1}]=0$ through $\Gamma_\ve$.
Since the test functions converge strongly two-scale, we can pass to the limit in the nonlinear functions. Denote the two-scale limit of $([u_\ve], X_\ve)$ in $L^2(0,T; L^2(\Gamma_\ve))^2$ by $U=([u_0], X_0)$, and denote $H_0=L^2(\Omega \times \Gamma)^2$. Passing to the limit along a subsequence in \eqref{eq:ineq-2} yields
\begin{align}
\label{eq:ineq-3}
\begin{split}
&\limsup_{\ve \to 0} \int_0^T \Psi(t) \Big(e^{-2\rho t} {\ve}\|U_\ve\|_{H_\ve}^2 - \ve \|U_\ve^{\rm in}\|_{H_\ve}^2\Big)\, \dd t\\ 
&\le 
- 2 \int_0^T \Psi(t)\int_0^t e^{-2\rho s}\Big( \frac{1}{c_m} \int_{\Omega}\int_{Y^c \cup Y^e} \sigma(\nabla_x u_{-1} + \nabla_y u_0) \cdot (\nabla_x \varphi_{-1} + \nabla_y \varphi_0)\, \dd y\, \dd x\\ 
&+ \frac{1}{c_m} \int_{\Omega}\int_\Gamma S_M([\varphi_0], \phi_2)[u_0]\, \dd S\, \dd x - \int_{\Omega}\int_\Gamma f([\varphi_0], \phi_2)X_0\, \dd S\, \dd x+ \rho \|\Phi\|_{H_0}^2\Big)\, \dd s \, \dd t\\
& - 2 \int_0^T \Psi(t)\int_0^t e^{-2\rho s}
\Big(
\frac{1}{c_m} \int_{\Omega}\int_{Y^c \cup Y^e} \sigma(\nabla_x (u_{-1}-\varphi_{-1}) + \nabla_y (u_0-\varphi_0)) \cdot (\nabla_x \varphi_{-1} + \nabla_y \varphi_0)\, \dd y\, \dd x\\ 
&+ \frac{1}{c_m} \int_{\Omega}\int_\Gamma \big(\overline{S}_M-S_M([\varphi_0], \phi_2)\big)[\varphi_0]\, \dd S\, \dd x 
- \int_{\Omega}\int_\Gamma (\overline{f}-f([\varphi_0], \phi_2))\phi_2\, \dd S\, \dd x\\ 
&+ 2 \rho (U - \Phi, \Phi)_{H_0}
\Big)\, \dd s\, \dd t\\
& -2\int_0^T \Psi(t)\int_0^t e^{-2\rho s} \frac{1}{c_m}\int_{\Omega} \int_{Y^c \cup Y^e} \sigma (\mathbf{e}_j + \nabla_y \mathcal{M}_j) \partial_{x_j} p_0 \cdot (\nabla_x u_{-1} + \nabla_y u_0)\,\dd y\, \dd x\,  \dd s\, \dd t.
\end{split}
\end{align}
In the term $\langle \mathbb{A}_\ve(\Phi_\ve), \Phi_\ve \rangle$, since $\phi_1^\ve$ is not a harmonic extension with the prescribed jump $[\phi_1^\ve]$, we use the minimizing-property inequality, namely
\begin{align*}
\langle \mathcal{L}_\ve [\phi_1^\ve], [\phi_1^\ve]\rangle \le \int_{\Omega_\ve^c \cup \Omega_\ve^e} \sigma_\ve \nabla \phi_1^\ve \cdot \nabla \phi_1^\ve \, \dd x.
\end{align*}
The passage to the limit in the last term uses the strong convergence of the corrected gradient $\nabla p_\ve$ (see \eqref{eq:convergence_p_eps}). By density, the last inequality is valid for all $\varphi_{-1} \in L^2(0,T; H_0^1(\Omega))$, $\varphi_0 \in L^2((0,T)\times \Omega; H_{\mathrm{per}}^1(Y^c \cup Y^e))$.
Let us choose the test functions $\varphi_{-1}=u_{-1}$, $\varphi_0=u_0$, and $\phi_2= X_0$ in \eqref{eq:ineq-3}, \eqref{eq:ineq-4},  and apply the product rule in \eqref{eq:weak-v_eps_w_eps-3}. We obtain
\begin{align}
\label{eq:ineq-4}
\begin{split}
&\int_0^T \Psi(t) \Big(e^{-2\rho t}\|U\|_{H_0}^2
- \|U^{in}\|_{H_0}^2\Big)\, \dd t\\
&= - 2 \int_0^T \Psi(t)\int_0^t e^{-2\rho s}\Big( \frac{1}{c_m} \int_{\Omega}\int_{Y^c \cup Y^e} \sigma|\nabla_x u_{-1} + \nabla_y u_0|^2\, \dd y\, \dd x\\ 
&+ \frac{1}{c_m} \int_{\Omega}\int_\Gamma \overline{S}_M[u_0]\, \dd S\, \dd x -  \int_{\Omega}\int_\Gamma \overline{f} X_0\, \dd S\, \dd x + \rho \|U\|_{H_0}^2\\
&{+\frac{1}{c_m}\int_{\Omega} \int_{Y^c \cup Y^e} \sigma (\mathbf{e}_j + \nabla_y \mathcal{M}_j)\partial_{x_j} p_0 \cdot \big(\nabla_x u_{-1} + \nabla_y u_0\big)\, \dd y \, \dd x}\Big)\, \dd s\, \dd t.
\end{split}
\end{align}
Since $U_\ve=([u_\ve], X_\ve)$ converges weakly two-scale to $U=([u_0], X_0)$ in $L^2(\Gamma_\ve)$, for an arbitrary nonnegative $\Psi \in C_0([0,T])$, the following semi-continuity holds:
\begin{align*}
\int_0^T \Psi\|U\|_{H_0}^2\, \dd t
\le \liminf_{\ve \to 0} \int_0^T \Psi \ve \|U_\ve\|_{H_\ve}^2\, \dd t.
\end{align*}
Subtracting \eqref{eq:ineq-4} from \eqref{eq:ineq-3}  and rearranging the terms, we obtain the following inequality:
\begin{align}
\label{eq:ineq-5}
\begin{split}
&0\le \limsup_{\ve \to 0}\frac{1}{2}\int_0^T \Psi(t) e^{-2\rho t}\Big(\ve \|U_\ve\|_{H_\ve}^2 - \|U\|_{H_0}^2\Big)\, \dd t\\
& \le \int_0^T \Psi(t)\int_0^t e^{-2\rho s}\Big( \frac{1}{c_m} \int_{\Omega}\int_{Y^c \cup Y^e} \sigma|\nabla_x (u_{-1}-\varphi_{-1}) + \nabla_y (u_0-\varphi_0)|^2\, \dd y\, \dd x\\ 
&+ \frac{1}{c_m} \int_{\Omega}\int_\Gamma (\overline{S}_M-S_M([\varphi_0], \phi_2))([u_0]-[\varphi_0])\, \dd S_y\, \dd x\\
&-  \int_{\Omega}\int_\Gamma (\overline{f}-f([\varphi_0], \phi_2)) (X_0-\phi_2)\, \dd S_y\, \dd x+ \rho \|U-\Phi\|_{H_0}^2\Big)\, \dd s \, \dd t.
\end{split}
\end{align}
Let us now take 
\begin{align*}
\varphi_{-1}= u_{-1}+ \delta \widetilde{\varphi}_{-1},\quad \varphi_0 = u_0 + \delta \widetilde{\varphi}_{0},\quad \phi_2 = X_0 + \delta \widetilde{\phi}_2.
\end{align*}
Combining this upper bound with the preceding weak lower-semicontinuity inequality gives, for any nonnegative $\Psi \in C_c^\infty(0,T)$,
\begin{align*}
\lim_{\ve \to 0}\int_0^T \Psi(t) e^{-2\rho t}\ve \|U_\ve\|_{H_\ve}^2\, \dd t = \int_0^T \Psi(t) e^{-2\rho t} \|U\|_{H_0}^2\, \dd t.
\end{align*}
Writing $\Psi$ as a sum of positive and negative parts, we have the last equality for each $\Psi \in C_c^\infty(0,T)$.
In the next step, we divide by $\delta>0$ and pass to the limit as $\delta \to 0^+$. After that we divide by $\delta<0$ and pass to the limit as $\delta \to 0^-$, which yields
\begin{align*}
&\overline{S}_M = \lim_{\delta \to 0} S_M([u_0 + \delta \widetilde{\varphi}_{0}], X_0 + \delta \widetilde{\phi}_2) = S_M([u_0], X_0), \\
&\overline{f} = \lim_{\delta \to 0} f([u_0 + \delta \widetilde{\varphi}_{0}], X_0 + \delta \widetilde{\phi}_2) = f([u_0], X_0).
\end{align*}
By the a priori estimates,
\begin{align*}
&|\ve\|U_\ve(t, \cdot)\|_{H_\ve}^2 -\ve\|U_\ve(s, \cdot)\|_{H_\ve}^2| = 2 \ve \Big|\int_s^t (U_\ve(\tau, \cdot), \partial_{\tau}U_\ve(\tau, \cdot))_{H_\ve}\, \dd \tau\Big|\\
&\le 
2 \left(\sup_{\tau \in [s, t]} \ve\|U_\ve(\tau, \cdot)\|_{H_\ve}^2\right)^{1/2}
\left(\int_s^t \ve\|\partial_\tau U_\ve(\tau, \cdot)\|_{H_\ve}^2\, \dd \tau \right)^{1/2}|t-s|^{1/2}\le C |t-s|^{1/2}.
\end{align*}
Thus, $G_\ve(t)=e^{-2\rho t}\ve \|U_\ve\|_{H_\ve}^2$ is uniformly bounded and equicontinuous on $[0,T]$, and by the Arzelà–Ascoli theorem, $G_\ve$ converges to $G_0= e^{-2\rho t} \|U\|_{H_0}^2$ on $[0,T]$. Since $e^{-2\rho t}>0$, it yields the pointwise convergence
\begin{align*}
\sup_{t\in [0,T]}\Big|\ve \|U_\ve\|_{H_\ve}^2- \|U\|_{H_0}^2\Big| \to 0, \quad \ve \to 0,
\end{align*}
{which proves \eqref{eq:strong-2scale}.}

The existence of a solution $(u_{-1}, u_0, X_0)$ to the macroscopic problem \eqref{eq:effective-problem-coupled-strong} can be proved by using Galerkin approximations as in Section \ref{sec:existence}, a fixed point argument, or the methods in \cite{showalter1974degenerate}, \cite{showalter2013monotone}. Note that the time derivative $\partial_t u_{-1}$ is not present in the equation, while $[u_0]$ and $X_0$ evolve in time, so the system is of elliptic-parabolic type. 
Let us prove the uniqueness. Assume that there exist two solutions $(u_{-1}, u_0, X_0)$ and $(\widetilde{u}_{-1}, \widetilde{u}_0, \widetilde{X}_0)$. Then the differences satisfy
\begin{align*}
&\frac{c_m}{2} \frac{\dd}{\dd t} \int_{\Omega} \int_{\Gamma} [u_{0}-\widetilde{u}_{0}]^2\, \dd S\, \dd x
+ \frac{1}{2} \frac{\dd}{\dd t} \int_{\Omega} \int_{\Gamma} (X_0-\widetilde{X}_0)^2\, \dd S\, \dd x\\
& +\int_\Omega \int_{Y^c \cup Y^e} \sigma |\nabla_x(u_{-1}-\widetilde{u}_{-1}) + \nabla_y(u_0-\widetilde{u}_0)|^2\, \dd y\, \dd x\\
&+ \int_{\Omega} \int_{\Gamma} \big(S_M([u_0], X_0) - S_M([\widetilde{u}_0], \widetilde{X}_0)\big)\,[u_{0}-\widetilde{u}_{0}]\, \dd S\, \dd x\\
&- \int_{\Omega} \int_{\Gamma} \big(f([u_0], X_0) - f([\widetilde{u}_0], \widetilde{X}_0)\big)\,[X_0-\widetilde{X}_0]\, \dd S\, \dd x=0.
\end{align*}
Using the Lipschitz continuity of $S_M$ and $f$, the coercivity of $\sigma$, and applying the Grönwall inequality, we see that $[u_0-\widetilde{u}_0]=0$ and $X_0-\widetilde{X}_0=0$ for $(x,y)\in \Omega \times \Gamma$. Then the last identity yields $\int_{\Omega} \int_{Y^c \cup Y^e}|\nabla_x(u_{-1}-\widetilde{u}_{-1}) + \nabla_y(u_0-\widetilde{u}_0)|^2\, \dd y \, \dd x=0$, and thus, due to the periodicity of $u_0, \widetilde{u}_0$ in $y$, we have $\nabla_x(u_{-1}-\widetilde{u}_{-1})=0$ a.e. in $\Omega$, and $\nabla_y(u_0-\widetilde{u}_0)=0$ for a.a. $y\in Y^c \cup Y^e$. Since $u_{-1}=\widetilde{u}_{-1}=0$ on $\partial \Omega$, $u_{-1}=\widetilde{u}_{-1}$ for almost all $x\in \Omega$. Finally, since $[u_0-\widetilde{u}_0]=0$ for a.a. $y\in \Gamma$, $u_0(t,x,\cdot)\sim \widetilde{u}_0(t,x,\cdot)$ as elements of the quotient space $(H^1(Y^c) \times H_{\mathrm{per}}^1(Y^e))/\mathbb R$. The uniqueness is proved, and thus, the whole sequence $(\ve u_\ve, [u_\ve], X_\ve)$ converges, as $\ve \to 0$, to the unique solution $(u_{-1}, [u_0], X_0)$ of the limit problem \eqref{eq:effective-problem-coupled-strong} in the senses specified in Theorem \ref{th:main}.


\section{Numerical simulations}\label{sec:numerics}
{For a general domain $\widehat{\Omega}$ and a time-dependent pulse $\widehat{g}$, the macroscopic dimensionless model \eqref{eq:u_(-1)-gamma-1} is of the form ${\rm div}_x J_{\eff}(t,x)=0$, where 
\begin{align*}
J_\eff = \int_{Y^c \cup Y^e} \sigma (\nabla_y u_0 + \nabla_x u_{-1})\, \dd y
\end{align*}
is the macroscopic current that depends in a non-explicit way on time and the gradient of the electric potential $\nabla u_{-1}$ through $u_0$.}
In Section \ref{sec:general-algorithm} we present an algorithm for computing $(u_{-1}, u_0, X_0)$ solving \eqref{eq:u_(-1)-gamma-1}, \eqref{eq:u_0}, and \eqref{eq:X_0-gamma-1}. In Section \ref{sec:parallel-plates} we consider a specific case when the physical domain is $\widehat{\Omega}=(0,L_0)\times \mathbb R$, and a constant voltage difference $\widehat{g}$ is imposed between $\widehat x_1=0$ and $\widehat x_1=L_0$. Section \ref{sec:numerical_results} contains the simulation results for this simplified case. We focus on the effective conductivity $\sigma_\eff$ and the average membrane conductivity $\overline{S}_m$ over the cell membrane $\Gamma$. We illustrate how both quantities depend on the applied electric field and how they evolve with time. 

\subsection{General solution algorithm}
\label{sec:general-algorithm}
In order to solve \eqref{eq:u_(-1)-gamma-1} for $u_{-1}$ one needs to compute the matrices $A,B$ and the vector $C$, which in turn contain $\mathcal{M}, \chi_j^0$ and $z$, meaning that one needs to solve first the cell problems \eqref{eq:cell-M}, \eqref{eq:T-eps^(-1)-g/eps} and \eqref{eq:chi^0-gamma-1}. Solving \eqref{eq:cell-M} to obtain $\mathcal{M}$ is a stationary coupled cell problem, but \eqref{eq:T-eps^(-1)-g/eps} and \eqref{eq:chi^0-gamma-1} are nontrivial evolution problems. Specifically, they are nonlinear coupled cell problems, where problem \eqref{eq:chi^0-gamma-1} depends on the additional parameter $\tau$ (the initial time) giving rise to additional memory effects. Due to the steepness of the function $\beta$, we get abrupt changes in the jump $[u_0]$ and the degree of poration $X_0$. Algorithm~\ref{alg:complete_solution_between_plates_new} describes the steps to solve \eqref{eq:u_(-1)-gamma-1} for $u_{-1}$ as well as \eqref{eq:[w0]-gamma-1} for $[u_0]$ and \eqref{eq:X_0-gamma-1} for $X_0$ in the general case. 

\begin{algorithm}[H]
    \caption{Solving for $u_{-1}$, $u_0$, $X_0$ 
    } 
\begin{algorithmic}[1]
\State Determine $\mathcal M(y)$ by solving \eqref{eq:cell-M} in $Y^c\cup Y^e$.
\State Compute $A_{ij}$ using \eqref{def:A}.
\State At $t_0=0$, set $[z](0,x,y)=[u_0](0,x,y)=V^{\rm in}(x,y)$, $[\chi_j^0](0, 0, x,y)= -\frac{\sigma^c}{c_m} ( \nabla_y \mathcal{M}_j^c(y) + e_j) \cdot \n$, $X_0(0,x,y)=X^{\rm in}(x,y)$, and compute $B(0,0,x)$ using \eqref{def:B}.
	\For {$k=1,2,\ldots$}
        \For {$m=0,\ldots,k-1$}
            \State Use $[\chi^0_j](t_{k-1},\tau_m,x,y)$ and $X_0(t_{k-1},x,y)$ to solve for $[\chi^0_j](t_{k},\tau_m,x,y)$
            \State Compute $B_{ij}(t_k,\tau_m,x)$ using \eqref{def:B}
        \EndFor
        \State Set $\displaystyle [\chi^0_j](t_k,\tau_k,x,y) = -\frac{\sigma^c}{c_m} ( \nabla_y \mathcal{M}_j^c(y) + e_j) \cdot \n$ and set $B(t_k,\tau_k,x) = B(0,0,x)$.
        \State Given $z(t_{k-1},x,y)$, solve \eqref{eq:T-eps^(-1)-g/eps} to get $z(t_k,x,y)$, and compute $C(t_k,x)$ using \eqref{def:C}.
        \State Solve \eqref{eq:u_(-1)-gamma-1} to determine $u_{-1}(t_k, x)$.
        \State Integrate over $\tau \in [0,t_k]$ to compute $[u_0](t_k,x,y)$ from \eqref{eq:[w0]-gamma-1}.
		\State Given $X_0(t_{k-1},x,y)$ and $[u_0](t_{k-1},x,y)$, use \eqref{eq:X_0-gamma-1}
        to determine $X_0(t_k,x,y)$.
	\EndFor
\end{algorithmic} 
\label{alg:complete_solution_between_plates_new}
\end{algorithm}
In the above algorithm, it is easiest to assume a regular grid with $t_k = k \Delta t$ and $\tau_m=m \Delta t$ to be able to apply initial conditions where $t=\tau$. 
Adaptive time-stepping schemes could be applied to provide more efficient discretizations in both $t$ and $\tau$, and higher order methods such as Runge-Kutta schemes could also be used for time-stepping, but the outline in Algorithm~\ref{alg:complete_solution_between_plates_new} still applies.

\subsection{Algorithm for the case of two parallel plates} \label{sec:parallel-plates}
{The aim of this section is to compute the macroscopic conductivity of the tissue. 
To this end,
we consider a simplified test case when the physical domain is $\widehat{\Omega} = (0,L_0) \times \mathbb{R}$, and the two parallel infinite electrodes with a constant voltage difference $\widehat{g}$ are located at $\widehat{x}_1=0$ and $\widehat{x}_1=L_0$. The homogenized equations \eqref{eq:coupled-u_eff} are formulated and will be solved in dimensionless variables $(t,x)$, where $x\in \Omega=(0,1)\times \mathbb R$. We assume also zero initial condition $V^{\rm in}=0$, $X^{\rm in}=0$. This implies that $z$ solving problem \eqref{eq:T-eps^(-1)-g/eps} is identically zero for all times regardless of $X_0(t,y)$, and thus $C(t,x)=0$.
Then clearly $u_{-1}(x)=0$ is a solution to \eqref{eq:u_(-1)-gamma-1} in $\Omega$.
Moreover, if $g(x_1=0)=0$ and $g(x_1=1)=g=\mbox{Const}$, we can solve \eqref{eq:p_0} and get $p_0 = g x_1$. 
}
Thus, $\nabla (u_{-1} + p_0) = g e_1$ for all times, and we obtain from \eqref{eq:u_0} that
\begin{equation}
    [u_0](t,y) = g \int_0^t [\chi_1^0](t,\tau,y)\, \dd\tau.
\label{eq:w0_jump_between_plates}
\end{equation}
To compute the solution $([u_0],X_0)$, we need then only to solve the cell problem \eqref{eq:cell-M} once with $j=1$ to obtain the initial condition for $[\chi_1^0]$, and then repeatedly solve the cell problem for $[\chi_1^0]$ and the ODEs for $X_0$. {The dimensionless parameters are given in Table \ref{tab:dimensionless_coeffs}. Since $\nabla_x (u_{-1}+p_0)$ is constant, we can factor it out in \eqref{eq:u_(-1)-gamma-1} and define the effective conductivity tensor in [S/m]:
\begin{equation}
\label{eq:sigma_eff_tensor}
    \Sigma_{\text{eff}}(t) = \Sigma_0 ({A} + \mathcal{B}(t)) = \Sigma_0(A + \int_0^t B(t,\tau) \, \dd \tau),
\end{equation}
where $\Sigma_0=\widehat{\sigma_c}/\sigma_c=1.9$ S/m.
}
Notice that $p_0$ satisfies the homogenized equation $\di{A \nabla p_0}=0$ and thus, $A$ is present in \eqref{eq:sigma_eff_def} even though $u_{-1}=0$.

In two dimensions, let $Y=(0,1)^2$, and let
$\mathcal M_1\in H_{\mathrm{per}}^1(Y)$ be the solution of
\eqref{eq:cell-M}, normalized by
\(
\int_Y \mathcal M_1(y)\,\mathrm{d}y=0.
\)
Assume that both phases are invariant under
reflection about the line $y_2=\tfrac12$.
The reflected function
\(
\mathcal M_1(y_1,1-y_2)
\)
satisfies the same periodic cell problem and the same zero-mean
normalization as $\mathcal M_1$. By uniqueness,
\(
\mathcal M_1(y_1,1-y_2)
=
\mathcal M_1(y_1,y_2).
\)
Consequently,
\[
(\partial_{y_2}\mathcal M_1)(y_1,1-y_2)
=
-(\partial_{y_2}\mathcal M_1)(y_1,y_2).
\]
And thus
\(
A_{21}
=
\int_Y
\sigma(y)\,\partial_{y_2}\mathcal M_1(y)\,\mathrm{d}y
=
0.
\)
Similarly, one can check that $\chi_1^0(t, \tau, y_1, y_2)=\chi_1^0(t, \tau, y_1, -y_2)$ for symmetric periodicity cells, and therefore we get that $\chi_1^0$ is even in $y_2$, and its derivative is odd. Note that the initial condition for $\chi_1^0$ is symmetric in $y_2$. By \eqref{def:B} we conclude $B_{21} = 0$. Thus, since $\nabla(u_{-1}+p_0)=g\, e_1$, we will compute the first element of $\Sigma_\eff$ \eqref{eq:sigma_eff_tensor}:
\begin{equation}
\label{eq:sigma_eff_def}
    \sigma_{\text{eff}}(t) = \Sigma_0 (A_{11} + \mathcal{B}_{11}(t)) := \Sigma_0(A_{11} + \int_0^t B_{11}(t,\tau) \, \dd \tau).
\end{equation}
Thus, the complexity of the problem is reduced, and we obtain the following algorithm for computing \eqref{eq:sigma_eff_def}:
\begin{algorithm}[H]
    \caption{Solving for $u_0$, $X_0$ and $\mathcal{B}$ between two plates} 
\begin{algorithmic}[1]
\State Determine $\mathcal M_1(y)$ by solving \eqref{eq:cell-M} in $Y^c\cup Y^e$.
\State Compute $A_{11}$ using \eqref{def:A}.
\State At $t_0=0$: $[\chi_1^0](0, 0,y)= -\frac{\sigma^c}{c_m} ( \nabla_y \mathcal{M}_1^c(y) + e_1) \cdot \n$, $X_0(0,y)=X^{\rm in}$, and $\mathcal{B}(0)=0$.
	\For {$k=1,2,\ldots$}
        \For {$m=0,\ldots,k-1$}
            \State Use $[\chi^0_1](t_{k-1},\tau_m,y)$ and $X_0(t_{k-1},y)$ to solve for $[\chi^0_1](t_{k},\tau_m,y)$
            \State Compute $B_{11}(t_k,\tau_m)$
        \EndFor
        \State Set $\displaystyle [\chi^0_1](t_k,\tau_k,y) = -\frac{\sigma^c}{c_m} ( \nabla_y \mathcal{M}_1^c(y) + e_1) \cdot \n$ and compute $B(t_k,\tau_k)$. 
        \State Integrate over $\tau \in [0,t_k]$ to compute $[u_0](t_k,y)$ and $\mathcal{B}(t_k)$.
		\State Given $X_0(t_{k-1},y)$ and $[u_0](t_{k-1},y)$, use \eqref{eq:X_0-gamma-1}
        to determine $X_0(t_k,y)$.
	\EndFor
\end{algorithmic} 
\label{alg:complete_solution_between_plates}
\end{algorithm}
\subsubsection{Computation of $A$}
In order to solve the cell problem for $\mathcal{M}$ \eqref{eq:cell-M}, we write it in component form using the definition of the jump to get
\begin{align}
    &\diy{\sigma ( \nabla_y \mathcal{M}_j(y) + e_j)}  = 0, \, & y \in \,\, &Y^c\cup Y^e,  \nonumber\\
    & \sigma^c ( \nabla_y \mathcal{M}_j^c(y) + e_j) \cdot \n = \sigma^e ( \nabla_y \mathcal{M}_j^e(y) + e_j) \cdot \n =: I_m(y)  , \, & y \in \,\, &\Gamma, \label{eq:cell-M_comp}\\
    & \mathcal{M}_j^c(y) - \mathcal{M}_j^e(y)  =: V(y), \, & y \in \,\, &\Gamma.\nonumber
\end{align}
We formulate the problem for a general jump $V(y)$, noting that $V=0$ for the cell problem \eqref{eq:cell-M}.

The variational formulation is derived similarly to \cite{kuchta2021solving}, multiplying by test functions $v^c \in H^1(Y^c)$ and $v^e \in H_{\rm per}^1(Y^e)$ and integrating the divergence by parts to yield the problem: find $\mathcal{M}_j^c \in H^1(Y^c)$ and $\mathcal{M}_j^e \in H_{\rm per}^1(Y^e)$ such that
\begin{align*}
    &\int_{Y^c} \sigma^c ( \nabla_y \mathcal{M}_j^c(y) + e_j) \cdot \nabla v^c \, \dd y - \int_\Gamma \sigma^c ( \nabla_y \mathcal{M}_j^c(y) + e_j) \cdot \n \,  v^c \, \dd S = 0, \\
    &\int_{Y^e} \sigma^e ( \nabla_y \mathcal{M}_j^e(y) + e_j) \cdot \nabla v^e \, \dd y + \int_\Gamma \sigma^e ( \nabla_y \mathcal{M}_j^e(y) + e_j) \cdot \n \, v^e \, \dd S = 0.
\end{align*}
Recognizing the membrane current density $I_m$ defined in \eqref{eq:cell-M_comp} in the boundary integrals we identify three unknowns $\mathcal{M}_j^c \in H^1(Y^c),\ \mathcal{M}_j^e \in H_{\rm per}^1(Y^e)$ and $I_m \in H^{-1/2}(\Gamma)$. A suitable normalization should be chosen to ensure the uniqueness. Now multiply the third equation of \eqref{eq:cell-M_comp} by a test function $j_m \in H^{-1/2}(\Gamma)$ to get
\begin{align*}
    \langle \mathcal{M}_j^c,  j_m\rangle  - \langle \mathcal{M}_j^e,  j_m\rangle = \langle V,  j_m\rangle.
\end{align*}
The surface integrals below are understood as $H^{-1/2}(\Gamma)-H^{1/2}(\Gamma)$ duality pairings when appropriate.
The multi-dimensional primal form of \eqref{eq:cell-M_comp} therefore reads: find $\mathcal{M}_j^c \in H^1(Y^c), \mathcal{M}_j^e \in H_{\rm per}^1(Y^e)$ and $I_m \in H^{-1/2}(\Gamma)$ such that
\begin{align}
    &\int_{Y^c} \sigma^c \nabla_y \mathcal{M}_j^c(y) \cdot \nabla v^c \, \dd y = \int_\Gamma I_m v^c \, \dd S - \int_{Y^c} \sigma^c e_j \cdot \nabla v^c \, \dd y, \nonumber \\
    &\int_{Y^e} \sigma^e \nabla_y \mathcal{M}_j^e(y) \cdot \nabla v^e \, \dd y = - \int_\Gamma I_m v^e \, \dd S - \int_{Y^e} \sigma^e e_j \cdot \nabla v^e \, \dd y, \label{eq: var_form_M} \\
    &\int_\Gamma \mathcal{M}_j^c j_m \, \dd S - \int_\Gamma \mathcal{M}_j^e j_m \, \dd S = \int_\Gamma V j_m \, \dd S, \nonumber
\end{align}
for all $v^c \in H^1(Y^c), v^e \in H_{\rm per}^1(Y^e)$ and $j_m \in H^{-1/2}(\Gamma)$. In the vector form:
\begin{align}
a(u,v) = L(v), 
\label{eq:cell-Matrix_problem_M}
\end{align}
where for $u = (\mathcal{M}_j^c, \mathcal{M}_j^e,I_m )$ and $v= (v^c, v^e, j_m)$
\begin{align} \label{eq: a_mat L_mat L_mat}
    a(u,v) = \begin{pmatrix}
        \int_{Y^c} \sigma^c \nabla_y \mathcal{M}_j^c \cdot \nabla v^c \, \dd y -\int_\Gamma I_m v^c \, \dd S \\[2mm]
        \int_{Y^e} \sigma^e \nabla_y \mathcal{M}_j^e \cdot \nabla v^e \, \dd y + \int_\Gamma I_m v^e \, \dd S \\[2mm]
        \int_\Gamma \mathcal{M}_j^c j_m \, \dd S - \int_\Gamma \mathcal{M}_j^e j_m \, \dd S
    \end{pmatrix}, \quad
    L(v) = \begin{pmatrix}
                - \int_{Y^c} \sigma^c e_j \cdot \nabla v^c \, \dd y \\[2mm]
                - \int_{Y^e} \sigma^e e_j \cdot \nabla v^e \, \dd y \\[2mm]
                \int_\Gamma V j_m \, \dd S
    \end{pmatrix}.
\end{align}
Inspired by the code used in \cite{kuchta2021solving} for solving a transmission problem with continuous flux and a given jump, 
we  solve \eqref{eq: var_form_M} implementing \eqref{eq: a_mat L_mat L_mat} in block form using the PETSc solver to obtain $u = (\mathcal{M}_j^c, \mathcal{M}_j^e,I_m )$. 
Using $\mathcal{M}_j^c$ and $\mathcal{M}_j^e$, $A_{ij}$ is computed from \eqref{def:A}.

\subsubsection{Computation of $\mathcal{B}(t)$}

Following Algorithm \ref{alg:complete_solution_between_plates}, in order to compute the time-dependent effective coefficient $\mathcal{B}(t)$ in \eqref{eq:sigma_eff_def}, we need to solve the nonlinear coupled cell problem \eqref{eq:chi^0-gamma-1} to obtain $\chi^0$ for each time point $t_k$ and each discrete $\tau_m \in [0,t_k]$ to be used in the numerical integration in \eqref{eq:sigma_eff_def}. 

To this end, we rewrite \eqref{eq:chi^0-gamma-1} using the definition of the jump
\begin{align}
    &\diy{\sigma  \nabla_y \chi_j^0(t,\tau,y) }  = 0, \, & y \in \,\, &Y^c\cup Y^e,  \nonumber\\
    & \sigma^c  \nabla_y \chi_j^{0c}(t,\tau,y)\cdot \n = \sigma^e  \nabla_y \chi_j^{0e}(t,\tau,y)\cdot \n  =: I_m(t,\tau,y)  , \, & y \in \,\, &\Gamma, \label{eq:chi_problem_Im}\\
    & J(y,t,\tau) := [\chi_j^0](t,\tau,y) = \chi_j^{0c}(t,\tau,y) - \chi_j^{0e}(t,\tau,y),\, & y \in \,\, &\Gamma.\nonumber \\
    & c_m \partial_t[\chi_j^0](t,\tau,y) + S_m(X_0(t,y))[\chi_j^0](t,\tau,y) = -I_m(t,\tau,y). \, & y \in \,\, &\Gamma, \nonumber
\end{align}
Note that the first three equations define a mapping of the jump $J$ to the boundary flux $I_m$ for each fixed $t$ and $\tau$, corresponding to the periodic-cell analog of operator $\mathcal{L}_\ve$ in \eqref{def:L_eps}.
The variational formulation for these three equations is obtained analogously to \eqref{eq: var_form_M}, yielding the problem
\begin{align*}
     a_{\chi}(u,v) = L_{\chi}(v),
\end{align*}
 where for $u = (\chi_j^{0c}, \chi_j^{0e},I_m )$ and $v= (v^c, v^e, j_m)$,
\begin{align*}
    a_{\chi}(u,v) = \begin{pmatrix}
        \int_{Y^c} \sigma^c \nabla_y \chi_j^{0c} \cdot \nabla v^c \, \dd y -\int_\Gamma I_m v^c \, \dd S \\[2mm]
        \int_{Y^e} \sigma^e \nabla_y \chi_j^{0e} \cdot \nabla v^e \, \dd y + \int_\Gamma I_m v^e \, \dd S \\[2mm]
        \int_\Gamma \chi_j^{0c} j_m \, \dd S - \int_\Gamma \chi_j^{0e} j_m \, \dd S
\end{pmatrix}, \quad
L_{\chi}(v) = \begin{pmatrix}
        0 \\[2mm]
         0 \\[2mm]
        \int_\Gamma J j_m \, \dd S
        \end{pmatrix}
\end{align*}
The code to solve the above cell problem is analogous to the code to solve \eqref{eq:cell-Matrix_problem_M}-\eqref{eq: a_mat L_mat L_mat} for the cell problem for $\mathcal{M}$. Note that we are solving this linear system of equations for each $t_k$ and $\tau_m \in [0,t_k]$, therefore repeatedly solving the same problem up to $J$ on the right-hand side.  Since $a_{\chi}(u,v)$ is not changing, we create the assembled block matrix before starting the for loops over time and use LU decomposition to speed up the solver.\\
The solver gives us $I_m$, which we can use to solve the ODE given by the fourth equation in \eqref{eq:chi_problem_Im}
\begin{align*}
    \partial_t[\chi_j^0](t,\tau,y) = - \frac{1}{c_m} \Big(I_m(t,\tau,y) + S_m(X_0(t,y))[\chi_j^0](t,\tau,y) \Big).
\end{align*}
To solve the above ODE we use a semi-implicit Euler method and compute 
\begin{align*}
    [\chi_j^0](t_{k+1},\tau,y_s) = \frac{1}{\big(1+ \frac{\Delta t}{c_m}S_m(X_0(t_k,y_s))\big)} \Big( [\chi_j^0](t_{k},\tau,y_s) - \frac{\Delta t}{c_m} I_m(t_k,\tau,y_s) \Big)
\end{align*}
for $\Delta t = t_{k+1}- t_k$ and each mesh point $y_s \in \Gamma$.
Similarly, we solve the ODE for $X_0$ using a semi-implicit Euler method for time step size $\Delta t = t_{k+1}- t_k$: 
\begin{align*}
    X_0(t_{k+1},y_s) = \frac{X_0(t_k,y_s) + \frac{\Delta t}{ \tau_{\text{max}}} \beta([u_0](t_{k+1},y_s))}{(1+ \frac{\Delta t}{\tau_{\text{max}}})},
\end{align*}
with
\begin{align*} \tau_{\text{max}} =
\begin{cases}
\tau_{\text{ep}} \,\, \mbox{if} \,\, \beta([u_0](t_k,y_s)) \geq X_0(t_k,y_s),\\ 
\tau_{\text{res}} \,\, \mbox{if} \,\, \beta([u_0](t_k,y_s)) < X_0(t_k,y_s),
\end{cases}
\end{align*}
for each mesh point $y_s \in \Gamma$. With $\chi_j^{0c}$ and $\chi_j^{0e}$ for all times we can compute $B_{ij}$ by \eqref{def:B} and therefore also $[u_0]$, the jump of $u_0$, by \eqref{eq:w0_jump_between_plates} and the effective conductivity $\sigma_{\text{eff}}$ by \eqref{eq:sigma_eff_def}.

\subsection{Numerical results}
\label{sec:numerical_results}
The simulations were performed using the FEM package FEniCSx 0.9 \cite{BarattaEtal2023,AlnaesEtal2014}. We use gmsh (version 4.14.1) to generate the mesh. 
We perform discretization using FEM on the periodicity cell $Y=Y^c \cup \Gamma \cup Y^e$ containing an intra- and an extracellular domain, where $Y^c$ is a disc of radius $r=0.25$ centred at $(0.5,0.5)$, and the periodic cell is $Y=(0,1)^2$. To implement the periodic boundary conditions, we use the Multi Point Constraint package \verb|dolfinx_mpc| in FEniCSx 0.9. 


For the simulations we follow \cite{kavian2014classical} and choose 
\begin{align*}
    \beta(\xi) = \frac{1+\tanh(k_{\text{ep}}(|\xi | - V_{\text{rev}}))}{2}, \quad k_{\text{ep}} = 40 \, \mbox{V}^{-1}, \quad V_{\text{rev}} = 1.5 \mbox{V}.
\end{align*}
To compare our numerical results with the ones in \cite{kavian2014classical}, we translate the values of the boundary data $g/\ve$ into the strength of the electric field $|E|$ measured in V/cm.  To this end, we choose a typical size of the periodicity cell to be $\ell=2\cdot 10^{-4}$ m and the size of a sample to be $L_0=10^{-2}$ m.
\begin{align*}
|E|=\frac{g {U_0}}{\ve L_0}=\frac{g {U_0}}{\ell}=\frac{g\cdot 10^4}{2} \, \mbox{V/m} = \frac{g\cdot 10^2}{2}\, \mbox{V/cm}.
\end{align*}
The values of the electric field for different values of $g$ are given in Table \ref{table:g vs E}.
\begin{table}[htp]
\centering
\begin{tabular}{c|c|c|c|c|c|c|c|c}
        $g$ [-] & 0 & 2 & 4& 6 & 10 & 50 & 75 & 100  \\ \hline
        $|E|$ [V/cm] & 0 & 100& 200& 300& 500 &  2500 & 3750 & 5000
\end{tabular}
\caption{Relation between the applied pulse and the electric field.}
\label{table:g vs E}
\end{table}

In Figure \ref{fig:B} we present the simulations of $B(t,\tau)$ given by  \eqref{def:B}. In the parallel-plate configuration, the longitudinal effective conductivity is given by \eqref{eq:sigma_eff_def}, so only $B_{11}$ is required. Figure~\ref{fig:Bttau} depicts $B_{11}$ over time $t$ and $\tau$ (see the computational domain in time Figure~\ref{fig:2}) for pulse magnitude $|E| = 500 \, V/cm$. Note that $B_{11}$ converges to zero as $t$ grows, but since \eqref{eq:chi^0-gamma-1} is coupled with the equation for $X_0$, the rate depends on the value of $X_0$ at time $t$, and thus it depends on the applied electric field $E$, as seen in Figure \ref{fig:Bt0}. For fixed $\tau = 0$, we compare in Figure~\ref{fig:Bt0} $B_{11}$ over $t$ for different pulse magnitudes $|E| = 500 \, V/cm$ to $5000 \, V/cm$. We observe that the higher the pulse magnitude, the faster $B_{11}$ tends to zero.
\begin{figure}[H]
    \centering
    \begin{subfigure}{0.49\textwidth}
        \centering
        \includegraphics[width=\linewidth]{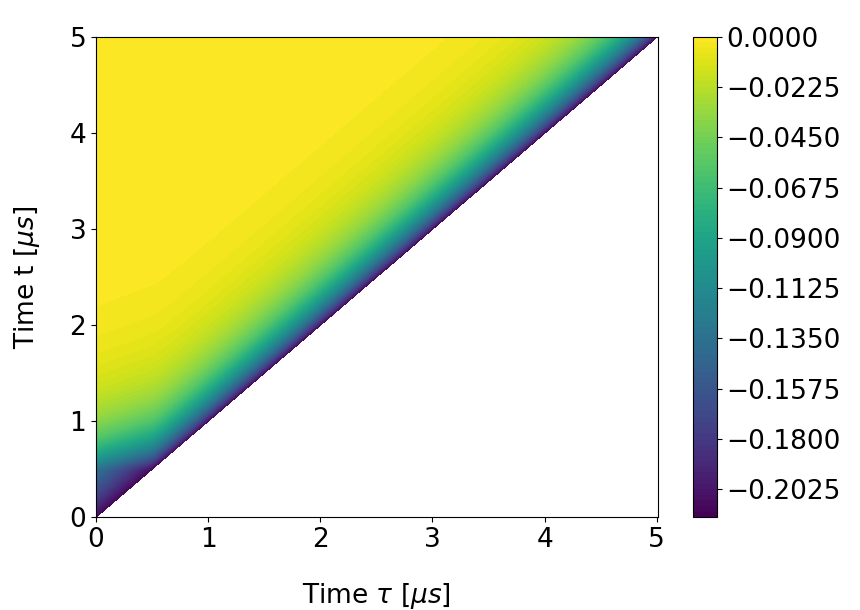}
        \caption{Coefficient $B_{11}$ over time $t$ and $\tau$, $|E| = 500$ V/cm.}
        \label{fig:Bttau}
    \end{subfigure}
    \hfill
    \begin{subfigure}{0.49\textwidth}
        \centering
        \includegraphics[width=\linewidth]{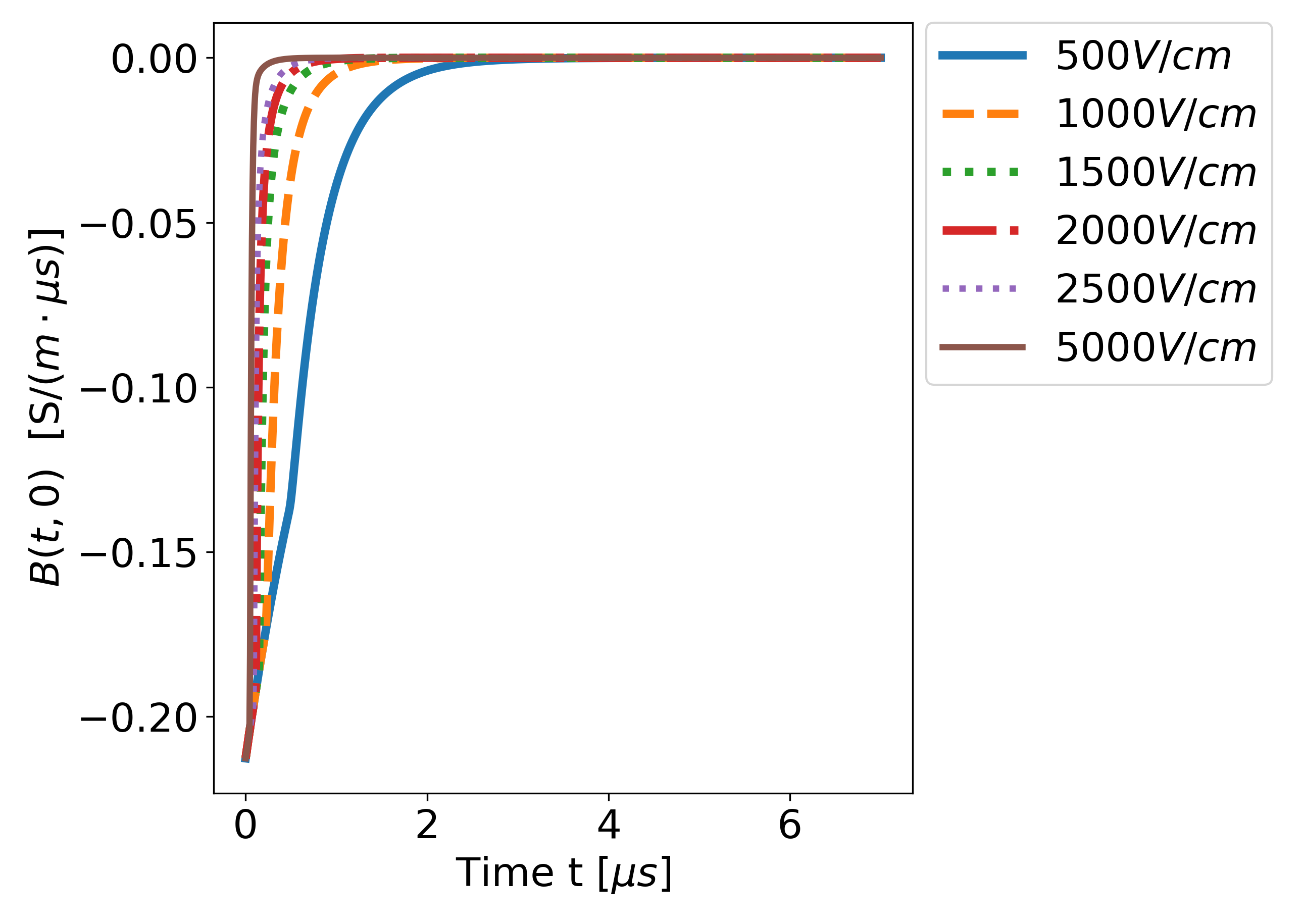}
        \caption{$B_{11}$ over time $t$ for $\tau = 0$ and different $|E|$.}
        \label{fig:Bt0}
    \end{subfigure}
    
    \caption{Simulations of $B$ given by equation \eqref{def:B}.}
    \label{fig:B}
\end{figure}

Figure~\ref{fig:effective_conductivity_over_time} depicts the effective conductivity $\sigma_{\text{eff}}$ given by \eqref{eq:sigma_eff_def} over time $t$ for different pulse magnitudes $|E| = 0$ V/cm to $5000$ V/cm. The magnified section on the right focuses on the conductivity for the electric field $|E|$ between $1500$ V/cm and $5000$  V/cm. The macroscopic model captures the non-monotone evolution of the effective conductivity, which agrees with experimental data \cite{ivorra2009electric}. 
\begin{figure}[H]
\centering
\includegraphics[width=\linewidth]{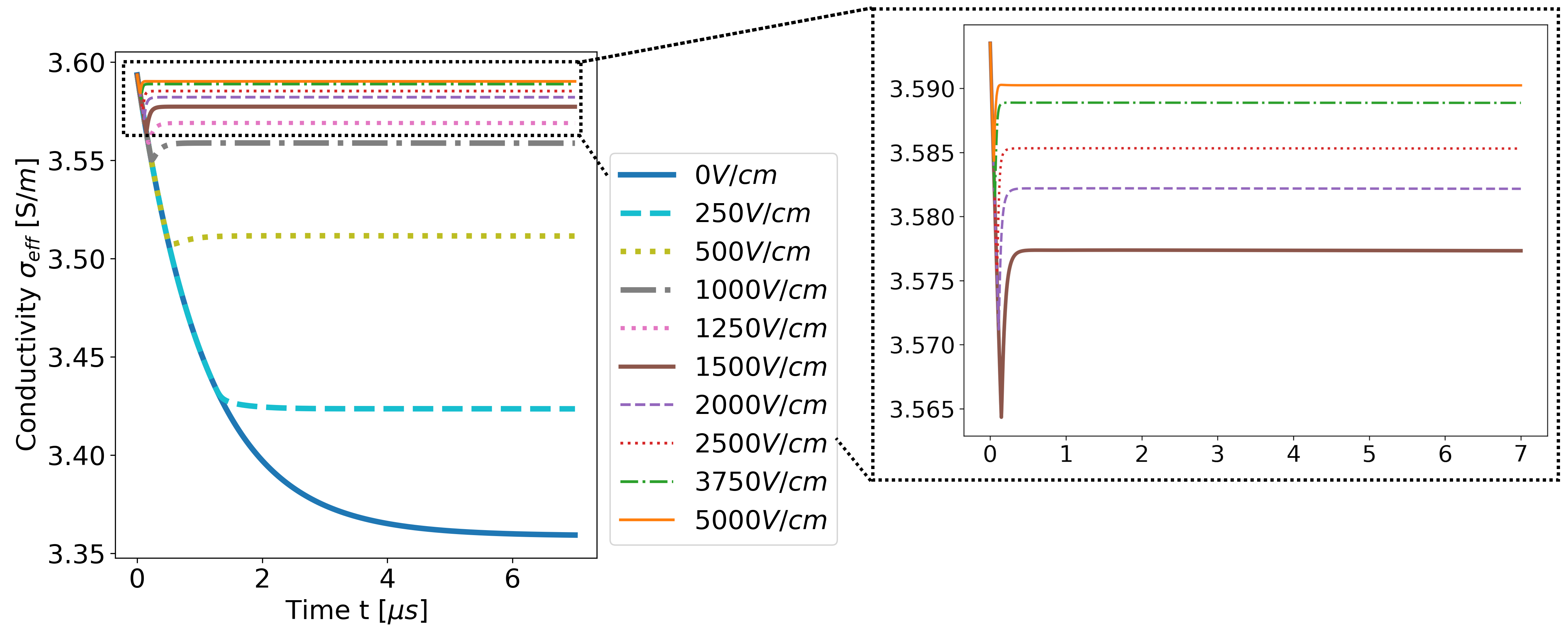}  
\caption{Evolution of the effective conductivity $\sigma_{\text{eff}}$ for different pulse magnitudes.}
\label{fig:effective_conductivity_over_time}
\end{figure}
Namely, observe a characteristic drop in conductivity in Figure \ref{fig:effective_conductivity_over_time}. The membrane is a thin lipid bilayer that separates the inside of the cell from the outside. It is poorly conducting before poration, but it is surrounded by conductive fluids (cytoplasm inside, extracellular fluid outside).
{The membrane charges up like a capacitor, temporarily storing energy.} During this charging phase, ions are pulled toward the membrane, reducing their movement in the bulk solution.
This causes a temporary drop in measured conductivity.
Once the membrane is sufficiently charged, 
some defects form allowing ions to pass through the membrane — and conductivity increases sharply. At the same time, this reduces the membrane's resistance, allowing current to flow more easily. As a result, the voltage across the membrane drops.

The value for $\sigma_{\eff}$ at $t=0$ corresponds to {purely conducting cells, when the electric potential is continuous through the membrane}, and is equal to $\Sigma_0 A_{11}$ \eqref{def:A} given in terms of the solution $\mathcal{M}$ to the cell problem $\eqref{eq:cell-M}$ independent of time. Additionally, when $g=0$ and no electroporation occurs, the limiting value of $\sigma_{\eff}$ as $t\rightarrow \infty$ {is bounded from below by the effective conductivity of a collection of fully insulated inclusions $\sigma_{\rm ins}$, mathematically corresponding to zero Neumann condition on $\Gamma$ in the cell problem.} This situation is equivalent to setting $\sigma^c=0$ in  \eqref{eq:cell-M}. Both these limiting cases are classical problems, and the values have been computed semi-analytically for periodic arrays of cylinders to high accuracy by Perrins et al. \cite{perrins1979transport}. Inserting our parameter values yields $\sigma_{\eff}\big|_{t=0} = 3.5941$ and $\sigma_{\rm ins} = 3.3587$ (the dashed line in Figure \ref{fig:sigma_eff_g0}), respectively for the two cases, which is in agreement within numerical errors with our results, as seen in Figures \ref{fig:effective_conductivity_over_time}, \ref{fig:sigma_eff_g0}, and \ref{fig:sigma_eff_over_g}.
\begin{figure}[H]
\centering
\includegraphics[width=0.5\linewidth]{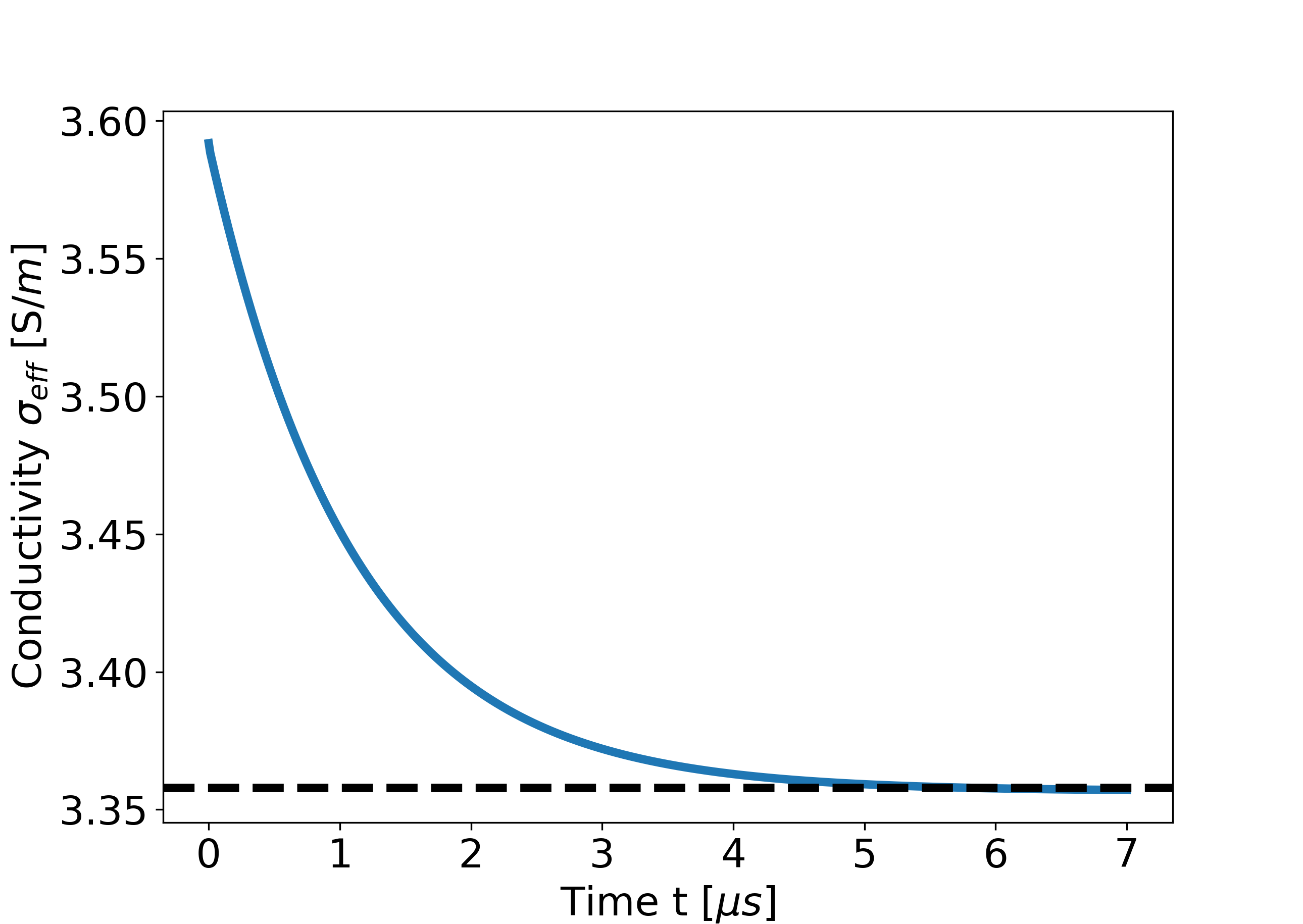}
\caption{Evolution of the effective conductivity $\sigma_{\text{eff}}$ for $|E| = 0 \, V/cm$.}
\label{fig:sigma_eff_g0}
\end{figure}
Next, we illustrate the dependence of the effective conductivity on the applied electric field. 
If we compare the values of the conductivity $\sigma_\eff$ at the final time in Figure~\ref{fig:effective_conductivity_over_time}, we observe that the value increases when increasing the pulse magnitude $|E|$. In Figure~\ref{fig:sigma_eff_over_g}, we plot the effective conductivity as a function of the pulse magnitude $|E|$. The end times are given in Table~\ref{tab:end_time}. The obtained nonlinear dependence in the form of a sigmoid function agrees with the experimental results \cite{ivorra2009electric}, \cite{corovic2013modeling}. We note that up to the pulse magnitude of about $160$ V/cm the value of $\sigma_\eff$ is almost constant, and then it increases drastically to later flatten for pulse magnitudes over $1000$ V/cm. 
\begin{figure}[H]
    \centering
\includegraphics[width=\linewidth]{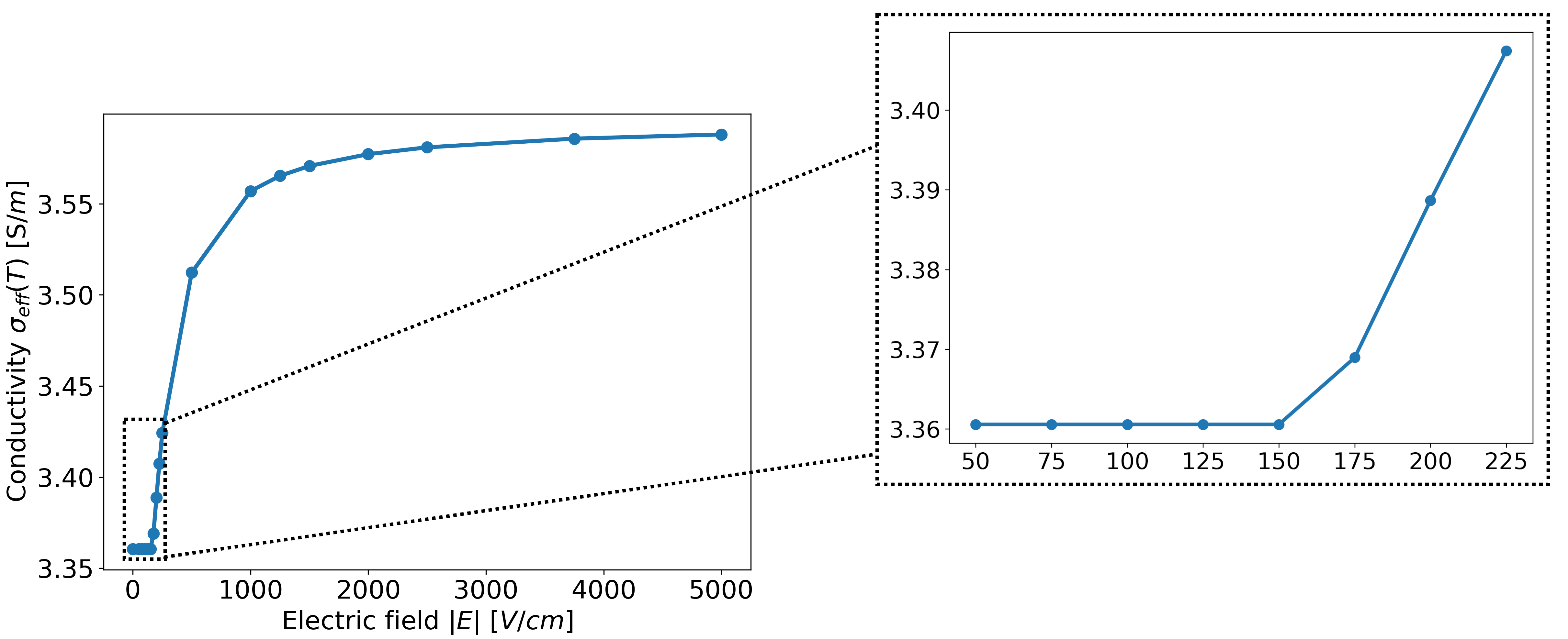}
    \caption{Dependence of $\sigma_{\eff}(T)$ on $|E|$.}
    \label{fig:sigma_eff_over_g}
\end{figure}
Note that even if the microscopic conductivity $\sigma^{c, e}$ is constant, the macroscopic one depends both on time and the magnitude of the applied electric field, due to the effects of electroporation. 

In Figure \ref{fig:convergence} we illustrate the convergence of the algorithm for $\sigma_\eff$. 
To estimate the convergence rate, we compute, for each mesh size $m$, the absolute value of the difference in effective conductivity computed using mesh size $m$ and a refined one with mesh size $m/2$. The slope of the line thus achieved in a log-log plot is then compared with the reference slopes $1$ and $2$. {We plot the error differences} for the final time $T=$5 $\mu$s with the fixed time step $0.1$ $\mu$s in a log-log plot in  Figure~\ref{fig:conv_mesh}, which shows the expected second-order convergence rate. As for the time discretization, we use the semi-implicit Euler method with the observed first order convergence. The error differences comparing $\sigma_\eff$ for time step $\Delta t$ with time step $\Delta t/1.5$ for the final time $T=$5 $\mu$s with the fixed mesh size 0.01 are plotted in Figure~\ref{fig:conv_time}, and the convergence is in agreement with theory.
\begin{figure}[htbp]
    \centering
    \begin{subfigure}{0.49\textwidth}
        \centering
        \includegraphics[width=\linewidth]{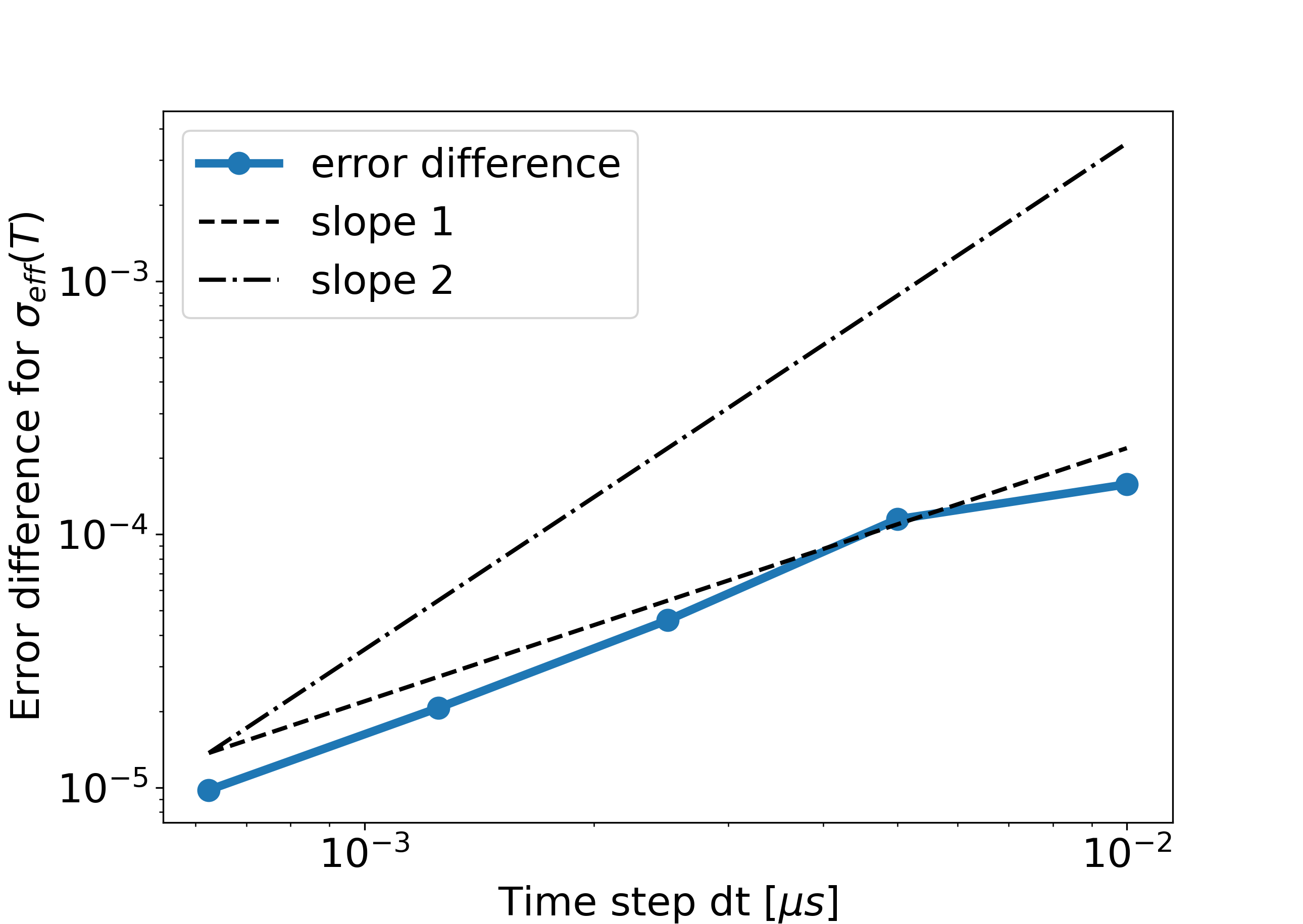}
        \caption{Convergence for $\sigma_{\eff}(T)$ as a function of time step. }
        \label{fig:conv_time}
    \end{subfigure}
    \hfill
    \begin{subfigure}{0.49\textwidth}
        \centering
        \includegraphics[width=\linewidth]{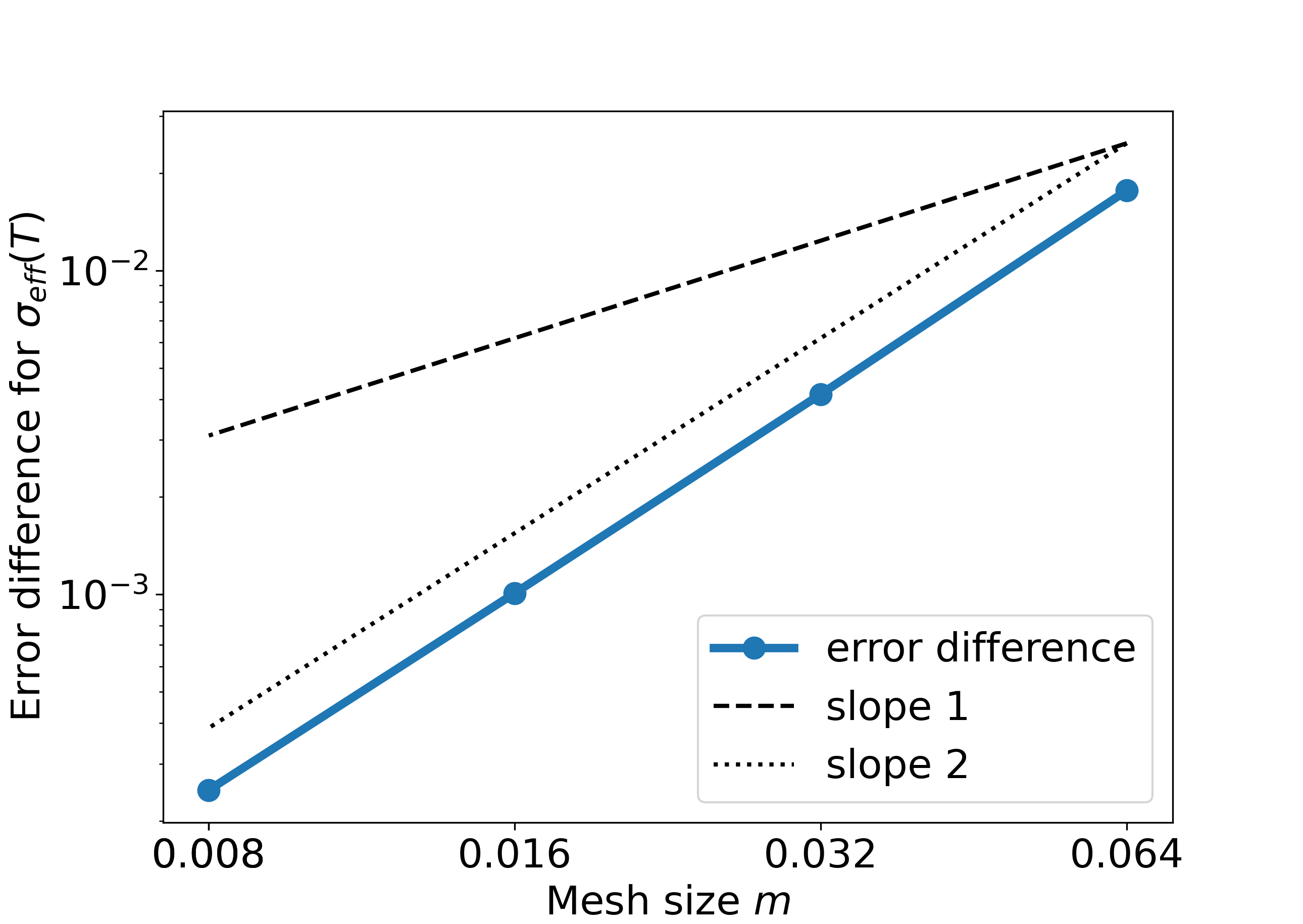}
        \caption{Convergence for $\sigma_{\eff}(T)$ as a function of mesh size.} \label{fig:conv_mesh}
    \end{subfigure}
    \caption{Numerical convergence for $\sigma_\eff$.}
    \label{fig:convergence}
\end{figure}

In Figures \ref{fig:Sm_over_g} and \ref{fig:Sm_over_time} we plot the average membrane conductivity 
\begin{align*}
    {\widehat{\overline{S_m}}(t) = \frac{1}{|\Gamma|} \int_\Gamma (\widehat{S}_{\rm L} + (\widehat{S}_{\text{ir}} - \widehat{S}_{\rm L}) X_0(\frac{\widehat t}{T_0},y)) \, \dd S_y}.
\end{align*}
as a function of applied electric field and a function of time.  
In Figure~\ref{fig:Sm_over_time} we compare the conductivity for different pulse magnitudes $|E| = 0$ V/cm to $5000$ V/cm. One can see the qualitative agreement with the average conductivity for a single cell in \cite{kavian2014classical}. 
In Figure~\ref{fig:Sm_over_g} we plot the conductivity over the pulse magnitude $|E|$, where we use the end times presented in Table~\ref{tab:end_time}, since the steady state is reached at that time.
\begin{figure}[H]
    \centering
    \begin{subfigure}{0.45\textwidth}
        \centering
        \includegraphics[width=\linewidth]{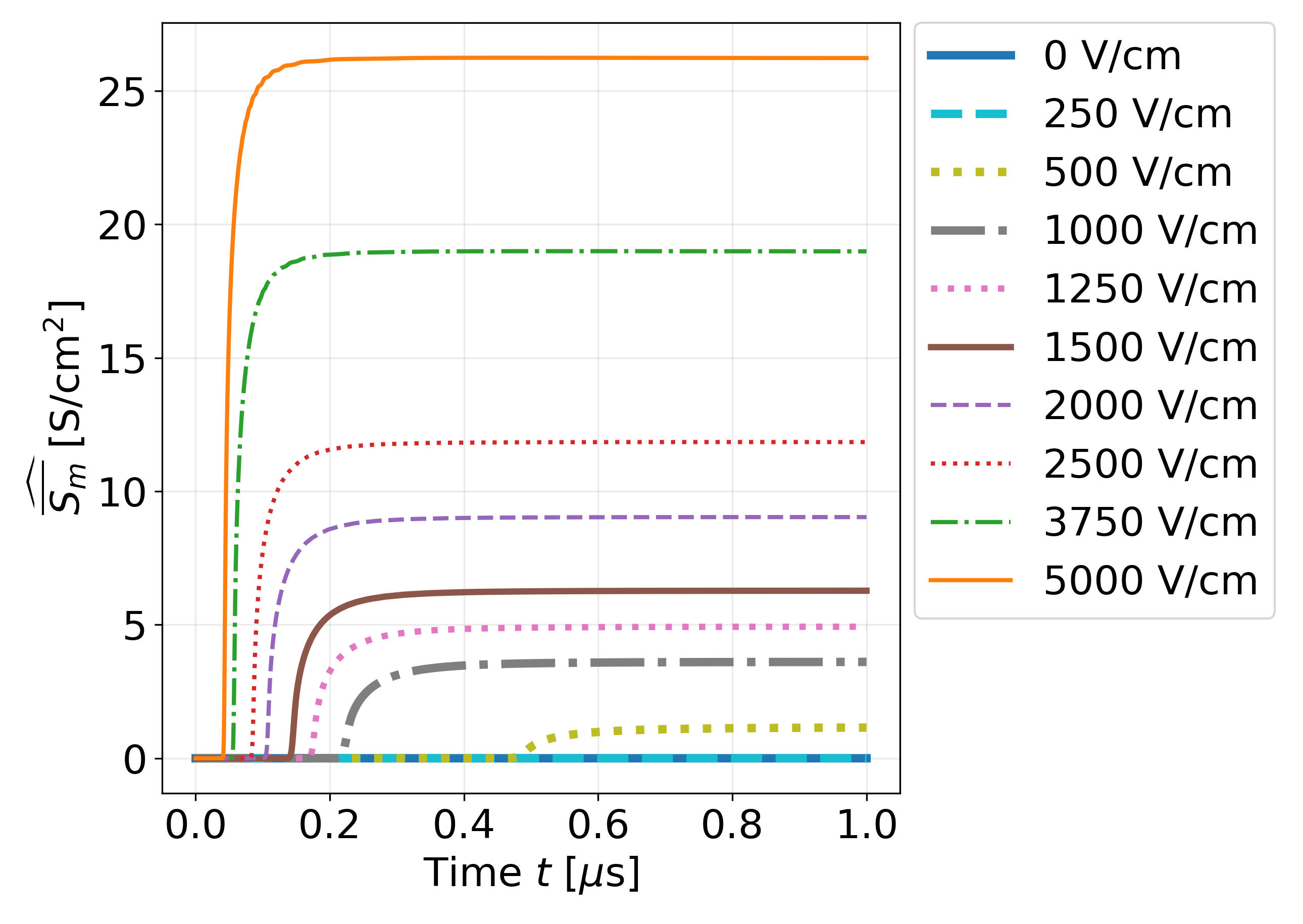}
        \caption{Evolution of the average membrane conductivity $\widehat{\overline{S_m}}$.}
        \label{fig:Sm_over_time}
    \end{subfigure}
    \begin{subfigure}{0.45\textwidth}
        \centering
        \includegraphics[width=\linewidth]{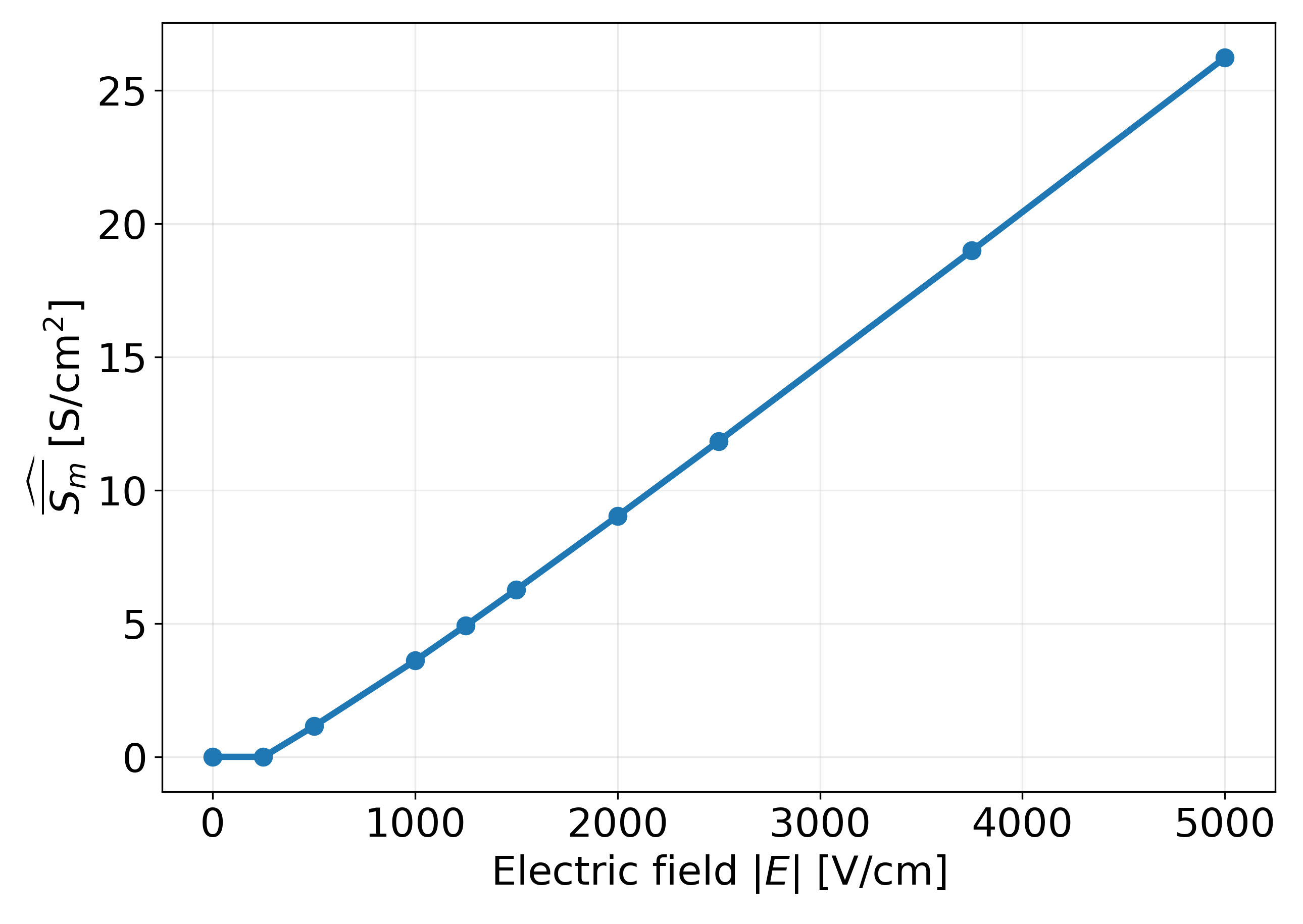}
        \caption{The average membrane conductivity $\widehat{\overline{S_m}}(T)$ over pulse magnitudes.}
        \label{fig:Sm_over_g}
    \end{subfigure}
    \hfill
    
    \caption{Average membrane conductivity $\overline{S_m}$ for $\dd t =0.0001$, end time $T=1\, \mu$s, and mesh size $0.01$.}
\end{figure}

\begin{table}[]
    \centering
    \begin{tabular}{c|c|c|c|c|c|c|c|c|c|c}
        $|E|$ [V/cm] & 0 & 250 & 500 & 1000 & 1250 & 1500 & 2000 & 2500 & 3750 & 5000  \\ \hline
        $T$ [$\mu$s] & 7 & 7 & 5 & 2.5 & 2 & 1.5 & 1.25 & 1 & 0.75 & 0.5 
    \end{tabular}
    \caption{End times of the simulations for different pulse magnitudes.}
    \label{tab:end_time}
\end{table}

In Figure \ref{fig:potential}, we compare the behavior of the membrane potential with the one for a single cell in \cite{kavian2014classical}. In Figure \ref{fig:v_fix_angle}, we plot the approximation $[u_0]$ of the jump $[u_\ve]$ on the membrane as a function of time at an angular position corresponding to the cell's pole and for pulse magnitude $|E| = 500$ V/cm. The plot qualitatively agrees with Figure 8(a) in \cite{kavian2014classical}. At the final time $T=5 \, \mu$s we compute $[u_0]$ along the interface $\Gamma$  for $|E| = 500 $ V/cm (see Figure~\ref{fig:v_fix_time}). Similar to Figure 8(b) in \cite{kavian2014classical} we obtain a smooth, bipolar variation of the transmembrane potential along the cell perimeter under an applied electric field. One side of the cell, facing the field, exhibits a pronounced positive peak, while the opposite side displays a symmetric negative peak of similar magnitude. Between these extremes, the curve transitions gradually through zero near the equator, creating a shape that resembles a distorted cosine wave due to nonlinear effects in the model.
\begin{figure}[htbp]
    \centering
    \begin{subfigure}{0.49\textwidth}
        \centering
        \includegraphics[width=\linewidth]{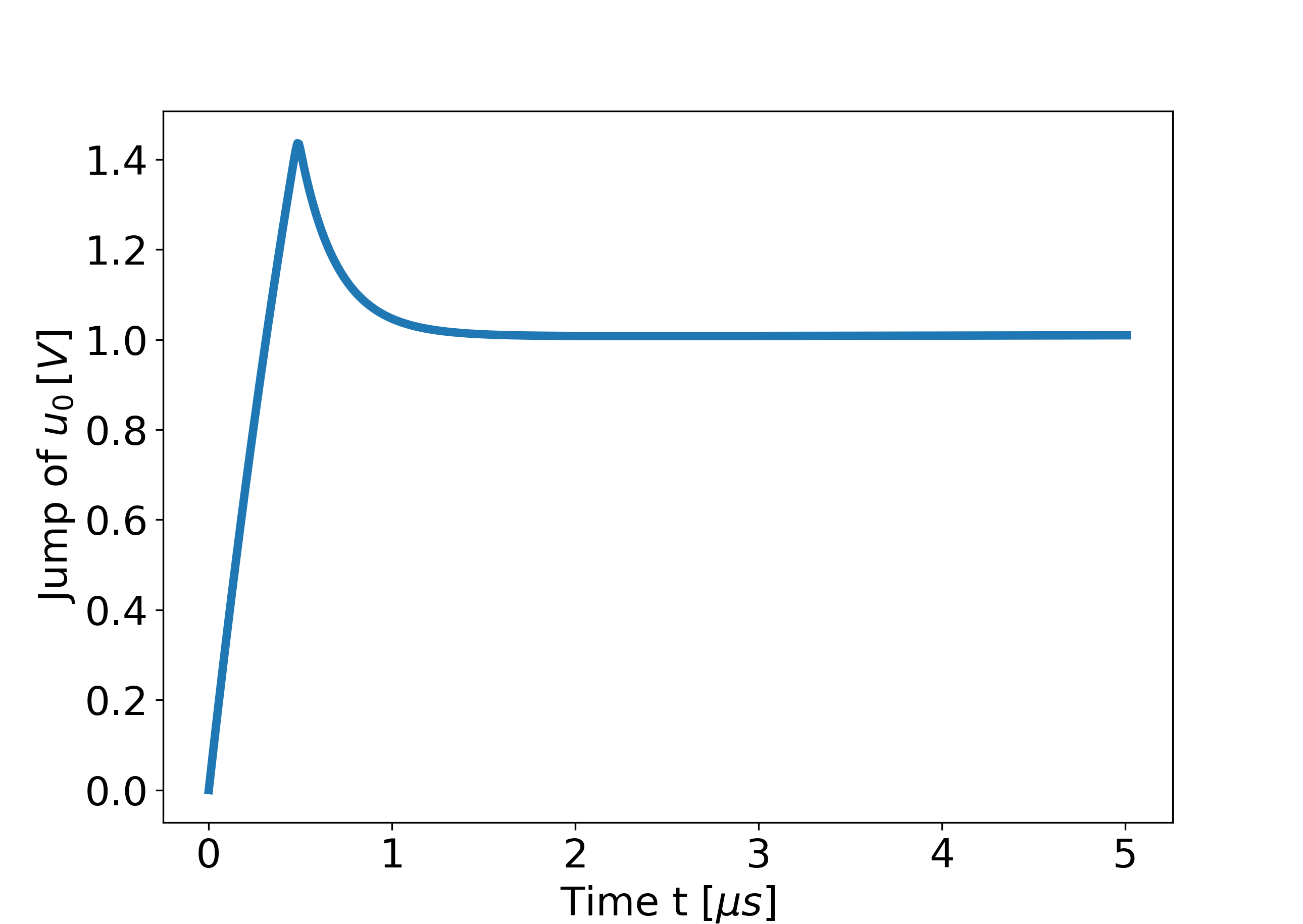}
        \caption{Evolution of $[u_0]$ at the cell's pole.}
        \label{fig:v_fix_angle}
    \end{subfigure}
    \hfill
    \begin{subfigure}{0.49\textwidth}
        \centering
        \includegraphics[width=\linewidth]{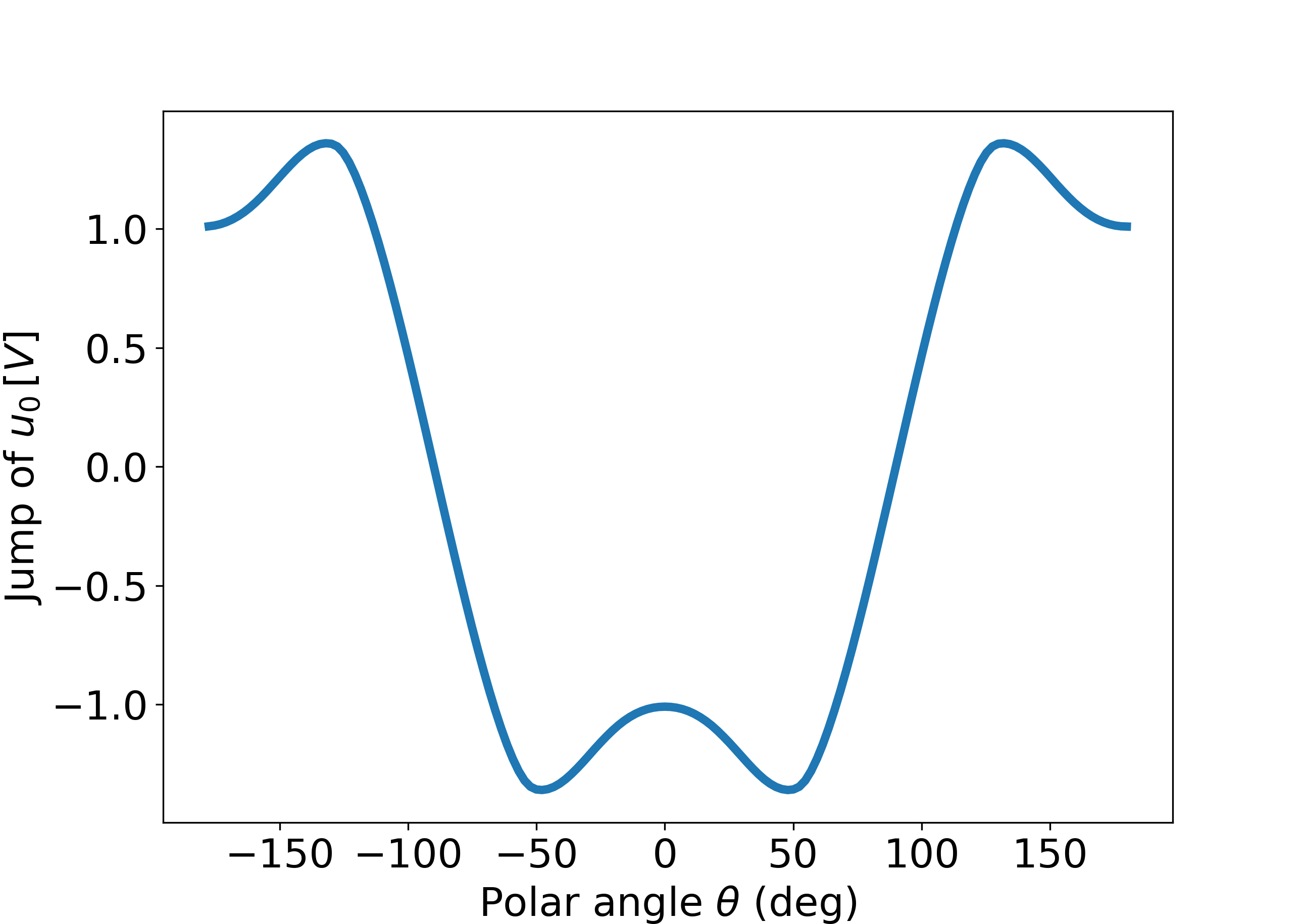}
        \caption{$[u_0]$ after $5 \mu s$ along the boundary of the cell.}
        \label{fig:v_fix_time}
    \end{subfigure}
    \caption{Simulations of the approximation $[u_0]$ of the membrane potential $[u_\ve]\big|_{\Gamma_\ve}$, $|E| = 500$ V/cm, for time step $1/120 \mu s$ and mesh size $0.01$. }
    \label{fig:potential}
\end{figure}

\section*{Conclusions}
In this work, we perform the multiscale analysis of a phenomenological electropermeabilization model introduced in \cite{kavian2014classical} for a single cell. The electric potential satisfies electrostatic equations in the extra- and intracellular domains, while the jump across the cell membrane evolves according to a nonlinear law coupled with an ordinary differential equation for the poration degree. The nonlinear term involving surface conductivity is neither monotone nor Lipschitz, which makes the analysis nontrivial. We establish existence and uniqueness of a global solution via Galerkin approximations, exploiting a one-sided Lipschitz property (a generalized monotonicity condition) along the solutions. A dimensional analysis reveals a specific scaling in the microscopic problem with respect to a small parameter $\ve$, representing the ratio between the typical cell size and the tissue sample size. We derive uniform (in $\ve$) a priori estimates, obtain formal asymptotics for the electric potential, and rigorously justify the expansion using two-scale convergence. The main analytical result is a macroscopic electropermeabilization model with memory effects, where the effective current is nonlinear in the electric field and time-dependent. The macroscopic model captures the nontrivial evolution of effective conductivity, including the characteristic drop that reflects the capacitive behavior of the lipid bilayer, in agreement with experimental data. We compute numerically the effective electric potential, surface conductivity, and tissue conductivity. Although the intracellular and extracellular bulk conductivities are constant, experimental observations \cite{ivorra2009electric, corovic2013modeling} show that tissue conductivity varies with the electric field. Our principal contribution is a rigorous derivation of an effective model with nonlinear, history-dependent current. Numerical simulations of this model exhibit a sigmoid-like dependence of effective conductivity on the applied field, which qualitatively agrees with experimental results reported in the literature.

\appendix
\section{Well-posedness of the microscopic problem}
\label{sec:existence}
In this section we prove the existence and uniqueness of a solution of \eqref{eq:orig-prob} by Galerkin approximations. Since \(\ve\) is fixed, it is enough to prove the well-posedness for the fixed domain containing finitely many intracellular components, as in the subsequent setup. After introducing the necessary notation, we proceed with the properties of the Dirichlet-to-Neumann map in Lemma \ref{lm:L}. In Lemma \ref{lm:bound-w} we prove the boundedness of $w$, the second component of the solution representing the degree of poration. Lemma \ref{lm:existence-v_w} provides the main existence result of this section. Even though the argument is quite standard, the non-local Dirichlet-to-Neumann operator and the coupled PDE-ODE system prevent a direct application of the standard results considered here. Moreover, we establish additional regularity for a weak solution, which is used when obtaining a priori estimates in Lemma \ref{lm:apriori-est}.  

A reduction to an evolution equation on a manifold was used in \cite{lions1969quelques} (Chapter 1, Section 11) for proving an existence result by Galerkin approximations for nonlinear monotone parabolic and hyperbolic problems. Our situation is, however, different: the quadratic function $(S_0+S_1 w) v$ is neither one-sided Lipschitz nor monotone, but becomes such along the solutions. Namely, for a local-in-time solution, we will prove a uniform $L^\infty$-bound (see Lemma \ref{lm:bound-w}) which will ensure the Lipschitz property of the nonlinearity in the finite-dimensional case. In the infinite-dimensional case, when proving the uniqueness of a solution, the different norms are not equivalent any more, which results in a generalized monotonicity property \eqref{eq:est_generalized_monotonicity_existence}, similar to the one considered in \cite{liu2011existence}. Note in particular the factor depending on $\|v_2\|_{L^4(\Gamma)}$. Since this coefficient is integrable in time for a solution of our problem, we can use the Grönwall inequality and conclude the uniqueness. 

Throughout this section, we denote by $Y\subset \mathbb R^3$ a bounded Lipschitz domain. Let $Y_j^c$, $j=1, \ldots, N$, be simply connected bounded Lipschitz domains with pairwise disjoint closures, compactly contained in $Y$. Assume that $\Gamma_j = \partial Y_j^c$ is connected, and set
\[
Y^c = \cup_{j=1}^N Y_j^c, \quad
\Gamma = \cup_{j=1}^N \Gamma_j, \quad
Y^e = Y\setminus \overline{Y^c}.
\]
We assume that $Y^e$ is connected. 
The standard scalar product in $L^2(\Gamma)$ is denoted by $(\cdot, \cdot)_{L^2(\Gamma)}$ and the dual pairing by $\langle \cdot, \cdot \rangle:= \langle \cdot, \cdot \rangle_{H^{-1/2}(\Gamma), H^{1/2}(\Gamma)}$. 
Let us introduce a Dirichlet-to-Neumann-type operator $\mathcal{L} : H^{1/2}(\Gamma) \to H^{-1/2}(\Gamma)$ by setting, for each  $\phi\in H^{1/2}(\Gamma)$, \begin{align}
\label{def:L}
    \langle \mathcal{L} v , \phi \rangle_{H^{-1/2}(\Gamma), H^{1/2}(\Gamma)} = \int_{Y^c \cup Y^e} \sigma \nabla z^{(v)} \cdot \nabla z^{(\phi)} \, \dd x,
\end{align}
where $z^{(q)}=(z_i^{(q)}, z_e^{(q)}) \in H^1(Y^c) \times H^1(Y^e)$, for a given jump $q \in H^{1/2}(\Gamma)$ ($q=v, \phi$), solves the following problem:
\begin{align}
\label{eq:microscopic-prob_stationary}
\begin{alignedat}{2}
-\di{\sigma \nabla z^{(q)}}  &=0 &\quad& \mbox{in}\,\, Y^c \cup Y^e, \\
[\sigma \nabla z^{(q)}  \cdot \n] &= 0 &\quad&\mbox{on}\,\, \Gamma,\\
[z^{(q)}]&= q
&\quad&\mbox{on}\,\,\Gamma,
\\
z^{(q)}   &= 0&\quad&\mbox{on}\,\, \partial Y.
\end{alignedat}
\end{align}
The goal of this section is to prove the well-posedness of the following Cauchy problem on $\Gamma$ (see Lemma \ref{lm:existence-v_w}):
\begin{align}
    \label{eq:abstract-prob}
    \begin{split}
        &\partial_t v + (\mathcal{L} + S_0\, {\rm Id})v + S_1\, w\, v = p,\\
        &\partial_t w = f(v,w),\\
        &v(t=0)=v_0, \,\, w(t=0)=w_0.
    \end{split}
\end{align}
The operator $\mathcal{L}$ in \eqref{eq:abstract-prob} corresponds to $\frac{\ve}{c_m}\mathcal{L}_\ve$ and $p$ to $-\frac{1}{c_m}\sigma^c \nabla p_\ve \cdot n$ in \eqref{eq:prob-on-Gamma_eps}, and $c_m=1$ is used throughout this section. 
The next lemma gives the main properties of the transmission Dirichlet-to-Neumann operator $\mathcal{L}$.
\begin{lemma}
\label{lm:L}
The operator $\mathcal{L}$ defined by \eqref{def:L} is symmetric, non-negative, and satisfies the following properties: 
\begin{itemize}
\item[(i)] For any $\alpha > 0$, there exists $\Lambda >0$ depending on $\alpha$ such that, for any $v \in H^{1/2}(\Gamma)$,
\[
\langle \mathcal{L} v , v\rangle_{H^{-1/2}(\Gamma), H^{1/2}(\Gamma)} + \alpha \|v\|_{L^2(\Gamma)}^2\ge \Lambda \|v\|^2_{H^{1/2}(\Gamma)}.
\]
\item[(ii)] 
$
|\langle \mathcal{L} v , \varphi\rangle_{H^{-1/2}(\Gamma), H^{1/2}(\Gamma)}|
\leq C \| v \|_{H^{1/2}(\Gamma)} \| \varphi\|_{H^{1/2}(\Gamma)}$, for any $v, \varphi \in H^{1/2}(\Gamma)$.
\item[(iii)] 
$
\|\mathcal{L} v\|_{H^{-1/2}(\Gamma)}
\le C\| v \|_{H^{1/2}(\Gamma)}$, for any $v \in H^{1/2}(\Gamma)$.
\end{itemize}
\end{lemma}
\begin{proof}
(i) Denote by $\Gamma_j$ the $j$th connected component of $\Gamma$ so that $\Gamma = \cup_{j=1}^N \Gamma_j$, and assume for simplicity that $Y_j^c$ is simply connected. Then the number of connected components of $Y^c$ and $\Gamma$ coincide.
The kernel of $\mathcal{L}$ is ${\rm ker} \mathcal{L} = {\rm span} \{\mathbbm{1}_{\Gamma_1}, \ldots, \mathbbm{1}_{\Gamma_N} \}$, where the characteristic function $\mathbbm{1}_{\Gamma_j}$ is equal to $1$ on $\Gamma_j$ and zero otherwise. 
Note first that there exists $\lambda > 0$ such that, for any $v\in H^{1/2}(\Gamma)$,
\begin{align*}
\langle \mathcal{L} v , v \rangle_{H^{-1/2}(\Gamma), H^{1/2}(\Gamma)}
\ge \lambda \inf_{\kappa \in {\rm ker} \mathcal{L}} \|v- \kappa\|_{H^{1/2}(\Gamma)}^2.
\end{align*}
Indeed, let $z^{(v)}=(z_i^{(v)},z_e^{(v)})$ be the solution of the transmission
problem \eqref{eq:microscopic-prob_stationary} with jump $q = v$. 
For each connected component $Y_j^c$, $j=1, \ldots, N$, of the intracellular domain, define
\[
c_j=
\frac1{|Y_j^c|}
\int_{Y_j^c} z_i^{(v)}\,dx,
\quad
\kappa=\sum_{j=1}^N c_j\mathbbm{1}_{\Gamma_j}\in {\rm ker} \mathcal{L}.
\]
By the trace theorem and the Poincaré inequality,
\begin{align*}
\|v-\kappa\|_{H^{1/2}(\Gamma)}
&\leq
C
\left(
\sum_{j=1}^N
\|z_i^{(v)}-c_j\|_{H^1(Y_j^c)}^2
\right)^{1/2}
+
C\|z_e^{(v)}\|_{H^1(Y^e)}\\
& \leq C
\|\nabla z^{(v)}\|_{L^2(Y^c\cup Y^e)}.
\end{align*}
Combining this estimate with the ellipticity of $\sigma$, we obtain
\begin{align}
\label{Gaarding}
\langle \mathcal L v,v\rangle_{H^{-1/2}(\Gamma), H^{1/2}(\Gamma)}
\geq
\lambda
\inf_{\kappa\in {\rm ker} \, \mathcal{L}}
\|v-\kappa\|_{H^{1/2}(\Gamma)}^2
\end{align}
for some $\lambda>0$.
Suppose now that (i) does not hold for $v\neq 0$. Then, there exists $\alpha>0$ such that for any $\Lambda$, there exists $v_\Lambda$ satisfying
\[
\langle \mathcal{L} v_\Lambda , v_\Lambda\rangle_{H^{-1/2}(\Gamma), H^{1/2}(\Gamma)} + \alpha \|v_\Lambda\|_{L^2(\Gamma)}^2 < \Lambda \|v_\Lambda\|_{H^{1/2}(\Gamma)}^2.
\]
Thus, there exists a sequence $\{v_n\}$ such that $\|v_n\|_{H^{1/2}(\Gamma)} = 1$ and
\[
\langle \mathcal{L} v_n , v_n\rangle_{H^{-1/2}(\Gamma), H^{1/2}(\Gamma)} + \alpha \|v_n\|_{L^2(\Gamma)}^2 < \frac{1}{n} \to 0, \quad n \to \infty.
\]
Since $\mathcal{L}$ is non-negative and $\alpha>0$, both terms go to zero:
\begin{align}
\label{eq:coerc_L-1}
\langle \mathcal{L} v_n , v_n\rangle_{H^{-1/2}(\Gamma), H^{1/2}(\Gamma)} \to 0, \quad  \|v_n\|_{L^2(\Gamma)}^2 \to 0, \quad n \to \infty.
\end{align}
By \eqref{Gaarding} and \eqref{eq:coerc_L-1}, there exists a sequence $k_n$ such that $\|v_n- k_n\|_{H^{1/2}(\Gamma)} \to 0$, $n\to \infty$. Since $\|v_n\|_{L^2(\Gamma)} \to 0$, we have $\|k_n\|_{L^2(\Gamma)} \to 0$. The finite number of connected components in $\Gamma$ gives a finite-dimensional ${\rm ker} \mathcal{L}$, and thus we can use the equivalence of norms to get
\begin{align*}
\|k_n\|_{H^{1/2}(\Gamma)} \le C_N \|k_n\|_{L^2(\Gamma)} \to 0, \quad n\to \infty.
\end{align*}
Finally, 
\begin{align*}
\|v_n\|_{H^{1/2}(\Gamma)} \le
\|v_n - k_n\|_{H^{1/2}(\Gamma)}
+ \|k_n\|_{H^{1/2}(\Gamma)} \to 0, \quad n\to \infty,
\end{align*}
which contradicts the assumption $\|v_n\|_{H^{1/2}(\Gamma)}=1$.\\[5mm]
\medskip
\noindent
(ii) 
By the definition of $\mathcal{L}$ \eqref{def:L} and the boundedness of $\sigma$,
\begin{align*}
|\langle \mathcal{L} v , \varphi\rangle_{H^{-1/2}(\Gamma), H^{1/2}(\Gamma)}|
\leq 
C \| \nabla z^{(v)} \|_{L^2(Y^c \cup Y^e)} \|\nabla z^{(\varphi)}\|_{L^2(Y^c \cup Y^e)}.
\end{align*}
By \eqref{eq:microscopic-prob_stationary}, $z^{(q)}$ for $q= v, \varphi$, minimizes the energy among all admissible functions with a given jump $q$:
\begin{align*}
z^{(q)} = \arg \min_{\zeta \in V^{(q)}} 
\int_{Y^c \cup Y^e} \sigma \nabla \zeta \cdot \nabla \zeta \, \dd x,
\end{align*}
where $V^{(q)}=\{\zeta=(\zeta_i, \zeta_e) \in H^1(Y^c) \times H^1(Y^e): \,\, [\zeta]_\Gamma = q, \,\, \zeta_e\big|_{\partial Y}=0\}$. Thus, for any $\zeta^{(q)} \in V^{(q)}$, we have
\begin{align*}
\underline{\sigma} \| \nabla z^{(q)} \|_{L^2(Y^c \cup Y^e)}^2 \le 
\overline{\sigma} \|\nabla \zeta^{(q)} \|_{L^2(Y^c \cup Y^e)}^2,
\end{align*}
so it is enough to construct one lifting satisfying (ii).
By the trace theorem, for any $q\in H^{1/2}(\Gamma)$, there exists $\widetilde{\zeta}^{(q)} \in H^1(Y^c)$ such that $\widetilde{\zeta}^{(q)}\big|_\Gamma=q$ and
\[
\|\widetilde{\zeta}^{(q)}\|_{H^1(Y^c)} \le C \|q\|_{H^{1/2}(\Gamma)}, \quad 
\]
Taking $\zeta^{(q)}= \widetilde{\zeta}^{(q)}$ in $Y^c$ and zero in $Y^e$ gives $\zeta^{(q)} \in V^{(q)}$, and estimate (ii) is proved.
\noindent
(iii) is a direct consequence of (ii). Lemma \ref{lm:L} is proved.
\end{proof}

We summarize the assumptions on $p, f$, and the initial conditions below.
\begin{itemize}
    \item[(A1)] 
    \(\displaystyle 
    f(v, w)=\max \left(\frac{\beta(v)-w}{\tau_{\rm ep}}\, ,\,
\frac{\beta(v)-w}{\tau_{\rm res}}\right),
\)
and $\beta$ is even, nondecreasing {on $\mathbb{R}_+$}, Lipschitz continuous, and such that $
0 \leq \beta(\xi) \leq 1$, $\beta(\xi) \to 1$ as $|\xi| \to \infty$. The constants $S_0, S_1$ are positive.
    
    \item[(A2)] $p\in L^2(0,T; H^{-1/2}(\Gamma))$.
    \item[(A2')]$\partial_t p \in L^2(0,T; H^{-1/2}(\Gamma))$.
    \item[(A3)] $v_0 \in L^2(\Gamma)$, $w_0\in H^{1/2}(\Gamma) \cap L^\infty(\Gamma)$, and $0\le w_0\le 1$.
    \item[(A3')] $v_0\in H^{1/2}(\Gamma)$.
\end{itemize}
The following lemma providing a uniform bound for $w$, resulting from the comparison principle for ordinary differential equations is important for proving the well-posedness of the microscopic problem and the homogenization procedure.
\begin{lemma}
\label{lm:bound-w}
Given $b\in C([0,T]; L^2(\Gamma))$ such that $0\le b \le 1$ for every $t\in[0,T]$ and almost every
$x\in\Gamma$, assume that $w_0\in L^2(\Gamma)$. Then there exists a unique solution $w \in C^1([0,T]; L^2(\Gamma))$ of
\begin{align*}
    &\partial_t w = \max \left(\frac{b-w}{ \tau_{\rm ep}}\, ,\,
    \frac{b-w}{\tau_{\rm res}}\right), \quad t>0,\\
    &w(t=0)=w_0,
\end{align*}
where $0 < \tau_{\rm ep} < \tau_{\rm res}$.
If in addition $0\le w_0 \le 1$ for almost all $x\in \Gamma$, then $0\le w\le 1$ for all $t\in[0,T]$ and almost all $x\in \Gamma$, and
\[
\|w(t)-w(s)\|_{L^\infty(\Gamma)}
\le\frac{|t-s|}{\tau_{\rm ep}},
\qquad s,t\in[0,T].
\]
In particular,
$w\in C^{0,1}([0,T];L^\infty(\Gamma))$.
\end{lemma}
\begin{proof}
Since $F(t,w)=\max \left(\frac{b-w}{ \tau_{\rm ep}}\, ,\,
\frac{b-w}{\tau_{\rm res}}\right): [0,T]\times L^2(\Gamma) \to L^2(\Gamma)$ satisfies 
\[
\|F(t,u)-F(s, z)\|_{L^2(\Gamma)}
\le 
\frac{1}{\tau_{\rm ep}} (\|b(t, \cdot)-b(s, \cdot)\|_{L^2(\Gamma)} + \|u-z\|_{L^2(\Gamma)}),
\]
$F$ is continuous in $t$ and Lipschitz continuous in $w$. Existence and uniqueness of a solution therefore follow from the Cauchy-Lipschitz theorem (Theorem 3.8.1 and the remark for the Banach-space version at the end of Chapter 3, \cite{ciarlet2025linear}). 

Let us prove that the uniform estimate $0\le w_0\le 1$ is preserved. Denote by $w^-(t,x)=\max\{-w, 0\}$ the negative part of $w$ and consider
\begin{align*}
    m(t):=\frac{1}{2}\|w^-(t, \cdot)\|_{L^2(\Gamma)}^2 = \frac{1}{2}\int_\Gamma |\max(-w, 0)|^2\, \dd S.
\end{align*}
Since $w_0\ge 0$, $m(0)=0$. Let us compute the Fréchet derivative of $\Phi(w)=\frac{1}{2}\|w^-\|_{L^2(\Gamma)}^2$. To this end, note that for $r\in \mathbb{R}$
\begin{align*}
\displaystyle
\phi(r) = \frac{1}{2}|r^-|^2
= \begin{cases}
\displaystyle
    \frac{1}{2}r^2, \,\, r<0\\ 0, \,\, r\ge 0
\end{cases}
\Rightarrow \quad
\phi'(r) = \begin{cases}
    r, \,\, r<0\\ 0, \,\, r\ge 0
\end{cases}
= \min(r,0)= - r^-.
\end{align*}
By the chain rule for integral functionals,
\begin{align*}
D\Phi(w)h = \int_\Gamma \phi'(w) h\, \dd S
= -\int_\Gamma w^- h\, \dd S.
\end{align*}
Since $w \in C^1((0,T); L^2(\Gamma))$, the chain rule gives
\begin{align*}
m'(t)= D\Phi(w) \partial_t w = - \int_\Gamma w^- \partial_t w \, \dd S
= \int_{\{x: \, w < 0\}} w F(t, w)\, \dd S.
\end{align*}
For $w<0$, $b-w \geq 0$, and we have $m'(t)\le 0$. Since $m(0)=0$ and $m(t) \ge 0$ for all $t$, the non-positive derivative implies that $m(t)=0$ for all $t$. Indeed, since $m(t)$ is absolutely continuous,
\begin{align*}
m(t)=m(0)+\int_0^t m'(s)\, \dd {s} \le 0,
\end{align*}
at the same time as $m(t)\ge 0$ by definition. Thus $m(t)=0$ and, consequently, $w\ge 0$.

In a similar fashion, we prove that $w\le 1$ for all $t>0$.

Consequently, $|b-w|\le 1$, and the equation yields
$|\partial_t w|\le\tau_{\rm ep}^{-1}$ almost everywhere
in space-time. Thus, $w$ has a representative which is absolutely continuous in time for almost every \(x\)
\[
w(t,x)-w(s,x)
=
\int_s^t F(r,w(r))(x)\,\dd r,
\]
and
\[
\|w(t)-w(s)\|_{L^\infty(\Gamma)}
\le\frac{|t-s|}{\tau_{\rm ep}},
\]
which proves the final assertion.
\end{proof}

The main result of this section is the well-posedness of \eqref{eq:abstract-prob}.
\begin{lemma}
\label{lm:existence-v_w}
    Under the hypotheses (A1)-(A3), there exists a unique solution $(v,w)$ of \eqref{eq:abstract-prob}
    \begin{align*}
    &v \in C([0,T]; L^2(\Gamma))\cap L^2(0,T; H^{1/2}(\Gamma)),\\
    &w \in C([0,T];H^{1/2}(\Gamma))\cap C^1([0,T]; L^2(\Gamma)).
    \end{align*}
    {If in addition (A2'), (A3') are satisfied, $v\in L^\infty(0,T; H^{1/2}(\Gamma))\cap W^{1, \infty}(0,T; H^{-1/2}(\Gamma))$ and $\partial_t v \in L^2(0,T; L^2(\Gamma))$.} 
\end{lemma}
\begin{proof}
Let $\{e_j\}_{j=1}^\infty$ be an orthonormal basis in $L^2(\Gamma)$ that is also a basis in $H^{1/2}(\Gamma)$. Observe that $e_j \in L^4(\Gamma)$ due to the continuous embedding $H^{1/2}(\Gamma)$ into $L^4(\Gamma)$ for a two-dimensional surface $\Gamma$. Define also $P_n:H^{-1/2}(\Gamma) \to H_n$, with $H_n ={\rm span} \{e_1, \ldots, e_n\} \subset H^{1/2}(\Gamma)$, by
\begin{align*}
    P_n y = \sum_{i=1}^n \langle y, e_i \rangle e_i, \quad y\in H^{-1/2}(\Gamma).
\end{align*}
One can choose a basis $\{e_j\}_{j=1}^\infty$ such that $\|P_n \phi\|_{H^{1/2}(\Gamma)} \le \|\phi\|_{H^{1/2}(\Gamma)}$ uniformly in $n$.
For any $v\in H^{1/2}(\Gamma), \phi \in H_n$ we have
\begin{align*}
    \langle P_n \mathcal{L} v, \phi \rangle
    = (P_n\mathcal{L} v, \phi)_{L^2(\Gamma)}
    = \langle \mathcal{L} v, \phi \rangle.
\end{align*}
We introduce Galerkin approximations
\(
v_n(t,x) = \sum_{i=1}^{n} v_{in}(t) e_i(x)
\)
and seek $v_n: [0,T] \to H_n$ and $w_n: [0,T] \to L^2(\Gamma)$ satisfying
\begin{equation}
\label{finite-dimensional}
\begin{aligned}
\displaystyle 
& {\partial_t} v_n = -P_n (\mathcal{L}+  S_0\, {\rm Id}) v_n - P_n S_1\, w_n\, v_n + P_n p,\\[2mm]
\displaystyle
&{\partial_t} w_n = f(v_n,w_n),\\
& v_n(0) = v_{0,n}= P_n v_0, \quad w_n(0) = w_{0,n}= w_0.
\end{aligned}
\end{equation}
\begin{lemma}
\label{lm:est-fin-dim}
    Let the assumptions of Lemma \ref{lm:existence-v_w} be satisfied. Then there exists a unique global solution of \eqref{finite-dimensional} 
    \begin{align*}
    &v_n \in C([0,T]; L^2(\Gamma))\cap L^2(0,T; H^{1/2}(\Gamma)), \,\, \partial_t v_n \in  L^2(0,T; H^{-1/2}(\Gamma)),\\
    &w_n \in C([0,T];H^{1/2}(\Gamma))\cap C^1(0,T; L^2(\Gamma))
    \end{align*}
    such that for all $n$ and some constant $C$, independent of $n$, the energy estimate holds:
    \begin{align}
    \label{estimates-finite-dim}
    \begin{split}
    &\sup_{0\le t\le T} \|v_n\|_{L^2(\Gamma)}^2 + \|v_n\|_{L^2(0,T; H^{1/2}(\Gamma))}^2
    +  \|\partial_t v_n\|_{L^2(0,T; H^{-1/2}(\Gamma))}^2\\
    &+ \sup_{0\le t\le T}\|w_n(t, \cdot)\|_{H^{1/2}(\Gamma)}^2
    + \sup_{0\le t\le T}\|\partial_t w_n\|_{L^2(\Gamma)}^2
    + \|\partial_t w_n\|_{L^2(0,T; H^{1/2}(\Gamma))}^2 \\ 
    &\le C(1+\|p\|_{L^2(0,T; H^{-1/2}(\Gamma))}^2 + \|v_0\|_{L^2(\Gamma)}^2 + \|w_0\|_{H^{1/2}(\Gamma)}^2).
    \end{split}
    \end{align}
{Moreover, if in addition (A2') and (A3') hold, then $v_n \in L^\infty(0,T; H^{1/2}(\Gamma))$, $\partial_t v_n \in L^2(0,T; L^2(\Gamma))$ and 
\begin{align}
\label{eq:est-partial_t v_m}
\begin{split}
&\sup_{0\le t\le T}\|v_n\|_{H^{1/2}(\Gamma)}^2
+ \|\partial_t v_n\|_{L^2(0,T; L^2(\Gamma))}^2\\ 
&\le C(\sup_{0\le t \le T}\|p\|_{H^{-1/2}(\Gamma)}^2 + 
\|\partial_t p\|_{L^2(0,T; H^{-1/2}(\Gamma))}^2
+\|v_0\|_{H^{1/2}(\Gamma)}^2).
\end{split}
\end{align}
}
\end{lemma}
\begin{proof}
System \eqref{finite-dimensional} has a local solution by the Carathéodory existence theorem in $\mathbb{R}^n \times L^2(\Gamma)$ (see Theorem 2.6, \cite{holly2000applications}). 

{Since the initial condition is bounded $0\le w_0\le 1$, the local solution $w_n$ satisfies the uniform estimate $|w_n|\le C$ and $|\partial_t w_n|\leq 1/\tau_{\rm ep}$  for $t\in [0,T]$ (Lemma~\ref{lm:bound-w})}. 
To prove that the solution does not blow up in finite time, we multiply \eqref{finite-dimensional} by $(v_n, w_n)$, integrate over $\Gamma$, and use Lemmas \ref{lm:L}, \ref{lm:bound-w} to obtain
\begin{align}
\frac{1}{2}\frac{d}{dt}\|v_n\|_{L^2(\Gamma)}^2 + \Lambda  \|v_n\|_{H^{1/2}(\Gamma)}^2&\le K_0 + K_1\|v_n\|_{L^2(\Gamma)}^2 + K_2\|p\|_{H^{-1/2}(\Gamma)}^2. \label{est-v_m}
\end{align}
The Grönwall inequality yields
\begin{align*}
\sup_{t\in [0,T]}(\|v_n(t,\cdot)\|_{L^2(\Gamma)}^2 + \|w_n(t,\cdot)\|_{L^2(\Gamma)}^2) \le C (1 + \|v_0\|_{L^2(\Gamma)}^2 + \|w_0\|_{L^2(\Gamma)}^2 + \|p\|_{L^2(0,T; H^{-1/2}(\Gamma))}^2).
\end{align*}
Thus, $(v_n, w_n)$ is a global solution. Further, integrating \eqref{est-v_m} with respect to $t$, we obtain an integral in time estimate for $\|v_n\|_{H^{1/2}(\Gamma)}$. 

To derive an estimate for the time derivative $\partial_t v_n$, we choose a test function $\phi\in H^{1/2}(\Gamma)$ such that $\|\phi\|_{H^{1/2}(\Gamma)}\le 1$ and decompose it as $\phi=\phi_n + \phi_{\perp}$, where 
\begin{align*}
\phi_n \in {\rm span}\{e_k\}_{k=1}^n, \quad
(\phi_\perp, e_k)=0 \,\,\mbox{for}\,\, k=1, \ldots, n. 
\end{align*}
Thus, 
\begin{align*}
&\langle \partial_t v_n (t, \cdot), \phi\rangle = \langle \partial_t v_n (t, \cdot), \phi_n\rangle 
= -(P_n (\mathcal{L}+ S_0{\rm Id}) v_n, \phi_n)_{L^2(\Gamma)} - S_1 (P_n v_n w_n, \phi_n)_{L^2(\Gamma)} + (P_n p, \phi_n)_{L^2(\Gamma)}.
\end{align*}
Using the growth condition on $\mathcal{L}$ (Lemma \ref{lm:L}), the global Lipschitz continuity of $f$ and the uniform boundedness of $w_n$, we obtain
\begin{align*}
&|\langle \partial_t v_n, \phi\rangle| \le C(\|v_n\|_{H^{1/2}(\Gamma)} + \|v_n\|_{L^2(\Gamma)} + \|p\|_{H^{-1/2}(\Gamma)}), \quad \text{a.a.}\, t\in[0,T].
\end{align*}
Thus, 
\begin{align*}
\begin{split}
&\int_0^T \|\partial_t v_n\|_{H^{-1/2}(\Gamma)}^2dt \le C(\|v_n\|_{L^2(0,T; H^{1/2}(\Gamma))}^2 + \|p\|_{L^2(0,T;H^{-1/2}(\Gamma))}^2),
\end{split}
\end{align*}
and the estimate for $\partial_t v_n$ follows from the estimate for $\|v_n \|_{L^2(0,T; H^{1/2}(\Gamma))}$ proved above.

Next, suppose (A2') and (A3') hold. Multiplying \eqref{finite-dimensional} by $\partial_t v_n$ we arrive at
\begin{align*}
\|\partial_t v_n\|_{L^2(\Gamma)}^2
&= - \frac{1}{2}\frac{\dd}{\dd t}(P_n (\mathcal{L}+S_0{\rm Id})v_n \,, \, v_n)_{L^2(\Gamma)} 
- S_1(P_n v_n w_n \,,\, \partial_t v_n)_{L^2(\Gamma)}
+ \frac{\dd}{\dd t}\langle p, v_n\rangle - \langle \partial_t p, v_n\rangle.
\end{align*}
Integrating in time, using Lemma \ref{lm:L}, and applying the Cauchy inequality with a parameter in a similar way as above, we obtain \eqref{eq:est-partial_t v_m}.

Since there are no spatial derivatives in the equation for $w_n$, the regularity in $x$ is inherited from the initial condition and $v_n$. We first estimate the Sobolev-Slobodeckij seminorm
\begin{align*}
|w_n|_{H^{1/2}(\Gamma)}^2 = \int_\Gamma \int_\Gamma \frac{|w_n(t,x)-w_n(t,y)|^2}{|x-y|^3}\dd S_x \dd S_y.
\end{align*}
Denote
\begin{align*}
\delta w_n = w_n(t, x) - w_n(t, y), \quad x,y\in \Gamma.
\end{align*}
Then $\delta w_n$ satisfies the equation
\begin{align}
\label{eq:aux_est_w-1}
\partial_t \delta w_n = f(v_n(t,x), w_n(t,x)) - f(v_n(t,y), w_n(t,y)).
\end{align}
Multiplying \eqref{eq:aux_est_w-1} by $\delta w_n$ and writing $\delta w_n = \delta \beta - (\delta \beta - \delta w_n)$ with 
\(
\delta \beta = \beta(v_n(t,x)) - \beta(v_n(t,y)), 
\)
we obtain
\begin{align*}
\frac{1}{2}\partial_t |\delta w_n|^2
& = (f(v_n(t,x), w_n(t,x)) - f(v_n(t,y), w_n(t,y))) \delta \beta\\ 
& \quad - (f(v_n(t,x), w_n(t,x)) - f(v_n(t,y), w_n(t,y))) (\delta \beta - \delta w_n)\\
& \le 
C|\delta \beta - \delta w_n|\, |\delta \beta|
-\frac{1}{\tau_{\rm res}} |\delta \beta - \delta w_n|^2 \\
& \le
C |\delta \beta|^2 - \frac{1}{2\tau_{\rm res}} |\delta \beta - \delta w_n|^2.
\end{align*}
Using the inequality $|a-b|^2\ge \frac{1}{2}|b|^2 - |a|^2$ and rearranging the terms, we have
\begin{align*}
\partial_t |\delta w_n|^2 + 
\frac{1}{{2}\tau_{\rm res}} |\delta w_n|^2
& \le
C |\delta \beta|^2.
\end{align*}
Using the Lipschitz continuity of $\beta$ and integrating in $t$ yields
\begin{align*}
|\delta w_n|^2 \le 
e^{-t/(2\tau_{\rm res})}|w_0(x)-w_0(y)|^2
+
C \int_0^t e^{-(t-s)/(2\tau_{\rm res})} |v_n(s,x)-v_n(s,y)|^2\, \dd s.
\end{align*}
Dividing by $|x-y|^3$ and integrating over $\Gamma \times \Gamma$, we obtain
\begin{align*}
|w_n|_{H^{1/2}(\Gamma)}^2
\le 
e^{-t/(2\tau_{\rm res})}|w_0|_{H^{1/2}(\Gamma)}^2
+
C \int_0^t e^{-(t-s)/(2\tau_{\rm res})} |v_n|_{H^{1/2}(\Gamma)}^2\, \dd s.
\end{align*}
Combining the last estimate with the uniform estimate $0\le w_n\le 1$ gives
\begin{align}
\label{eq:est-w_n-1}
|w_n|_{H^{1/2}(\Gamma)}^2 + \|w_n\|_{L^2(\Gamma)}^2
\le 
|\Gamma| + e^{-t/(2\tau_{\rm res})}|w_0|_{H^{1/2}(\Gamma)}^2
+
C \int_0^t e^{-(t-s)/(2\tau_{\rm res})} |v_n|_{H^{1/2}(\Gamma)}^2\, \dd s,
\end{align}
where $C$ is independent of $n$.
Using estimate \eqref{est-v_m} yields the estimate for $\sup_t \|w_n\|_{H^{1/2}(\Gamma)}^2$.

Finally, we prove the estimate for $\int_0^t \|\partial_t w_n(s, \cdot)\|_{H^{1/2}(\Gamma)}^2\, \dd s$ in \eqref{estimates-finite-dim}. For almost every
$t\in(0,T)$ and $x,y\in\Gamma$,
\[
|\partial_t w_n(t,x)-\partial_t w_n(t,y)|
\le \frac{C}{\tau_{\rm ep}}
\left(
|v_n(t,x)-v_n(t,y)|
+|w_n(t,x)-w_n(t,y)|
\right).
\]
Squaring, dividing by $|x-y|^3$, and integrating over
$\Gamma\times\Gamma$, we obtain
\[
|\partial_t w_n(t)|_{H^{1/2}(\Gamma)}^2
\le \frac{2C}{\tau_{\rm ep}^2}
\left(
|v_n(t)|_{H^{1/2}(\Gamma)}^2
+|w_n(t)|_{H^{1/2}(\Gamma)}^2
\right).
\]
Moreover, $0\le\beta(v_n),w_n\le1$ implies
$|\partial_t w_n|\le\tau_{\rm ep}^{-1}$. Therefore,
\begin{align*}
\|\partial_t w_n\|_{L^2(0,T;H^{1/2}(\Gamma))}^2
&\le C\,(T
+
\int_0^T
\left(
|v_n(t, \cdot)|_{H^{1/2}(\Gamma)}^2
+|w_n(t, \cdot)|_{H^{1/2}(\Gamma)}^2
\right)\, \dd t)\\
&\le C\left(
T|\Gamma|
+\|v_n\|_{L^2(0,T;H^{1/2}(\Gamma))}^2
+\|w_n\|_{L^2(0,T;H^{1/2}(\Gamma))}^2
\right).
\end{align*}
The previously established bounds for $v_n$ and $w_n$ now yield
\[
\|\partial_t w_n\|_{L^2(0,T;H^{1/2}(\Gamma))}^2
\le C_T\left(
1+\|v_0\|_{L^2(\Gamma)}^2
+\|w_0\|_{H^{1/2}(\Gamma)}^2
+\|p\|_{L^2(0,T;H^{-1/2}(\Gamma))}^2
\right),
\]
where $C_T$ is independent of $n$.
\end{proof}
Due to the uniform estimates \eqref{estimates-finite-dim}, there exists $v\in L^2(0,T; H^{1/2}(\Gamma)) \cap C([0,T]; L^2(\Gamma))$ with $\partial_t v \in L^2(0,T; H^{-1/2}(\Gamma))$ and $w\in H^1(0,T; H^{1/2}(\Gamma)) \cap W^{1, \infty}(0,T; L^2(\Gamma)) \subset C([0,T]; H^{1/2}(\Gamma))$ with $\partial_t w \in L^2(0,T; H^{1/2}(\Gamma))$ such that, up to a subsequence,
\begin{align*}
  &v_n \rightharpoonup v \,\, \mbox{in}\,\, L^2(0,T; H^{1/2}(\Gamma)),\\
  &v_n \overset{\ast}{\rightharpoonup} v \,\, \mbox{in}\,\, L^\infty(0,T; L^2(\Gamma)),\\
  &v_n \to v \,\, \mbox{strongly in}\,\, L^2(0,T; L^2(\Gamma)),\\
  & \partial_t v_n \rightharpoonup \partial_t v \,\, \mbox{in}\,\, L^2(0,T; H^{-1/2}(\Gamma)),\\
  &w_n  \rightharpoonup w \,\, \mbox{in}\,\, H^1(0,T; H^{1/2}(\Gamma)),\\
  &w_n  \overset{\ast}{\rightharpoonup} w \,\, \mbox{in}\,\, L^\infty(0,T; H^{1/2}(\Gamma)),\\
  &w_n \to w \,\, \mbox{strongly in}\,\, C([0,T]; L^2(\Gamma)),\\
  & \partial_t w_n \overset{\ast}{\rightharpoonup} \partial_t w \,\, \mbox{in}\,\, L^\infty(0,T; L^2(\Gamma)),\\
  & \mathcal{L} v_n \rightharpoonup \mathcal{L} v \,\, \mbox{in}\,\, L^2(0,T; H^{-1/2}(\Gamma)).
\end{align*}
Note that $v, w\in C([0,T]; L^2(\Gamma))$. Passing to the limit as $n\to \infty$, one proves that the limit functions $(v,w)$ solve
\begin{align*}
&\int_0^t\langle \partial_\tau v (\tau, \cdot), \varphi\rangle\, \dd \tau  + \int_0^t \big(\langle (\mathcal{L}+S_0{\rm Id}) v(\tau, \cdot), \varphi(\cdot)\rangle + S_1(v(\tau, \cdot)\, w(\tau, \cdot) \, ,\,  \varphi)_{L^2(\Gamma)}\big)\, \dd \tau = \int_0^t \langle p(\tau, \cdot), \varphi\rangle \, \dd \tau,\\
&\int_0^t(\partial_\tau w (\tau, \cdot), \psi)_{L^2(\Gamma)}\, \dd \tau=\int_0^t (f(v(\tau, \cdot), w(\tau, \cdot)), \psi)_{L^2(\Gamma)}\, \dd \tau,
\end{align*}
for any $(\varphi, \psi) \in H_m \times H_m$. Note that $w_n v_n \to w v$ strongly in $L^2(0,T; L^2(\Gamma))$. Since $v, w\in C([0,T]; L^2(\Gamma))$, $f(v,w)\in C([0,T]; L^2(\Gamma))$ and, consequently $w \in C^1([0,T]; L^2(\Gamma))$.

Multiplying the integral identity by smooth functions in $t$ and integrating, one can see that it is valid for all finite sums
\[
\varphi(t)=\sum_{k=1}^N\eta_k(t)e_k,
\qquad
\psi(t)=\sum_{k=1}^N\theta_k(t)e_k,
\qquad
\eta_k,\theta_k\in C_c^\infty(0,T).
\]
These classes are dense in $L^2(0,T;H^{1/2}(\Gamma))$ and $L^2(0,T;L^2(\Gamma))$,
respectively, and the
identities extend to arbitrary
$\varphi\in L^2(0,T;H^{1/2}(\Gamma))$ and $\psi\in L^2(0,T;L^2(\Gamma))$.

Taking test functions $\varphi \in C^1([0,T]; H^{1/2}(\Gamma)), \psi \in C^1([0,T]; L^2(\Gamma))$ satisfying $\varphi(T, \cdot)=0$, $\psi(T, \cdot)=0$, integrating by parts with respect to $t$, and passing to the limit in the weak formulation of \eqref{finite-dimensional}, we can identify the initial traces $v(0, \cdot)=v_0, w(0, \cdot)=w_0$. Thus, $(v,w)$ is a solution of \eqref{eq:abstract-prob}. {Note that under additional assumptions (A2'), (A3'), the uniform estimate \eqref{eq:est-partial_t v_m} is satisfied and thus $v\in L^\infty(0,T; H^{1/2}(\Gamma))$, $\partial_t v \in  L^2(0,T; L^2(\Gamma))$}.

To prove the uniqueness, we assume that there are two solutions $(v_1, w_1)$ and $(v_2, w_2)$ and use the test functions $v_1-v_2$, $w_1-w_2$ in the weak formulation of \eqref{eq:abstract-prob}.
\begin{align*}
&\frac{1}{2}\partial_t \|v_1-v_2\|_{L^2(\Gamma)}^2  = - \langle (\mathcal{L}+ S_0 {\rm Id}) (v_1-v_2), (v_1-v_2)\rangle - S_1(v_1 w_1 - v_2 w_2, (v_1-v_2))_{L^2(\Gamma)},\\
&\frac{1}{2}\partial_t \|w_1-w_2\|_{L^2(\Gamma)}^2=(f(v_1, w_1)-f(v_2, w_2), (w_1-w_2))_{L^2(\Gamma)}.
\end{align*}
{Since  $0\le w\le 1$,
\begin{align*}
&|(v_1 w_1 - v_2 w_2 \, ,\,  v_1-v_2)_{L^2(\Gamma)}|
\le |(w_1(v_1-v_2), (v_1-v_2))_{L^2(\Gamma)}|
+ |(v_2(w_1-w_2), (v_1-v_2))_{L^2(\Gamma)}|\\
&\le \|v_1-v_2\|_{L^2(\Gamma)}^2
+ \|w_1-w_2\|_{L^2(\Gamma)} \|v_2\|_{L^4(\Gamma)}\|v_1-v_2\|_{L^4(\Gamma)}\\
& \le \|v_1-v_2\|_{L^2(\Gamma)}^2
+ \frac{ \delta}{2} \|v_1-v_2\|_{L^4(\Gamma)}^2
+ \frac{1}{2\delta} \|w_1-w_2\|_{L^2(\Gamma)}^2\|v_2\|_{L^4(\Gamma)}^2\\
&\le \frac{ \delta}{2} \|v_1-v_2\|_{L^4(\Gamma)}^2
+ C_\delta(1+\|v_2\|_{L^4(\Gamma)}^2)\big(\|v_1-v_2\|_{L^2(\Gamma)}^2 + \|w_1-w_2\|_{L^2(\Gamma)}^2\big).
\end{align*}
{By the coercivity of $(\mathcal{L}+ S_0{\rm Id})$} and due to the continuous embedding of $H^{1/2}(\Gamma)$ into
$L^4(\Gamma)$, we can choose $\delta$ small enough such that 
\begin{align}
\label{eq:est_generalized_monotonicity_existence}
\begin{split}
    &-\langle (\mathcal{L}+S_0{\rm Id}) (v_1-v_2), (v_1-v_2)\rangle
- S_1(v_1 w_1 - v_2 w_2 \,,\,  (v_1-v_2))_{L^2(\Gamma)}\\
&\qquad \le 
C(1+\|v_2\|_{L^4(\Gamma)}^2)\big(\|v_1-v_2\|_{L^2(\Gamma)}^2 + \|w_1-w_2\|_{L^2(\Gamma)}^2\big).
\end{split}
\end{align}
Recalling the Lipschitz continuity of $f$, we obtain
\begin{align*}
&\|v_1-v_2\|_{L^2(\Gamma)}^2 + \|w_1-w_2\|_{L^2(\Gamma)}^2\\  
&\le \|v_1(0,\cdot)-v_2(0,\cdot)\|_{L^2(\Gamma)}^2
+ \|w_1(0,\cdot)-w_2(0,\cdot)\|_{L^2(\Gamma)}^2\\ 
&+  
C_1 \int_0^t (1+\|v_2\|_{H^{1/2}(\Gamma)}^2)\, (\|v_1-v_2\|_{L^2(\Gamma)}^2 + \|w_1-w_2\|_{L^2(\Gamma)}^2)\, d\tau.
\end{align*}
Since $v_2 \in L^2(0,T; H^{1/2}(\Gamma))$, the Grönwall inequality yields
\begin{align*}
&\|v_1-v_2\|_{L^2(\Gamma)}^2 + \|w_1-w_2\|_{L^2(\Gamma)}^2 \\ 
&\le \big(\|v_1(0,\cdot)-v_2(0,\cdot)\|_{L^2(\Gamma)}^2 
+ \|w_1(0,\cdot)-w_2(0,\cdot)\|_{L^2(\Gamma)}^2\big)\, \exp \big(
C_1 \int_0^t(1+\|v_2\|_{H^{1/2}(\Gamma)}^2)\, d\tau\big).
\end{align*}
Thus, the solution is unique.}
\end{proof}

\bibliographystyle{plain}
\bibliography{EP}

\end{document}